\documentclass[12pt,english,letterpaper]{article}
\usepackage[english]{babel}
\usepackage{amsmath,amssymb}
\usepackage{latexsym}
\usepackage{enumerate}
\usepackage{color}
\usepackage{xcolor}
\usepackage{graphicx}
\usepackage{amsmath}
\usepackage{amsfonts}
\usepackage{amscd}
\usepackage{amsmath}
\usepackage{amssymb}
\usepackage{amstext}

\newtheorem{theorem}{Theorem}[section]

\newtheorem{corollary}[theorem]{Corollary}

\newtheorem{definition}[theorem]{Definition}

\newtheorem{lemma}[theorem]{Lemma}

\newtheorem{proposition}[theorem]{Proposition}
\newtheorem{remark}[theorem]{Remark}

\newenvironment{proof}[1][Proof]{\noindent\textbf{#1.} }{\ \rule{0.5em}{0.5em}}
\allowdisplaybreaks

\begin{document}

\title{ Nonlinear Thermodynamic Formalism: Mean-field Phase Transitions,
Large Deviations and Bogoliubov's Variational Principle}
\author{Jean-Bernard Bru, Walter de Siqueira Pedra and Artur  O. Lopes}
\maketitle

\begin{abstract}
{\scriptsize {Let $\Omega =\{1,2,\ldots ,d\}^{\mathbb{N}}$, $T$ be the shift
acting on $\Omega $, $\mathcal{P}(T)$ the set of $T$-invariant
probabilities, and $h(\rho )$ the entropy of $\rho \in \mathcal{P}(T)$.
Given a H\"{o}lder potential $A:\Omega \rightarrow \mathbb{R}$ and a
continuous function $F:\mathbb{R}\rightarrow \mathbb{R}$, we investigate the
probabilities $\rho _{F,A}$ that are maximizers of the }nonlinear pressure
of $A$ and $F$ defined by}

{\scriptsize $\,\,\,\,\,\,\,\,\,\,\,\,\,\,\,\,\,\,\,\,\,\,\,\,\,\,\,\,\,\,\,%
\,\,\,\,\,\,\,\,\,\,\,\,\,\,\,\,\mathfrak{P}_{F,A}:=\sup_{\rho \in \mathcal{P%
}(T)}\left\{ F(\int A(x)\rho (\mathrm{d}x))+h(\rho )\right\} .$ }

{\scriptsize \smallskip \noindent {$\rho _{F,A}$} is called a \textit{%
nonlinear equilibrium}; a nonlinear phase transition occurs when there is
more than one. In the case $F$\ is convex or concave, we combine Varadhan's
lemma and Bogoliubov's variational principle to characterize them via the
linear pressure problem and self-consistency conditions. Let $\mu \in 
\mathcal{P}(T)$ be the maximal entropy measure, $\varphi
_{n}(x)=n^{-1}(\varphi (x)+\varphi (T(x))+\cdots +\varphi (T^{n-1}(x)))$ and 
$\beta >0$.}\newline
{\scriptsize (I) We also consider the limit measure $\mathfrak{m}$ on $%
\Omega $, so that $\forall \psi \in C(\Omega )$, }

{\scriptsize \smallskip $\,\,\,\,\,\,\,\,\,\,\,\,\,\,\,\,\,\,\,\,\,\,\,\,\,%
\,\,\,\,\,\,\,\,\,\,\,\,\,\,\,\,\,\,\,\,\,\,\int \psi (x)\,\mathfrak{m}\,(%
\mathrm{d}x)\,\,=\lim_{n\rightarrow \infty }\frac{\,\int \,\psi (x)\,\,\,e^{%
\frac{\beta n}{2}\,\,A_{n}((x)^{2}}\,\,\mu \,(\mathrm{d}x)\,}{\int e^{\frac{%
\beta n}{2}\,\,A_{n}((x)^{2}}\mu \,(\mathrm{d}x)\,\,}.$ \smallskip }

{\scriptsize \noindent We call $\mathfrak{m}$ a \textit{quadratic mean-field
Gibbs probability}; it may not be shift-invariant.}\newline
{\scriptsize (II) Via subsequences $n_{k}$, $k\in \mathbb{N}$, we study the
limit measure $\mathfrak{M}$ on $\Omega $, so that $\forall \psi \in
C(\Omega )$, \smallskip }

{\scriptsize $\,\,\,\,\,\,\,\,\,\,\,\,\,\,\,\,\,\,\,\,\,\,\,\,\,\,\,\,\,\,\,%
\,\,\,\,\,\,\,\,\,\,\,\,\,\,\,\,\int \psi (x)\mathfrak{M}(\mathrm{d}%
x)=\lim_{k\rightarrow \infty }\frac{\,\int \psi _{n_{k}}(x)e^{\frac{\beta
n_{k}}{2}A_{n_{k}}(x)^{2}}\mu (\mathrm{d}x)}{\int e^{\frac{\beta n_{k}}{2}%
A_{n_{k}}(x)^{2}}\mu (\mathrm{d}x)}.$ \smallskip }

{\scriptsize \noindent We call $\mathfrak{M}$ a \textit{quadratic mean-field
equilibrium probability}; it is shift-invariant.}\newline
{\scriptsize Both cases (I) and (II) can be related to self-consistency
conditions characterizing nonlinear equilibria {$\rho _{F,A}$ for $%
F(x)=\beta x^{2}/2$. In particular, }$\mathfrak{M}$ belongs to the closed
convex hull of nonlinear equilibria. Explicit examples are given. }
\end{abstract}



\section{Introduction}

Consider $\Omega =\{1,2,\ldots ,d\}^{\mathbb{N}}$ and the shift $T$ acting
on $\Omega $. Let $\mathcal{P}$ be the set of all Borel probabilities on $%
\Omega $ and $\mathcal{P}(T)\subseteq \mathcal{P}$, the set of $T$-invariant
probabilities. Given a H\"{o}lder potential $A:\Omega \rightarrow \mathbb{R}$
and a convex or concave function $F:\mathbb{R}\rightarrow \mathbb{R}$ (in
particular, it is continuous), our main aim is to investigate the set of $T$%
-invariant probabilities maximizing the so-called nonlinear pressure problem 
\begin{equation}
\sup_{\rho \in \mathcal{P}(T)}\left\{ F%
\Big (%
\int A(x)\rho (\mathrm{d}x)%
\Big )%
+h(\rho )\right\} ,  \label{Rur}
\end{equation}%
where $h(\rho )$ is the (Kolmogorov-Sinai) entropy of $\rho $.

We could also consider, with obvious adaptations, the multidimensional case
in which $F$ is a continuous function $\mathbb{R}^{k}\times \mathbb{R}%
^{l}\rightarrow \mathbb{R}$, $k,l\in \mathbb{N}_{0}$, $k+l\geq 1$, such
that, for all $(x,y)\in \mathbb{R}^{k}\times \mathbb{R}^{l}$, $F(\cdot ,y)$
is a convex function $\mathbb{R}^{k}\rightarrow \mathbb{R}$ and $F(x,\cdot )$
is a concave function $\mathbb{R}^{l}\rightarrow \mathbb{R}$. Our goal here
is not to achieve maximum generality -- that will be done elsewhere \cite%
{BPL} -- but to explore the main ideas in the simplest way possible and
discuss explicit examples. For this reason, we will limit ourselves here to
the case of one-dimensional nonlinearity, that is, we will only consider the
case $k+l=1$.

A probability maximizing \eqref{Rur} is called a nonlinear equilibrium
probability for the pair $F,A$. Given a potential $A$, if there exists more
than one probability maximizing the nonlinear pressure, we say that a
nonlinear phase transition takes place. When $F(x)=x$, we recover in %
\eqref{Rur} the standard (called here linear) case. If $F(x)=\pm\, x^{2}/2$,
we speak about the quadratic case. An important issue is establishing a
relationship between the nonlinear pressure problem and the standard
(linear) pressure problem for another effective potential that is connected
to $A$.

Related to the above problem, in a series of papers \cite%
{LeWa1,LeWa2,BKL,Barre,Barre1,Ku,Ding,Zhu}, a rigorous approach to analyze
questions in mean-field theory from the ergodic point of view is introduced,
in particular for the Curie-Weiss-Potts models. These results constitute the
foundations of a new area called \textit{nonlinear thermodynamical formalism}%
. In the above references, the nonlinear equilibrium probabilities are
standard (linear) equilibrium probabilities for linear combinations of
potentials that appear in the nonlinear term of the pressure. Here, using an
appropriate version of the so-called Bogoliubov's approximation, we are able
to describe such linear combinations exactly, in contrast to the previous
works. This allows us, in particular, to detect phase transitions by showing
the non-uniqueness of the linear combinations of potentials. As far as we
know, this approach is new in the context of the nonlinear thermodynamical
formalism.

Note that Bogoliubov's approximation was originally invented in 1947 to
obtain a microscopic theory of helium superfluidity \cite{Bogoliubov1}. This
is connected to the approximating Hamiltonian method used to study
mean-field theories, as defined by Bogoliubov Jr., Brankov, Kurbatov,
Tonchev and Zagrebnov in the seventies and eighties \cite%
{Bogjunior,AHM-non-poly1,AHM-non-poly2,approx-hamil-method0,approx-hamil-method,approx-hamil-method2}%
. In the case of quantum lattices, an extensive development of this method
appeared in the 2013 monograph \cite{BP2}.

In Section \ref{lar} we recall some well-known results from the
thermodynamical formalism of symbolic dynamical systems, as well as results
from large deviations theory, and discuss how they provide a natural
framework for the variational problem of the nonlinear pressure for H\"{o}%
lder potentials $A:\Omega \rightarrow \mathbb{R}$. In this framework,
Section \ref{Bogo} is dedicated to applications of Bogoliubov's variational
principle in the scope of the thermodynamical formalism of symbolic
dynamical systems, for the case where $F$ is a convex or concave function.
The concave case is the more involved one, and a separate subsection
(Subsection \ref{conc}) is devoted to this case. The relationship between
the nonlinear pressure problem and the standard one is one of the central
questions of Section \ref{lar}: in Section \ref{Bogo} we show in this
context how the so-called self-consistency condition plays an important role
and naturally emerges from Bogoliubov's variational principle. In Section %
\ref{Bog} we introduce the concept of mean-field free energy (only in the
quadratic case, for simplicity), which provides an alternative and useful
way to get nonlinear equilibrium probabilities. In Section \ref{qua} we
consider the quadratic case, that is, $F(x)=\pm \beta x^{2}/2$, where the
parameter $\beta >0$ refers to the inverse temperature in statistical
physics. In Section \ref{mae} we consider a special choice of potential $A$,
for which explicit expressions for the quadratic pressure problem can be
obtained; we take advantage of the results obtained in Section \ref{qua}. We
then present examples of quadratic (nonlinear) phase transitions. In Section %
\ref{conj} we analyze the quadratic mean-field Gibbs probabilities and give
explicit examples showing the existence of phase transitions in this
setting. Notice that the authors of \cite{LeWa1} discussed this case when $%
\mu $ is the probability of maximal entropy, and we adapt their proof with $%
\mu $\ being replaced with the equilibrium probability $\mu _{f}$ of an
arbitrary H\"{o}lder continuous potential $f:\Omega \rightarrow \mathbb{R}$.
In this case, we get different forms for the self-consistency conditions
related to the quadratic equilibrium probabilities. In Section \ref{QLVE} we
study quadratic mean-field equilibrium probabilities (see \eqref{q1a1} and
Definition \ref{dedeus}). Finally, in Section \ref{Til} we make an
observation concerning the tilting property of the large deviation theory in
the thermodynamical formalism and Bogoliubov's variational problem.

\section{Large deviations and nonlinear equilibrium probabilities}

\label{lar}

\subsection{Equilibrium probabilities}

$\Omega $ is the set of infinite strings on a finite alphabet $\{1,2,\ldots
,d\}$ ($d\in \mathbb{N}$), that is, $\Omega :=\{1,2,\ldots ,d\}^{\mathbb{N}}$%
. Denote by $T$ the shift $T:\Omega \rightarrow \Omega $, defined by 
\begin{equation*}
T(x_{1},x_{2},\ldots ):=(x_{2},x_{3},\ldots )\text{{}}
\end{equation*}%
for all $x=(x_{1},x_{2},\ldots )$. A case of particular interest is $d=2$,
which, for convenience, is identified with $\Omega =\{-1,1\}^{\mathbb{N}}$.
We consider on $\Omega $ the metric 
\begin{equation}
d(x,y):=\left( \frac{1}{2}\right) ^{\min \{n\text{ }:\text{ }x_{n}\neq
y_{n}\}},  \label{tyr}
\end{equation}%
with $x=(x_{1},x_{2},\ldots )$, $y=(y_{1},y_{2},\ldots )$. Observe that $%
(\Omega ,d)$ is a compact metric space.

Let $\mathcal{P}(T)$ be the (compact, convex) space of $T$-invariant
probabilities, always endowed with the weak$^{\ast }$ topology. Given a H%
\"{o}lder continuous potential $A:\Omega \rightarrow \mathbb{R}$, define 
\begin{equation}
P(A):=\sup_{\rho \in \mathcal{P}(T)}\left\{ \rho (A)+h(\rho )\right\} ,
\label{iore}
\end{equation}%
where $h(\rho )$ is the Kolmogorov-Sinai entropy of $\rho $ (see Chapter 4
of \cite{Walters} or Chapter 3 of \cite{PP}) and, as usual, 
\begin{equation*}
\rho (A)=\int A(x)\rho (\mathrm{d}x).
\end{equation*}

We call $P(A)$ the \textit{linear (or standard) pressure} of $A$. Regarding
physics, a given potential $A$ as above corresponds to the Hamiltonian $H=-A$
in statistical mechanics. For the corresponding problem in $C^{\ast }$%
-algebras, i.e., for the study of the quantum version of equilibrium
probabilities, we refer to \cite{BP1} and \cite{BP2}.

Note that, clearly, 
\begin{equation}
P(A-\log d)=\sup_{\rho \in \mathcal{P}(T)}\left\{ h(\rho )+\rho (A)\right\}
-\log d.\text{{}}  \label{iore43}
\end{equation}%
This remark is, of course, trivial and serves only to emphasize the
important role played by the factor $\log d$. It is nothing but the maximum
entropy. A few other remarks in this sense are given below.

Note that the Kolmogorov-Sinai entropy $h(\rho )$ is (weak$^{\ast }$)
upper-continuous and affine on the compact convex space $\mathcal{P}(T)$.
See chapter 6 and Theorems 8.1 and 8.2 in \cite{Walters}. In particular, the
variational problem (\ref{iore}) has a nonempty compact face of maximizers
which are nothing but equilibrium probabilities:

\begin{definition}
\label{equilibrium probability linear}The probabilities $\mu _{A}\in 
\mathcal{P}(T)$ maximizing the right-hand side of (\ref{iore}) are called
the linear (or standard) equilibrium probabilities for $A$.
\end{definition}

\noindent If $A$ is of H\"{o}lder class, then the equilibrium probability is
unique.

Linear pressures and equilibrium probabilities can be studied via the
so-called Ruelle operator. The Ruelle operator $\mathcal{L}_{A}$ for a
continuous potential $A:\Omega \rightarrow \mathbb{R}$ acts on functions $%
\psi :\Omega \rightarrow \mathbb{R}$ in the following way: For each $x\in
\Omega :=\{1,2,\ldots ,d\}^{\mathbb{N}}$, 
\begin{equation}
\mathcal{L}_{A}(\psi )(x)=\sum_{a=1}^{d}e^{A(ax)}\psi
(ax)=\sum_{\{y\,|\,T(y)=x\}}e^{A(y)}\psi (y),  \label{popoiu1}
\end{equation}%
where $ax:=(a,x_{1},x_{2},\ldots )$ for any $x=(x_{n})_{n\in \mathbb{N}}\in
\Omega $.

Given a continuous potential $A:\Omega \rightarrow \mathbb{R}$, we define
the dual operator $\mathcal{L}_{A}^{\ast }$ on the space of the Borel finite
measures on $\Omega $ as the operator that maps the finite measure $%
\mathfrak{v}$ to the finite measure $\mathfrak{u}=\mathcal{L}_{A}^{\ast }(%
\mathfrak{v})$ defined by 
\begin{equation}
\mathfrak{u}\left( \psi \right) =\int \psi \,d\mathfrak{u}=\int \psi (x)\,%
\mathcal{L}_{A}^{\ast }(\mathfrak{v})(\mathrm{d}x)=\int \mathcal{L}_{A}(\psi
)(x)\,\mathfrak{v}(\mathrm{d}x)  \label{feij2}
\end{equation}%
for any $\psi \in C(\Omega )$.

With these definitions, the Ruelle(-Perron-Frobenius) theorem connects
linear pressures and equilibrium probabilities with properties of the Ruelle
operator:

\begin{theorem}
\label{rer1}If $A$ is of H\"{o}lder class, there exists a strictly positive H%
\"{o}lder eigenfunction $\psi _{A}$ for $\mathcal{L}_{A}:C(\Omega
)\rightarrow C(\Omega )$, associated to a strictly positive eigenvalue $%
\lambda _{A}$ which is simple\footnote{%
It is also isolated from the rest of the spectrum when $\mathcal{L}_{A}$ is
restricted to the set of H\"{o}lder functions.}, equals the spectral radius
of $\mathcal{L}_{A}$ and satisfies $\log \lambda _{A}=P(A)$. Moreover, there
exists an eigenprobability $\nu _{A}$ such that $\mathcal{L}_{A}^{\ast }(\nu
_{A})=\lambda _{A}\nu _{A}$, and the unique (linear) equilibrium probability 
$\mu _{A}$ for $A$ is the measure $\psi _{A}\nu _{A}$ properly normalized.
\end{theorem}

\begin{proof}
See \cite{PP}, in particular, Theorem 2.2 and Theorem 3.5. E.g., for the
equality $\log \lambda _{A}=P(A)$, see Theorem 3.5 in \cite{PP} or Section 3
in \cite{FaJi}, while $\mu _{A}=C\psi _{A}\nu _{A}$ is a consequence of
Theorem 2.2 and 3.5 in \cite{PP}.
\end{proof}

Here we call $\nu _{A}$ the \textit{linear Gibbs probability} for $A$, in
order to highlight the distinction with the concept of linear equilibrium
probability for $A$. When setting definitions for the nonlinear case
(extending the linear one) we will be consistent with this terminology (as
for instance in Definitions \ref{loloy} and \ref{dedeus}).

We say that the potential $A$ is normalized if $\mathcal{L}_{A}(1)=1$. The
potential $A=-\log d$ is an example of a normalized potential and, in this
case, the equilibrium probability $\mu _{-\log d}$ is nothing but $\mu $,
the maximum entropy probability. Remark, moreover, that, in this special
case, one has $\log \lambda _{A}=\log 1=P(A)=0$. Note that $h(\mu )=-\log d$%
. In all the paper, $\mu $ denotes the probability of maximal entropy for $T$%
.

In the next sections, we will be interested in the following nonlinear
problem: Given a continuous function $F:\mathbb{R}\rightarrow \mathbb{R}$,
determine the $T$-invariant probabilities that are maximizers for the
nonlinear pressure for $A$ and $F$: 
\begin{equation}
\mathfrak{P}_{F}=\mathfrak{P}_{F,A}:=\sup_{\rho \in \mathcal{P}(T)}\left\{
F(\rho (A))+h(\rho )\right\} .  \label{Kbom}
\end{equation}%
Compare with the linear case given by Equation (\ref{iore}). Similar to
Definition \ref{equilibrium probability linear}, we extend the definition of
equilibrium probabilities to the nonlinear situation:

\begin{definition}
\label{nonlinear equilibrium probabilities}The probabilities $\rho _{A}=\rho
_{F,A}\in \mathcal{P}(T)$ maximizing the right-hand side of (\ref{Kbom}) are
called \textit{nonlinear equilibrium probabilities} for $A$ and $F$.
\end{definition}

Note that, unlike the linear case, a convex combination of nonlinear
equilibrium probabilities for $A$ and $F$ do not have to be a nonlinear
equilibrium probability for $A$. In particular, the set of maximizers of the
variational problem (\ref{Kbom}) is not necessarily a face in $\Omega $ as
in the linear situation. Furthermore, the set of nonlinear equilibrium
probabilities can have many elements, unlike the linear case for H\"{o}lder
functions $A$. This yields to phase transitions:

\begin{definition}
\label{phase transition}Given $A$ and $F$, we say that a phase transition
occurs for the nonlinear pressure problem if there is more than one $T$%
-invariant probability maximizing \eqref{Kbom}, that is, when the nonlinear
equilibrium probability is not unique.
\end{definition}

\noindent In Section \ref{Examples}, we give an example of a phase
transition occurring for the nonlinear pressure problem, while in Example %
\ref{boex}, we present a case where there is no phase transition.

It is clear that the above problem is equivalent to asking for the $T$%
-invariant probabilities that realize the following supremum:%
\begin{equation}
\sup_{\rho \in \mathcal{P}(T)}\left\{ F(\rho (A))+h(\rho )-\log d\right\} .
\label{Kbomm}
\end{equation}%
For $F(x)=x$ we just get $\mathfrak{P}_{F,A}=P(A)-\log d$ or, equivalently, $%
\mathfrak{P}_{F-\log d,A}=P(A)$. A case of particular interest is $%
F(x)=x^{2}/2$. In this situation, we write 
\begin{equation}
\mathfrak{P}_{2}(A):=\sup_{\rho \in \mathcal{P}(T)}\left\{ \frac{\rho (A)^{2}%
}{2}+h(\rho )-\log d\right\} ,  \label{Kibom}
\end{equation}%
where $h(\rho )$ is the (Kolmogorov-Sinai) entropy of $\rho $. We call $%
\mathfrak{P}_{2}(A)$ the \textit{quadratic pressure for }$A$. It is related
to the so-called Curie-Weiss model (see \cite{LeWa1}). The maximizing $T$%
-invariant probabilities will be called \textit{quadratic equilibrium
probabilities} for $A$. Consequently, we say that there exists a \textit{%
quadratic phase transition} when there is more than one $T$-invariant
probability maximizing \eqref{Kibom}.

In Section \ref{mae}, we will present examples of H\"{o}lder potentials $A$
and quadratic functions $F$, for which explicit expressions can be obtained
for the probabilities that maximize \eqref{Kbom}. This amounts to solving
Equation \eqref{utsx} given below. In these examples, $d=2$. More precisely,
we will consider in Section \ref{mae} potentials $A:\{-1,1\}^{\mathbb{N}%
}\rightarrow \mathbb{R}$ of the form%
\begin{equation}
A(x)=A(x_{1},x_{2},\ldots ,x_{n},\ldots )=\sum_{n=1}^{\infty }a_{n}\,x_{n},
\label{bin1}
\end{equation}%
where $a_{n}$ is a sequence of real numbers converging exponentially to
zero. We refer to \cite{CDLS} or Example 13 in Section 3.2 of \cite{LTF} for
an extensive study of properties of (linear) equilibrium probabilities for
this kind of potential.

In Section \ref{mae} we are particularly interested in the case that, for a
given H\"{o}lder potential $A$, the quadratic equilibrium probability is not
unique, i.e, there is a quadratic phase transition (as described in Remark %
\ref{yyx1}). The symmetry $P(tA)=P(-tA)$, $t\in \mathbb{R}$, can be used to
produce examples of such phase transitions (see Remark \ref{sisim}). In
fact, we will provide an explicit example of a quadratic phase transition by
making use of that precise symmetry.

For the quadratic case, we will also address issues related to Section 2.1
of \cite{LeWa1} in our Section \ref{conj}, where we consider \textit{%
quadratic mean-field Gibbs phase transitions} for $d=2$, a different notion
of phase transition, as compared with the previous concept of quadratic
(equilibrium) phase transition. This is related to the structure of
probabilities, which are called here quadratic mean-field Gibbs
probabilities. See below Section \ref{Quadratic mean-field probabilities},
in particular Equation \eqref{xiux23}, and Section \ref{laplace}, in
particular Equation \eqref{xiu}.

The quadratic case is also useful to illustrate the self-consistency
conditions, which are pivotal to describe nonlinear equilibrium
probabilities from the linear thermodynamic formalism. First, given a H\"{o}%
lder potential $A:\Omega \rightarrow \mathbb{R}$, one can show (see for
instance Theorem 3 in \cite{L3}, or Proposition 3.2 in \cite{Kif}) that, for
all $t\in \mathbb{R}$, 
\begin{equation}
c(t)=c_{A,\mu }(t):=\lim_{n\rightarrow \infty }\hat{c}_{n}(t)+\log d=P(tA),
\label{feij502}
\end{equation}%
where $P(t\,A)$ is the pressure of the potential $tA$ and, for each $n\in 
\mathbb{N}$, 
\begin{equation}
\hat{c}_{n}(t):={\frac{1}{n}}\log \int e^{t(A(x)+A(T(x))+A(T^{2}(x))+\cdots
+A(T^{n-1}(x))}\mu (\mathrm{d}x),  \label{feij50}
\end{equation}%
with $\mu $ being the maximal entropy probability. Note that $c(0)=\log d$.
It is also useful to consider the similar function 
\begin{equation}
\hat{c}(t)=\hat{c}_{A,\mu }(t):=P(tA)-\log d=\lim_{n\rightarrow \infty }\hat{%
c}_{n}(t).  \label{feij502bis0}
\end{equation}%
In Ergodic Theory, the quantity 
\begin{equation}
\hat{c}_{A,\mu }(t)=\hat{c}(t)=c(t)-\log d  \label{feij502bis}
\end{equation}%
is sometimes called the free energy for the pair $A,\mu $ at time $t$.

These functions are directly related to the existence of some Large
Deviation Principle (LDP) via the Varadhan(-Bryc) lemma, as explained below
in Section \ref{LD}. See in particular Equations (\ref{feij502})--(\ref%
{feij50}), which show that $c$ is nothing but some logarithmic moment
generating function. Interestingly, and perhaps surprisingly for
non-experts, the same functions define the self-consistency conditions
derived from Bogoliubov's variational problem, which allow us to obtain all
the nonlinear equilibrium probabilities. To our knowledge, such a link
between large deviations and Bogoliubov's approach is only known in the
quantum case, at least for the weakly imperfect superstable Bose gas \cite%
{LDP-BZ}.

Indeed, observe that the function $t\mapsto P(tA)$ is strictly convex,
unless $A$ is coboundary to a constant\footnote{%
The H\"{o}lder potential $A$ being coboundary to a constant means that it is
of the form $A=\alpha +B\circ T-B$ for some constant $\alpha \in \mathbb{R}$
and H\"{o}lder potential $B$.} (a particular case that we will avoid).
Moreover, $t\mapsto P(tA)$ is analytic if $A$ is H\"{o}lder (see Proposition
4.7 in \cite{PP}, or Theorem 8.2 in \cite{FaJi}). In this case, one can show
that 
\begin{equation}
c^{\prime }(t)=\hat{c}^{\prime }(t)=\mu _{tA}(A)=\int A(x)\mu _{tA}(\mathrm{d%
}x),  \label{outr}
\end{equation}%
where $\mu _{tA}$ is the unique linear equilibrium probability for the
potential $tA$ (see Proposition 4.10 in \cite{PP}). Moreover, 
\begin{equation}
\lim_{t\rightarrow \infty }c_{A,\mu }(t)=\infty =\lim_{t\rightarrow -\infty
}c_{A,\mu }(t).  \label{tryu17}
\end{equation}%
One of our main results in Section \ref{Bogo} is related to the so-called
self-consistency Equation \eqref{geger}, the quadratic case of which is
Equation \eqref{outr71}. This refers to the following statement for the
quadratic example:

\begin{theorem}
The equation in $t$ 
\begin{equation}
\hat{c}^{\prime }(t)=\mu _{tA}(A)=\int A(x)\mu _{tA}(\mathrm{d}x)=t
\label{utsx}
\end{equation}%
determines the possible values $t$ for which the linear equilibrium
probability for the potential $tA$ maximizes the quadratic pressure $%
\mathfrak{P}_{2}(A)$ for the potential $A$ (see (\ref{Kibom})).
\end{theorem}

Depending on the potential, there may be more than one solution to Equation %
\eqref{utsx} and a quadratic phase transition can occur. In Remark \ref{boex}
of Section \ref{mae}, for a certain choice of potential $A$ (and parameter $%
\beta >0$ that we introduce later on), we can determine the exact point $t$
at which the self-consistency condition holds true.

In Section \ref{The convex case}, we analyze the more general case where $F$
is an arbitrary convex function (i.e., $F$ is not necessarily quadratic) via
Bogoliubov's variational problem. Similar to the quadratic case, if $F$ is
convex, we present the associated self-consistency equation (see %
\eqref{geger}, a generalization of \eqref{utsx}), which determines the
equilibrium probabilities for the nonlinear pressure $\mathfrak{P}_{F,A}$,
defined by (\ref{Kbom}). Later on, in Section \ref{conc}, we will also
examine the more complex case where $F$ is a concave function.

Nevertheless, the combination of convex and concave functions in the
nonlinear variational problem (\ref{Kbom}) is not addressed here, as it
involves certain subtleties that would complicate our discussion and thus
make it much more obscure. This situation is, however, treated in a very
general way in our second article \cite{BPL}, for alphabets that are
potentially uncountable (unlike here) but still compact.

\subsection{Large deviations in the thermodynamical formalism\label{LD}}

In probability theory, the law of large numbers states that, as $%
n\rightarrow \infty $, the empirical mean of $n$ independent and identically
distributed random variables converges in probability to their expected
value, provided it exists. The central limit theorem refines this result by
describing the fluctuations of the empirical mean: when rescaled by $\sqrt{n}
$, these deviations from the expected value converge in distribution to a
normal law, assuming the variance is finite. Then, the large-deviation
theory \cite{DZ,DS89} addresses the probability of rare events in which the
empirical mean deviates from the expected value. For large $n\gg 1$, such
probabilities decay exponentially fast as $n\rightarrow \infty $ under a
so-called \emph{large deviation principle} (LDP). As a general reference for
the large deviation theory in the ergodic theory setting, we recommend \cite%
{Orey}.

We use this formalism below in the context of the (nonlinear) the
thermodynamical formalism of symbolic dynamical systems studied here. Large
deviations will be connected to nonlinear equilibrium probabilities, as
defined in the previous subsection, through the functions $c$ and $\hat{c}$
defined by (\ref{feij502})-(\ref{feij502bis}).

Bearing in mind the Varadhan(-Bryc) lemma (or Bryc's inverse varadhan lemma)
and the fact that $c$ is nothing but some logarithmic moment generating
function (see (\ref{feij502})--(\ref{feij50})), we define a (good) rate
function by applying the Legendre transform on $c$. To this end, we first
define the so-called ergodic maximal value of $A$ to be 
\begin{equation*}
em(A):=\sup_{\rho \in \mathcal{P}(T)}\rho (A)<\infty .
\end{equation*}%
Suppose that $em(A)\geq 0$. Then, one can show (see \cite{BLL}, \cite{G1},
or Section 6 in \cite{LTF}) that 
\begin{equation}
\lim_{t\rightarrow \infty }\frac{c_{A,\mu }(t)}{t}=\lim_{t\rightarrow \infty
}\frac{c(t)}{t}=em(A).  \label{tryu1}
\end{equation}%
Here, recall that $\mu $ is the maximal entropy probability and the
potential $A$ is also fixed, while, for simplicity, $c_{A,\mu }$ and $\hat{c}%
_{A,\mu }$ are often denoted by $c$ and $\hat{c}$, respectively. See again
Equations (\ref{feij502})--(\ref{feij502bis}). Further, let the real numbers 
$m_{A}$ and $M_{A}$\ be defined by the finite interval 
\begin{equation*}
\{c_{A,\mu }^{\prime }(t)\text{ }|\text{ }t\text{ }\in \mathbb{R}%
\}=(m_{A},M_{A})\subseteq \mathbb{R},\text{ \ \ }m_{A}<M_{A}.
\end{equation*}

Define now the following rate function $I=I_{A}$, as being the Legendre
transform of the function $\hat{c}$, that is, for any $x\in \mathbb{R}$, 
\begin{equation}
I(x)=I_{A}(x)=\sup_{t\in \mathbb{R}}\{tx-\hat{c}(t)\}=\sup_{t\in \mathbb{R}%
}\{tx-P(tA)+\log d\}.  \label{feij68}
\end{equation}%
See Equation (\ref{feij502bis0}). $I=I_{A}$ is called the (large deviation)
rate function for the pair $A,\mu $. See, for instance, \cite{Mord} or \cite%
{Orey}. $I$ is convex and analytic because the mapping $t\mapsto \hat{c}%
_{A,\mu }(t)$ is convex and analytic\footnote{%
Recall that $t\mapsto P(tA)$ is analytic if $A$ is H\"{o}lder. See \cite{PP}%
, in particular Proposition 4.7.}. Moreover, by Equation (\ref{outr}), 
\begin{equation}
c^{\prime }(0)=\mu (A)=\int A(x)\mu (\mathrm{d}x)\qquad \text{and}\qquad
I(\mu (A))=I_{A}(\mu (A))=0,  \label{sdsdf}
\end{equation}%
recalling once again that $\mu $ is the probability of maximal entropy.

The function $I=I_{A}$ is well defined in the finite interval $(m_{A},M_{A})$
(i.e., it takes finite values) and we set $I(x)=I_{A}(x):=\infty $ if the
corresponding supremum does not exist. $I_{A}(x)$ tends to $\infty $ when $x$
approaches the boundary of the interval $(m_{A},M_{A})$. In particular,the
domain of the rate function $I=I_{A}$ equals 
\begin{equation}
\mathrm{dom}(I_{A}):=\{x\in \mathbb{R}:I_{A}(x)<\infty \}=(m_{A},M_{A}).
\label{domain I}
\end{equation}%
Moreover, $I$ vanishes only at the point $\mu (A)\in (m_{A},M_{A})$, see (%
\ref{sdsdf}). Note that the function $I$ takes non-negative values since a
simple computation using (\ref{iore}) and (\ref{outr}) shows that 
\begin{equation}
I(x)=I_{A}(x)=\log d-h(\mu _{tA})\geq 0,  \label{myt68}
\end{equation}%
where $\mu _{tA}$ is now the equilibrium probability for the potential $tA$
and 
\begin{equation}
x=c^{\prime }(t)=\frac{dP(tA)}{dt}=\mu _{tA}(A)=\int A(x)\mu _{tA}(\mathrm{d}%
x).  \label{myt68bis}
\end{equation}%
(See Section \ref{qua}.) From the last observations, remark finally that $%
I=I_{A}$ is in particular not the $\infty $--constant function and has
compact level sets, i.e., $I^{-1}([0,m])=\{x\in \mathbb{R}:I(x)\leq m\}$ is
compact for any $m\geq 0$. Such a rate function is said to be \emph{good} in
the large deviation theory. For a more detailed discussion of all claims of
this paragraph, see, for instance, Section 8 of \cite{LTF} or \cite{L4}.

For each $n\in \mathbb{N}$ denote by $\mu _{n}=\mu _{n}^{A}$ the probability
measure such that, for any open interval $O\subseteq \mathbb{R}$, 
\begin{equation}
\mu _{n}(O)=\mu _{n}^{A}(O)=\mu \left( \left\{ z\mid A_{n}(z)\in O\right\}
\right) ,  \label{binc78}
\end{equation}%
where, for any $\varphi \in C(\Omega )$, the continuous functions $\varphi
_{n}$, $n\in \mathbb{N}$, are the so-called Birkhoff averages%
\begin{equation}
\varphi _{n}:=\frac{1}{n}(\varphi +\varphi \circ T+\cdots +\varphi \circ
T^{n-1}),\text{ \ \ }n\in \mathbb{N}.  \label{Eq Birk Av}
\end{equation}%
Clearly, for each $n\in \mathbb{N}$, the support of $\mu _{n}$ is inside the
interval $[-\Vert A\Vert _{\infty },\Vert A\Vert _{\infty }]$, where $\Vert
A\Vert _{\infty }$ is the supremum norm of $A$. In addition, for each $n\in 
\mathbb{N}$ and bounded Borel function $V:\mathbb{R}\rightarrow \mathbb{R}$,
one has 
\begin{equation}
\int V(z)\mu _{n}(\mathrm{d}z)=\int V\left( A_{n}(x)\right) \mu (\mathrm{d}%
x).  \label{brinc78}
\end{equation}%
One can show (see \cite{L3}, \cite{L4}, \cite{Orey} or \cite{Kif}) in our
case that, for any (open or closed) interval $B\subseteq \mathbb{R}$, 
\begin{equation}
\lim_{n\rightarrow \infty }{\frac{1}{n}}\log \mu _{n}\left( B\right)
=-\inf_{x\in B}\{I(x)\},\text{{}}  \label{aret23x}
\end{equation}%
where $I=I_{A}$ is the rate function defined above by (\ref{feij68}) for the
H\"{o}lder potential $A$ and the maximum entropy probability $\mu $.

\begin{remark}
By (\ref{feij502bis0}), (\ref{feij68}) and (\ref{myt68}), one has that 
\begin{equation}
xt=I(x)+\hat{c}(t)=I(x)+P(tA)-\log d\Longleftrightarrow t=I^{\prime }(x).
\label{feij99}
\end{equation}%
Equivalently, 
\begin{equation}
xt=I(x)-\log d+P(tA)=-h(\mu _{tA})+P(tA)\Longleftrightarrow t=I^{\prime }(x),
\label{feij991}
\end{equation}%
where $\mu _{tA}$ is the equilibrium probability for the potential $tA$. In
fact, $I(x)=\log d-h(\mu _{tA})\geq 0$ when $t=I^{\prime }(x).$
\end{remark}

Given a continuous and bounded function $F:\mathbb{R}\rightarrow \mathbb{R}$
and a continuous potential $A:\Omega \rightarrow \mathbb{R}$, Equation (\ref%
{aret23x}) indicates that both large-deviation upper and lower bounds are
satisfied, i.e., the sequence $\{\mu _{n}=\mu _{n}^{A}\}_{n\in \mathbb{N}}$
of probabilities satisfies a so-called \emph{Large Deviation Principle}
(LDP) with good rate function $I=I_{A}$ (see \eqref{feij68}) and speed $%
(n)_{n\in \mathbb{N}}$. By Theorem 1 of \cite{Kos}, it follows that 
\begin{equation}
\lim_{n\rightarrow \infty }\frac{1}{n}\log \int e^{nF\left( x\right) }\mu
_{n}(\mathrm{d}x)=\sup_{x\in \mathbb{R}}\{F(x)-I(x)\}=:\hat{c}(F)
\label{impi}
\end{equation}%
and 
\begin{equation}
I(x)=\sup_{F\in C(\mathbb{R})}\{F(x)-\hat{c}(F)\}.  \label{zeno1}
\end{equation}%
In particular, combined with (\ref{brinc78}), one obtains that, for any $%
t\in \mathbb{R}$,%
\begin{equation}
\lim_{n\rightarrow \infty }\frac{1}{n}\log \int e^{ntF\left( A_{n}(x)\right)
}\mu (\mathrm{d}x)=\sup_{x\in \mathbb{R}}\{tF(x)-I(x)\}=\hat{c}(tF).
\label{impio}
\end{equation}%
This result is in fact a direct application of the Varadhan(-Bryc) lemma,
which is a standard cornerstone of the large deviation theory and serves as
a starting point for large-deviation studies. It is a powerful tool with
wide-ranging applications. See, e.g., Theorem 2.1.10 in \cite{DS89} or
Theorem 4.3.1 in \cite{DZ}.

In the next section, we combine the Varadhan-Bryc lemma given above with the
variational principle (\ref{iore}) for the linear pressure and Bogoliubov's
variational principle to prove that $\hat{c}(F)$ (see (\ref{impi})) is
nothing else but the nonlinear pressure $\mathfrak{P}_{F}$, when $F$ is
convex or concave. Additionally, the equality $\hat{c}(F)=\mathfrak{P}_{F}$
will allow us, again via Bogoliubov's variational principle, to show that
nonlinear equilibrium probabilities are necessarily linear equilibrium
probabilities of (self-consistent) effective potentials. As already
mentioned, among other things, Bogoliubov's variational principle determines
these potentials and, hence, allow us to detect nonlinear phase transitions.

\begin{remark}
\label{mmx}Within the level-2 large deviations, considering empirical
probabilities $\sum_{j=0}^{n-1}\delta _{T^{j}(z)}$ for the maximal entropy
probability $\mu $, the associated large deviation rate function is $I(\rho
)=-s(\rho )$ for any probability $\rho $, where $s(\rho ):=h(\rho )-\log d$
for all $T$-invariant probabilities $\rho $ and $s(\rho ):=-\infty $ in
other cases. See, for instance, \cite{L3}.
\end{remark}

Similar large-deviation results to those already achieved with the maximal
entropy probability $\mu $ can also be obtained with any linear equilibrium
probability $\mu _{f}$ associated with a general H\"{o}lder potential $%
f:\Omega \rightarrow \mathbb{R}$. See Definition \ref{equilibrium
probability linear}. To demonstrate this, one should use the next result
(see \cite{Kif}, \cite{L3} or \cite{L4}):

\begin{proposition}
Let $\mu _{f}$ be the linear equilibrium probability for a H\"{o}lder
potential $f:\Omega \rightarrow \mathbb{R}$. Given another H\"{o}lder
function $A:\Omega \rightarrow \mathbb{R}$ and $t\in \mathbb{R}$, we have 
\begin{equation}
\hat{c}_{f,A}(t):=\lim_{n\rightarrow \infty }{\frac{1}{n}}\log \int
e^{tnA_{n}(x)}\mu _{f}(\mathrm{d}x)=P(f+tA)-P(f),  \label{feij67}
\end{equation}%
where $A_{n}$, $n\in \mathbb{N}$, are the Birkhoff averages defined by (\ref%
{Eq Birk Av}).
\end{proposition}

\noindent Compare this assertion with (\ref{feij502}).

From now on we will assume that $f$ is normalized, that is, $\mathcal{L}%
_{f}(1)=1$. Proceeding in exactly the same way as with the maximal entropy
probability $\mu $ (cf. (\ref{binc78})),\ for each $n\in \mathbb{N}$, let $%
\mu _{n}^{f,A}$ be the probability measure on $\mathbb{R}$ such that, for
any open interval $O\subseteq \mathbb{R}$, 
\begin{equation}
\mu _{n}^{f,A}(O)=\mu _{f}\left( \left\{ z\text{ }\mid \,A_{n}(z)\in
O\right\} \right) .  \label{binc781}
\end{equation}%
In this case, the large deviation rate function $I_{f,A}$ is the Legendre
transform of $\hat{c}_{f,A}$. Compare indeed (\ref{feij67}) with (\ref%
{feij50}) and (\ref{feij68}). Similar to the special case $f=-\log d$
discussed above, a simple computation using Theorem \ref{rer1} shows that%
\begin{eqnarray}
I_{f,A}(x) &=&tx-\hat{c}_{f,A}(t)=tx-P(f+tA)+P(f)  \notag \\
&=&tx-\log \lambda _{f+tA}+P(f)  \label{binc78x}
\end{eqnarray}%
for each real parameter $t$ satisfying the self-consistency condition 
\begin{equation}
\hat{c}_{f,A}^{\prime }(t)=x=\frac{d\,P(f+tA)}{dt}=\mu _{f+tA}(A)=\int
A(x)\mu _{f+tA}(\mathrm{d}x),  \label{uut}
\end{equation}%
where $\mu _{f+tA}$ is the linear equilibrium probability for the potential $%
f+t\,A$. Similar to (\ref{myt68})--(\ref{myt68bis}), for $t$ satisfying %
\eqref{uut}, we infer from (\ref{binc78x}) that 
\begin{equation}
I_{f,A}(x)=P(f)-h(\mu _{f+tA})-\mu _{f+tA}(f)\geq 0,  \label{binc78xz}
\end{equation}%
thanks to the definition (\ref{iore}) of the linear pressure of the
(normalized) H\"{o}lder potential $f:\Omega \rightarrow \mathbb{R}$. Note,
however, that 
\begin{equation*}
P(\log d+t\,A)=\log d+P(t\,A)=\log d+\log \lambda _{tA},
\end{equation*}%
but, in general, $P(f+tA)\neq P(f)+P(tA)$.

Using the above facts, in particular \eqref{binc78xz}, our arguments in
Sections \ref{Bogo} and \ref{mae} can be adapted to replace the maximal
entropy probability $\mu $ with any linear equilibrium probability $\mu _{f}$
associated with a (normalized) H\"{o}lder potential $f$. We will leave it to
the interested reader to work out the details, and will only provide the
relevant information.

\begin{remark}
\label{iimp}If one replaces $\mu $ with $\mu _{f}$ as above, given a convex
or concave function $F$, for nonlinear cases, one shall consider the
variational problem 
\begin{equation}
\sup_{\rho \in \mathcal{P}(T)}\left\{ \rho (f)+F(\rho (A))+h(\rho )\right\} ,
\label{bel33}
\end{equation}%
which generalizes (\ref{Kbom}) (see also (\ref{kjew})) beyond the special
choice $f=-\log d$. Our focus here is the original problem (\ref{Kbom}).
Questions related to the more general choice of probability $\mu _{f}$ will
be further discussed in Section \ref{conj} (see (\ref{xiux23})).
\end{remark}

\begin{remark}
\label{level-2 large deviations} Within the level-2 large deviations,
similar to the original case $f=-\log d$, considering the empirical
probabilities $\sum_{j=0}^{n-1}\delta _{T^{j}(z)}$ for the equilibrium
probability $\mu _{f}$, the associated large deviation rate function is now 
\begin{equation}
I(\rho )=P(f)-h(\rho )-\rho (f)\geq 0,  \label{biz1}
\end{equation}%
where, as before, $h(\rho )$ is the entropy of the invariant probability $%
\rho $ taken as $\infty $ for non-invariant probabilities. See \cite{Comman}%
, \cite{Kif}, \cite{ComLe} and \cite{Comman1}. Compare with Equation (\ref%
{binc78xz}) and Remark \ref{mmx}.
\end{remark}

\subsection{Quadratic mean-field probabilities\label{Quadratic mean-field
probabilities}}

In Section \ref{conj} we consider a related but different problem. Given $%
\beta >0$ and two H\"{o}lder functions $g,f:\Omega \rightarrow \mathbb{R}$,
we will be interested in the weak$^{\ast }$ limit probability measure $%
\mathfrak{m}=\mathfrak{m}_{\beta ,f,g}$ on $\Omega $, which, for any
continuous (real-valued) function $\psi \in C(\Omega )$, satisfies 
\begin{equation*}
\mathfrak{m}(\psi )=\int \psi (x)\mathfrak{m}(\mathrm{d}x)=
\end{equation*}
\begin{equation}
\lim_{n\rightarrow \infty }\frac{\int \psi (x)e^{\frac{\beta n}{2}%
g_{n}(x)^{2}}\mu _{f}(\mathrm{d}x)}{\int e^{\frac{\beta n}{2}%
g_{n}(x)^{2}}\mu _{f}(\mathrm{d}x)}=\lim_{n\rightarrow \infty }\frac{\mu
_{f}\left( \psi e^{\frac{\beta n}{2}g_{n}^{2}}\right) }{\mu _{f}\left( e^{%
\frac{\beta n}{2}g_{n}^{2}}\right) },  \label{xiux23}
\end{equation}%
where, for any $\varphi \in C(\Omega )$, we recall that $\varphi _{n}$, $%
n\in \mathbb{N}$, are the so-called the Birkhoff averages defined by (\ref%
{Eq Birk Av}) and $\mu _{f}$ is the linear equilibrium probability for the H%
\"{o}lder potential $f$. Notice that the limit probability $\mathfrak{m}$
(when it exists) is not necessarily $T$-invariant.

\begin{definition}
\label{loloy}We call such a weak$^{\ast }$ limit $\mathfrak{m}=\mathfrak{m}%
_{\beta ,f,g}$ the \textit{quadratic mean-field Gibbs probability} for $%
\beta ,\mu _{f}$ and $g$.
\end{definition}

\begin{definition}
\label{gog1}We say that there is no quadratic mean-field Gibbs phase
transition for the H\"{o}lder potentials $f,g$ and parameter $\beta \in 
\mathbb{R}$, when the weak$^{\ast }$ limit \eqref{xiux23} exist and is equal
to the eigenprobability of the adjoint of the Ruelle operator for some H\"{o}%
lder potential.
\end{definition}

\begin{definition}
\label{gogrexx}We say that a finite mean-field Gibbs phase transition takes
place for the H\"{o}lder potentials $f,g$ and parameter $\beta >0$ if the
corresponding $\mathfrak{m}$ is a non-trivial convex combination of
eigenprobabilities for different (not cohomologous\footnote{%
That is, the difference $f-g$ is not of the form $A\circ T-A$ for some H\"{o}%
lder potential $A$.}) H\"{o}lder potentials. That is, for a finite sequence
of H\"{o}lder potentials, $f_{j}$, that are not cohomologous to each other
and strictly positive constants, $\alpha _{j}$, whose sum is $1$,%
\begin{equation}
\mathfrak{m}(\psi )=\sum_{j}\alpha _{j}\nu _{f_{j}}\left( \psi \right)
\label{be1}
\end{equation}%
for any continuous function $\psi \in C(\Omega )$.
\end{definition}

As a working example, we will consider later in Section \ref{conj} the
maximal entropy probability $\mu $ and a H\"{o}lder potential $g:\{-1,1\}^{%
\mathbb{N}}\rightarrow \mathbb{R}$ of the form 
\begin{equation}
g(x)=g(x_{1},x_{2},.\ldots ,x_{n},\ldots )=\sum_{n=1}^{\infty }a_{n}\,x_{n},
\label{binc2xxx}
\end{equation}%
where $a_{n}$ is a sequence converging exponentially fast to zero.

The following result, which is proved in Section \ref{conj}, is a
consequence of the Ruelle(-Perron-Frobenius) theorem (Theorem \ref{rer1}):

\begin{theorem}
\label{otwr}If the potential $g$ defined by (\ref{binc2xxx}) is not zero,
given $\beta >0$ and $f=-\log 2$, one has that 
\begin{equation}
\mathfrak{m}=\mathfrak{m}_{\beta ,f,g}=\frac{\mu \left( h_{f+\beta
t_{1}g}\right) \nu _{f+\beta t_{1}g}+\mu \left( h_{f+\beta t_{2}g}\right)
\nu _{f+\beta t_{2}g}}{\mu \left( h_{f+\beta t_{1}g}\right) +\mu \left(
h_{f+\beta t_{2}g}\right) },  \label{xiuxxx}
\end{equation}%
where $t_{1},t_{2}$ are the two self-consistent parameters (see Remarks \ref%
{poiu} and \ref{yyx} and discussions of Section \ref{mae}), $\mu $ is the
maximal entropy probability and $h_{f+\beta t_{j}g}$, $\nu _{f+\beta t_{j}g}$%
, $j=1,2$, are respectively the main eigenfunction and the eigenprobability
for the Ruelle operator $\mathcal{L}_{f+\beta t_{j}g}$.
\end{theorem}

\noindent All these objects -- $t_{1}$, $t_{2}$, $h_{f+\beta t_{1}g}$, $\nu
_{f+\beta t_{2}g}$ -- can be explicitly computed and we can show the
occurrence of a finite mean-field Gibbs phase transition. In fact, in Remark %
\ref{boex} we present a case where there is no finite mean-field Gibbs phase
transition, and in Remark \ref{yyx}, a case where there exists.

To prove this theorem, we adapt the arguments of Section 2 of \cite{LeWa1}.
See Section \ref{conj} where we consider a more general setting, in which $%
\mu $ is replaced with equilibrium probabilities $\mu _{f}$ of general H\"{o}%
lder potentials $f$.

A remarkable fact is that for $f=-\log 2$ and a potential $g$ as in %
\eqref{binc2xxx}, the quadratic equilibrium probabilities (as mentioned
before in \eqref{Kibom}), that is the $T$-invariant probabilities maximizing 
$\mathfrak{P}_{2}(g)$, are the $T$-invariant probabilities $\mu _{f+\beta
t_{j}g}$, $j=1,2$, where $t_{j},$ $j=1,2$, satisfy the self-consistency
conditions as above.

In this paper, we also study a second type of quadratic mean-field
probabilities: Let $g:\Omega \rightarrow \mathbb{R}$ be again any fixed H%
\"{o}lder potential. For all $\beta >0$ and $n\in \mathbb{N}$, define the
probability measure $\mathfrak{M}^{(n)}$ on $\Omega $ by 
\begin{equation}
\mathfrak{M}^{(n)}(\psi )=\mathfrak{M}_{g,\beta }^{(n)}(\psi ):=\frac{\mu
\left( \psi _{n}e^{\frac{\beta n}{2}g_{n}^{2}}\right) }{\mu \left( e^{\frac{%
\beta n}{2}g_{n}^{2}}\right) }=\frac{\int \psi _{n}\left( x\right) e^{\frac{%
\beta n}{2}g_{n}\left( x\right) ^{2}}\mu (\mathrm{d}x)}{\int e^{\frac{\beta n%
}{2}g_{n}\left( x\right) ^{2}}\mu (\mathrm{d}x)}  \label{q1a1}
\end{equation}%
for any continuous (real-valued) function $\psi \in C(\Omega )$, where, for
any $\varphi \in C(\Omega )$, we recall again that $\varphi _{n}$, $n\in 
\mathbb{N}$, are the so-called the Birkhoff averages defined by (\ref{Eq
Birk Av}).

In Section \ref{QLVE}, we will be interested in\ weak$^{\ast }$ limits of
convergent subsequences $\mathfrak{M}^{(n_{k})}\rightarrow \mathfrak{M}%
^{\infty }$, a problem different from the one addressed in \eqref{xiux23}.

\begin{definition}
\label{dedeus} Any probability $\mathfrak{M}^{\infty }=\mathfrak{M}_{g,\beta
}^{\infty }$, which is the weak$^{\ast }$ limit of a convergent subsequence $%
\mathfrak{M}^{(n_{k})}$, $k\rightarrow \infty $, is called here a \textit{%
quadratic mean-field equilibrium probability} for the pair $g,\beta $.
\end{definition}

\noindent We will show in Section \ref{QLVE} the following statements:

\begin{theorem}
\label{kde}Given a H\"{o}lder potential $g:\Omega \rightarrow \mathbb{R}$, a
quadratic mean-field equilibrium probability is always $T$-invariant, and it
is in the closed convex hull of the \textit{nonlinear (quadratic)}
equilibrium probabilities for $g$.
\end{theorem}

\begin{corollary}
\label{kde coro}If there is a non-ergodic mean-field equilibrium probability
then the \textit{nonlinear (quadratic)} equilibrium probabilities for a H%
\"{o}lder potential $g:\Omega \rightarrow \mathbb{R}$ is non-unique, i.e., a
(nonlinear) phase transition takes place.
\end{corollary}

Another problem related to nonlinear phase transitions is addressed in
Section \ref{Til}: Given continuous functions $A:\Omega \rightarrow \mathbb{R%
}$ and $F:\mathbb{R}\rightarrow \mathbb{R}$, and a natural number $n\in 
\mathbb{N}$, we define the probability $\mathrm{m}_{n}^{F,A}$ on $\mathbb{R}$
by 
\begin{equation}
\mathrm{m}_{n}^{F,A}(O)=\frac{\int_{O}e^{nF(x)}\mu _{n}^{A}(\mathrm{d}x)}{%
Z_{n}^{F,A}}  \label{zeno2}
\end{equation}%
for any any open interval $O\subseteq \mathbb{R}$, where $\mu _{n}^{A}$ is
the probability measure defined by \eqref{binc78} and%
\begin{equation*}
Z_{n}^{F,A}=\mu _{n}^{A}\left( e^{nF}\right) =\int e^{nF(x)}\mu _{n}^{A}(%
\mathrm{d}x).
\end{equation*}%
(Notice here that $\mathrm{m}_{n}^{F,A}$, $n\in \mathbb{N}$, are measures on 
$\mathbb{R}$, and not on $\Omega $ as in \eqref{xiux23}, (\ref{q1a1}) or in
Equation (12) of \cite{LeWa1}.) Then, an important question is how to
estimate the limit (or limit of subsequences) 
\begin{equation}
\lim_{n\rightarrow \infty }\mathrm{m}_{n}^{F,A}(O)=:\theta (O).
\label{zero31}
\end{equation}%
We will address it in Section \ref{Til} and show from the LDP\footnote{%
Recall that LDP\ refers to \textquotedblleft Large Deviation
Principle\textquotedblright .} tilting property that the exponential
convergence rate of the above limit is directly related to Bogoliubov's
variational principle (which, in turn, encodes nonlinear phase transitions).
Also, in this case, we will present explicit examples.

\section{Bogoliubov's variational principle}

\label{Bogo}

As before, we will consider in this section the symbolic space $\Omega
=\{1,2,\ldots ,d\}^{\mathbb{N}}$ for general $d\in \mathbb{N}$, along with
the action of the shift $T:\Omega \rightarrow \Omega $. Recall that $\mu $
is the maximal entropy probability, i.e., the equilibrium probability for
the constant potential $A=-\log d$. Throughout this section, a fixed H\"{o}%
lder potential $A:\Omega \rightarrow \mathbb{R}$ is considered.

Let $\Phi :\mathbb{R}\rightarrow \mathbb{R\cup \{\infty \}}$ be any function
and $\Phi ^{\ast }:\mathbb{R}\rightarrow \mathbb{R\cup \{\infty \}}$ its
Legendre-Fenchel transform, i.e., the convex lower semicontinuous function
defined by%
\begin{equation*}
\Phi ^{\ast }(s)=\sup_{x\in \mathbb{R}}\left\{ sx-\Phi (x)\right\} .
\end{equation*}%
Observe that, by Fenchel's theorem, if $\Phi $ is itself convex and lower
semicontinuous, then it is equal to its double Legendre-Fenchel transform,
that is, the Legendre-Fenchel transform of $\Phi ^{\ast }$ is nothing else
but the original function $\Phi $. In other words, the Legendre-Fenchel
transform defines an involution in the set of all convex lower
semicontinuous functions $\mathbb{R}\rightarrow \mathbb{R\cup \{\infty \}}$.

For the given H\"{o}lder potential $A$, let $I_{A}:\mathbb{R}\rightarrow 
\mathbb{R\cup \{\infty \}}$ be the Legrendre-Fenchel transform of the convex
continuous function 
\begin{equation*}
s\mapsto \hat{c}(s):=P(sA)-\log d,
\end{equation*}%
where $P(sA)$ is the pressure of the potential $sA$. See in particular
Equations (\ref{feij502bis0}) and (\ref{feij68}) of the previous section. In
particular, by Fenchel's theorem, $I_{A}^{\ast }(s)=\hat{c}(s)$, since $\hat{%
c}$ is continuous and convex.

As mentioned before, the distributions $\mu _{n}=\mu _{n}^{A}$, $n\in 
\mathbb{N}$, of the Birkhoff averages $A_{n}$ (see (\ref{binc78})) satisfy a
Large Deviation Principle (LDP), whose rate function is precisely $I_{A}$.
Then, Equation (\ref{feij502bis0}) can be rewritten as follows: for all $%
s\in \mathbb{R}$,%
\begin{equation*}
\hat{c}(s)=\lim_{n\rightarrow \infty }\frac{1}{n}\ln \left( \int e^{nsx}\mu
_{n}^{A}(\mathrm{d}x)\right) =\sup_{x\in \mathbb{R}}\left\{
sx-I_{A}(x)\right\} .
\end{equation*}%
Recall also (\ref{impi}), i.e., for any continuous functions $F:\mathbb{R}%
\rightarrow \mathbb{R}$ and $A:\Omega \rightarrow \mathbb{R}$, 
\begin{equation}
\hat{P}(F):=\lim_{n\rightarrow \infty }\frac{1}{n}\ln \left( \int
e^{nF(x)}\mu _{n}^{A}(\mathrm{d}x)\right) =\sup_{x\in \mathbb{R}}\left\{
F(x)-I_{A}(x)\right\} ,  \label{fgh}
\end{equation}%
thanks to the Varadhan(-Bryc) lemma. We show below that, up to an explicit
constant, this quantity is nothing else but the nonlinear pressure (\ref%
{Kbom}) (up to the constant $-\log d$), that is,%
\begin{equation*}
\hat{P}(F)=\mathfrak{P}_{F}-\log d
\end{equation*}%
with 
\begin{equation}
\mathfrak{P}_{F}=\mathfrak{P}_{F,A}:=\sup_{\rho \in \mathcal{P}(T)}\left\{
F(\rho (A))+h(\rho )\right\} ,  \label{Kbombis}
\end{equation}%
when the continuous function $F$\ is either convex or concave.

\begin{remark}
\label{hdf}Let $\mathcal{M}$ be the set of all finite measures on $\Omega $.
It is natural to extend the entropy $h(\rho )$ to those measures $\rho \in 
\mathcal{M}$ that are not $T$-invariant probabilities, just by assigning to
them the value $-\infty $, similar to what is done in Remark \ref{level-2
large deviations}. Then, when taking the supremum over $\rho \in \mathcal{M}$
in expressions containing the term $h(\rho )$, the elements $\rho \in 
\mathcal{M}$\ that are not $T$-invariant probabilities are simply
disregarded. This will be done tacitly for the rest of the article.
\end{remark}

\subsection{The convex case\label{The convex case}}

In fact, if $F$ is \emph{convex}, as in the example $F(x)=x^{2}/2$, then one
arrives at \textquotedblleft Bogoliubov's variational
principle\textquotedblright\ for the nonlinear pressure, by writing $F$ as
its double Legendre-Fenchel transform and by commuting two suprema: 
\begin{eqnarray*}
\hat{P}(F) &=&\sup_{x\in \mathbb{R}}\left\{ F(x)-I_{A}(x)\right\} \\
&=&\sup_{x\in \mathbb{R}}\left\{ \sup_{s\in \mathbb{R}}\left\{ xs-F^{\ast
}(s)\right\} -I_{A}(x)\right\} \\
&=&\sup_{s\in \mathbb{R}}\left\{ -F^{\ast }(s)+\sup_{x\in \mathbb{R}}\left\{
xs-I_{A}(x)\right\} \right\} \\
&=&\sup_{s\in \mathbb{R}}\left\{ -F^{\ast }(s)+P(sA)-\log d\right\} .
\end{eqnarray*}%
For example, if $F(x)=x^{2}/2$ then $F^{\ast }(s)=s^{2}/2$ and 
\begin{equation*}
\hat{P}(F)=\sup_{s\in \mathbb{R}}\left\{ P(sA)-s^{2}/2\right\} -\log d.
\end{equation*}

Using Bogoliubov's variational principle, i.e., the equality 
\begin{equation}
\hat{P}(F)=\sup_{s\in \mathbb{R}}\left( -F^{\ast }(s)+\hat{c}(s)\right)
=\sup_{s\in \mathbb{R}}\left( -F^{\ast }(s)+P(sA)-\log d\right) ,
\label{utq}
\end{equation}%
and writing the classical pressure $P(sA)$ defined by \eqref{iore} as the
Legendre-Fenchel transform 
\begin{equation*}
P(sA)=\sup_{\rho \in \mathcal{M}}\left\{ \rho (sA)+h(\rho )\right\} -\log d%
\text{ },
\end{equation*}%
of (minus) the entropy $h$ by meanwhile taking into account Remark \ref{hdf}%
, we arrive at the following representation of $\hat{P}(F)$ in terms of a
variational principle for finite measures: 
\begin{eqnarray}
\hat{P}(F) &=&\sup_{s\in \mathbb{R}}\left\{ -F^{\ast }(s)+\sup_{\rho \in 
\mathcal{M}}\left\{ s\rho (A)+h(\rho )\right\} -\log d\right\}
\label{dfdfdfddfdf} \\
&=&\sup_{\rho \in \mathcal{M}}\sup_{s\in \mathbb{R}}\left\{ -F^{\ast
}(s)+s\rho (A)+h(\rho )-\log d\right\}  \notag \\
&=&\sup_{\rho \in \mathcal{M}}\left\{ h(\rho )-\log d+\sup_{s\in \mathbb{R}%
}\left\{ s\rho (A)-F^{\ast }(s)\right\} \right\}  \notag \\
&=&\sup_{\rho \in \mathcal{M}}\left\{ F(\rho (A))+h(\rho )-\log d\right\} ,
\label{klir}
\end{eqnarray}%
provided that $F=F^{\ast \ast }$, like when $F$ is continuous and convex.
From Remark \ref{hdf}, observe that the last $\sup $ is attained in $%
\mathcal{P}(T)\subseteq \mathcal{M}$. It also follows from last equality
that 
\begin{equation*}
\hat{P}(F)=\mathfrak{P}_{F,A}-\log d,
\end{equation*}%
see Equation (\ref{Kbombis}) given just above.

We call the functional $\mathfrak{p}:\mathcal{M}\rightarrow \mathbb{R}\cup
\{\infty \}$ defined by%
\begin{equation}
\mathfrak{p}(\rho ):=F(\rho (A))+h(\rho )-\log d,\text{ \ \ }\rho \in 
\mathcal{M},  \label{eq press func convex}
\end{equation}%
the nonlinear pressure functional associated with $A$ and $F$. This
terminology is consistent with Definition 2.4 of \cite{Vele}. Observe that
maximizers of this functional are precisely the nonlinear equilibrium
probabilities associated to $A$ and $F$, as given by Definition \ref%
{nonlinear equilibrium probabilities}.

By using Bogoliubov's variational principle again, we will show that
nonlinear equilibrium probabilities are self-consistent linear equilibrium
probabilities for continuous and convex functions $F$: Let $\varpi \in 
\mathcal{P}(T)$ be a nonlinear equilibrium probability, that is, $\mathfrak{p%
}(\varpi )=\hat{P}(F)$, thanks to\ Equation (\ref{klir}). Assume
additionally that%
\begin{equation}
F(\varpi (A))=\sup_{s\in \mathbb{R}}\left\{ s\varpi (A)-F^{\ast }(s)\right\}
=\max_{s\in \mathbb{R}}\left\{ s\varpi (A)-F^{\ast }(s)\right\} ,
\label{hyp1}
\end{equation}%
that is, the supremum with respect to $s$\ is attained. In other words,
there is $\bar{s}\in \mathbb{R}$ maximizing 
\begin{equation*}
s\mapsto -F^{\ast }(s)+s\varpi (A)\text{ }.
\end{equation*}
Note that this trivially holds when 
\begin{equation*}
\lim_{|s|\rightarrow \infty }\left\vert \frac{F^{\ast }(s)}{s}\right\vert
=\infty \text{ }.
\end{equation*}%
Then, for any nonlinear equilibrium probability $\varpi \in \mathcal{P}(T)$,
provided such a $\bar{s}\in \mathbb{R}$ exists, we deduce from elementary
manipulations that 
\begin{equation*}
\hat{P}(F)=\varpi (\bar{s}A)-F^{\ast }(\bar{s})+h(\varpi )-\log d=-F^{\ast }(%
\bar{s})+\sup_{\rho \in \mathcal{M}}\left\{ \rho (\bar{s}A)+h(\rho )\right\}
-\log d,
\end{equation*}%
which implies that $\varpi $ is the unique equilibrium probability for the H%
\"{o}lder potential $\bar{s}A$.

We show next that $\bar{s}$ must be a solution to Bogoliubov's variational
problem. Observe first from the hypothesis (\ref{hyp1}) for a nonlinear
equilibrium probability $\varpi \in \mathcal{P}(T)$ that 
\begin{eqnarray}
\hat{P}(F) &=&\sup_{\rho \in \mathcal{M}}\sup_{s\in \mathbb{R}}\left\{
-F^{\ast }(s)+s\rho (A)+h(\rho )-\log d\right\}  \notag \\
&=&-F^{\ast }(\bar{s})+\varpi (\bar{s}A)+h(\varpi )-\log d  \notag \\
&=&-F^{\ast }(\bar{s})+\sup_{\rho \in \mathcal{M}}\left\{ \rho (\bar{s}%
A)+h(\rho )\right\} -\log d  \notag \\
&=&-F^{\ast }(\bar{s})+P(\bar{s}A)-\log d\text{ },  \label{eq Bogo legal 1}
\end{eqnarray}%
bearing in mind the definition of the linear pressure, Equation (\ref{iore}%
), and Remark \ref{hdf}. By Bogoliubov's variational principle (\ref{utq}),
we deduce that $\bar{s}$ is necessarily a maximizer of the function 
\begin{equation}
s\mapsto -F^{\ast }(s)+P(sA)-\log d=-F^{\ast }(s)+\log \lambda _{-\log
d+sA}\ .  \label{sisi0}
\end{equation}%
The last above equality is a straightforward consequence of Equation %
\eqref{iore43} and the Ruelle(-Perron-Frobenius) theorem (Theorem \ref{rer1}%
).

\begin{remark}
\label{sisim}Suppose that $A$ is such that $P(sA)=P(-sA)$ and $F^{\ast
}(s)=F^{\ast }(-s)$ for all $s\in \mathbb{R}$. In this case, we get that $%
\bar{s}$ maximizes (\ref{sisi0}) if and only if $-\bar{s}$ also has this
property. In particular, (\ref{sisi0}) has more than one maximizer when $%
\bar{s}>0$. In the case $F$ is quadratic (and convex), i.e., $F(x)=\beta
x^{2}/2$ for some $\beta >0$, we will provide examples of nonlinear phase
transitions, as done in Section \ref{mae}. In these examples, the nonlinear
equilibrium probabilities can even be explicitly determined.
\end{remark}

Now, if $F^{\ast }$ is differentiable, then we can infer from the above
observations, in particular the fact that $\varpi $ is the unique
equilibrium probability for the H\"{o}lder potential $\bar{s}A$, that%
\footnote{%
The equality $\varpi (A)=\left. dP(sA)/ds\right\vert _{s=\bar{s}}$ is a
consequence of the weak$^{\ast }$ compactness of the space of $T$-invariant
probabilities, the weak$^{\ast }$ upper semicontinuity of the entropy, and
the fact that the convex function $P$ is Gateaux differentiable, that is, it
has a unique tangent functional at any point. We omit the details.} 
\begin{equation}
\varpi (A)=\left. \frac{d}{ds}P(sA)\right\vert _{s=\bar{s}}=(F^{\ast
})^{\prime }(\bar{s}).  \label{ciir}
\end{equation}%
Assume now that $(F^{\ast })^{\prime }$ is injective, that is, $(F^{\ast
})^{\prime }$ is strictly increasing, meaning that $F^{\ast }$ is strictly
convex. Denote by $\chi $ the inverse of $(F^{\ast })^{\prime }$ on its
image. Then, we arrive at 
\begin{equation}
\left. \frac{d}{ds}P(sA)\right\vert _{s=\chi (\varpi (A))}=\varpi (A),
\label{geger}
\end{equation}%
which is a self-consistency condition saying that $\varpi $ is an
equilibrium probability for the potential $\chi (\varpi (A))A$.

Conversely, using the same assumptions on $F^{\ast }$ and similar arguments,
one shows that, for any solution $\bar{s}\in \mathbb{R}$ to Bogoliubov's
variational problem, that is, any maximizer of (\ref{sisi0}), there is a
unique linear equilibrium probability $\varpi $ for the potential $\bar{s}A$
satisfying $\bar{s}=\chi (\varpi (A))$\ and which is meanwhile a nonlinear
equilibrium probability, that is, it maximizes the nonlinear pressure $%
\mathfrak{p}$ defined by (\ref{eq press func convex}).

In the quadratic case $F(x)=x^{2}/2$, one has that $(F^{\ast })^{\prime
}(s)=s$. Thus, the self-consistency equation reads in this case: 
\begin{equation}
\mu _{\bar{s}A}\left( A\right) =\int A(x)\mu _{\bar{s}A}\left( \mathrm{d}%
x\right) =\bar{s},  \label{srt}
\end{equation}%
where $\mu _{\bar{s}A}$ is the unique linear equilibrium probability for the
H\"{o}lder potential $\bar{s}A$. In Sections \ref{qua} and \ref{mae}, this
special (quadratic) case is analyzed in detail. In particular, in Section %
\ref{mae} we give examples for which there is more than one solution $\bar{s}
$ to Equation \eqref{srt}.

If one considers the linear equilibrium probability $\mu _{f}$ of a general H%
\"{o}lder potential $f$ instead of the maximum entropy probability $\mu $,
which corresponds to the special choice $f=-\log d$, and the corresponding
changes in \eqref{binc78} and \eqref{binc78x}, then \eqref{klir} has to be
adapted, taking into account the new large deviation rate function %
\eqref{binc78xz}. This yields a more general version of equality (\ref{klir}%
):%
\begin{equation}
\hat{P}(F)=\sup_{\rho \in \mathcal{M}}\left\{ F(\rho (A))+h(\rho )+\rho
(f)-P(f)\right\} .  \label{bel1}
\end{equation}%
Correspondingly, the self-consistency conditions should now be given by the
critical points of the function 
\begin{equation}
s\mapsto -F^{\ast }(s)+P(f+sA)=-F^{\ast }(s)+\log \lambda _{f+sA},
\label{joq1}
\end{equation}%
instead of \eqref{sisi0}.

\subsection{The concave case}

\label{conc}

The argument applied to the convex case in Section \ref{The convex case}
cannot be directly used if $F$ is concave and needs further justification.
First, rather than $F$, we have to write $G=-F$, which is continuous and
convex for continuous and concave $F$, as its double Legendre-Fenchel
transform to obtain from (\ref{fgh}) the identity%
\begin{eqnarray*}
\hat{P}(F) &=&\sup_{x\in \mathbb{R}}\left\{ -G(x)-I_{A}(x)\right\} \\
&=&\sup_{x\in \mathbb{R}}\left\{ -\sup_{s\in \mathbb{R}}\left\{ xs-G^{\ast
}(s)\right\} -I_{A}(x)\right\} \\
&=&\sup_{x\in \mathbb{R}}\left\{ \inf_{s\in \mathbb{R}}\left\{ -xs+G^{\ast
}(s)\right\} -I_{A}(x)\right\} .
\end{eqnarray*}%
By commuting the infimum and the supremum, we would then arrive at 
\begin{eqnarray}
\hat{P}(F) &=&\inf_{s\in \mathbb{R}}\sup_{x\in \mathbb{R}}\left\{
-xs+G^{\ast }(s)-I_{A}(x)\right\}  \notag \\
&=&\inf_{s\in \mathbb{R}}\left\{ G^{\ast }(s)+\sup_{x\in \mathbb{R}}\left\{
-xs-I_{A}(x)\right\} \right\}  \notag \\
&=&\inf_{s\in \mathbb{R}}\left\{ G^{\ast }(s)+P(-sA)-\log d\right\}  \notag
\\
&=&\inf_{s\in \mathbb{R}}\left\{ G^{\ast }(-s)+P(sA)-\log d\right\} ,
\label{erer}
\end{eqnarray}%
which is Bogoliubov's variational principle for the concave nonlinear
pressure. To justify the commutation of $\inf $ and $\sup $ we will use
Sion's min-max theorem, which is presented below.

To this end, note first that the domain of the rate function $I_{A}$ is
nothing but the finite interval $(m_{A},M_{A})$, see Equation (\ref{domain I}%
). Therefore, 
\begin{equation*}
\hat{P}(F)=\sup_{x\in \mathbb{R}}\left\{ -G(x)-I_{A}(x)\right\} =\sup_{x\in
(m_{A},M_{A})}\left\{ -G(x)-I_{A}(x)\right\} .
\end{equation*}%
Thus, assuming (as done above for $F^{\ast }$) that $G^{\ast }$ grows faster
than linearly, that is, 
\begin{equation}
\lim_{|x|\rightarrow \infty }\left\vert \frac{G^{\ast }(x)}{x}\right\vert
=\infty ,  \label{rtytry}
\end{equation}%
there is a \emph{compact} convex subset $K\subseteq \mathrm{dom}(G^{\ast
})\subseteq \mathbb{R}$ such that, for all $x\in (m_{A},M_{A})$, 
\begin{equation*}
\inf_{s\in \mathbb{R}}\left\{ -xs+G^{\ast }(s)\right\} =\min_{s\in K}\left\{
-xs+G^{\ast }(s)\right\} .
\end{equation*}%
Therefore, in this case, we can write the following equalities:%
\begin{eqnarray}
\hat{P}(F) &=&\sup_{x\in (m_{A},M_{A})}\left\{ -G(x)-I_{A}(x)\right\}  \notag
\\
&=&\sup_{x\in (m_{A},M_{A})}\left\{ -\sup_{s\in \mathbb{R}}\left\{
xs-G^{\ast }(s)\right\} -I_{A}(x)\right\}  \notag \\
&=&\sup_{x\in (m_{A},M_{A})}\min_{s\in K}\left\{ -xs+G^{\ast
}(s)-I_{A}(x)\right\} .  \label{sdsdsd}
\end{eqnarray}%
Observe that, for all $x\in (m_{A},M_{A})$, the mapping 
\begin{equation*}
s\mapsto -xs+G^{\ast }(s)-I_{A}(x)
\end{equation*}%
from the compact convex subset $K$ to $\mathbb{R}$ is convex and lower
semicontinuous, while, for all $s\in K$, the mapping 
\begin{equation*}
x\mapsto -xs+G^{\ast }(s)-I_{A}(x)\text{ },
\end{equation*}%
from $(m_{A},M_{A})$ to $\mathbb{R}$ is concave and upper semicontinuous.

Now we recall Sion's min-max theorem: Given a convex set $C$, a function $%
g:C\rightarrow \mathbb{R}$ is \textquotedblleft
quasi-convex\textquotedblright\ if all its level sets $g^{-1}((-\infty
,\alpha ))$, $\alpha \in \mathbb{R}$, are convex. Clearly, any convex
function is quasi-convex. (The converse is however, not true: for instance,
the function $g:\mathbb{R}\rightarrow \mathbb{R}$, $g(x)=|x|^{\frac{1}{2}}$,
is quasi-convex, but not convex.). $g:C\rightarrow \mathbb{R}$ is
\textquotedblleft quasi-concave\textquotedblright\ if $-g$ is quasi-convex.
Again, concave functions are special cases of quasi-concave ones. Having
these two definitions in mind, we can state Sion's min-max theorem:

\begin{proposition}[Sion's min-max theorem]
\label{prop Sion}Let $X$ and $Y$ be real topological vector spaces, $%
K\subseteq X$ a compact convex set, $C\subseteq Y$ a convex set and $%
f:K\times C\rightarrow \mathbb{R}$ a function such that, for all $(x,y)\in
K\times C$, $f(\cdot ,y):K\rightarrow \mathbb{R}$ is lower semicontinuous
and quasi-convex, whereas $f(x,\cdot ):C\rightarrow \mathbb{R}$ is upper
semicontinuous and quasi-concave. Then one has the following equality: 
\begin{equation*}
\min_{x\in K}\sup_{y\in C}f(x,y)=\sup_{y\in C}\min_{x\in K}f(x,y).
\end{equation*}%
In particular, $f$ has a conservative value, that is,%
\begin{equation*}
\inf_{x\in K}\sup_{y\in C}f(x,y)=\sup_{y\in C}\inf_{x\in K}f(x,y)\in \mathbb{%
R}.
\end{equation*}
\end{proposition}

\noindent For a simple proof of the above proposition, see \cite{Sion}.
Notice that Sion's min-max theorem does not imply the existence of a saddle
point for $f$, that is, a pair $(\bar{x},\bar{y})\in K\times C$ satisfying%
\begin{equation*}
f(\bar{x},\bar{y})=\sup_{y\in C}f(\bar{x},y)=\inf_{x\in K}f(x,\bar{y}%
)=\sup_{y\in C}\inf_{x\in K}f(x,y).
\end{equation*}%
By the celebrated \textquotedblleft von Neumann's min-max
theorem\textquotedblright , this situation occurs when $C$ is additionally
compact and the functions are not only quasi-convex or quasi-concave but
convex or concave:

\begin{proposition}[von Neumann's min-max]
Let $K$, $C$ and $f$ be as in Proposition \ref{prop Sion}. Assume
additionally that $C$ is compact, $f(\cdot ,y):K\rightarrow \mathbb{R}$ is
convex, and $f(x,\cdot ):C\rightarrow \mathbb{R}$ is concave. Then $f$ has a
saddle point $(\bar{x},\bar{y})\in K\times C$.
\end{proposition}

According to Sion's min-max theorem and the above observations, we obtain
from (\ref{sdsdsd}) that%
\begin{eqnarray}
\hat{P}(F) &=&\sup_{x\in (m_{A},M_{A})}\left\{ \min_{s\in K}\left\{
-xs+G^{\ast }(s)\right\} -I_{A}(x)\right\}  \notag \\
&=&\min_{s\in K}\left\{ G^{\ast }(s)+\sup_{x\in (m_{A},M_{A})}\left\{
-xs-I_{A}(x)\right\} \right\}  \notag \\
&=&\min_{s\in K}\left\{ G^{\ast }(s)+P(-sA)-\log d\right\}  \notag \\
&=&\inf_{s\in \mathbb{R}}\left\{ G^{\ast }(-s)+P(sA)-\log d\right\} .
\label{Bogoliubov's variational principle}
\end{eqnarray}
This is nothing else but Bogoliubov's variational principle asserted in (\ref%
{erer}) for the concave nonlinear pressure.

Exactly as in the convex case (cf. (\ref{dfdfdfddfdf})), by representing $%
P(sA)$ via the variational principle for finite measures, from the last
equality, we arrive at the nonlinear pressure functional for invariant
probabilities:%
\begin{eqnarray}
\hat{P}(F) &=&\inf_{s\in \mathbb{R}}\left\{ G^{\ast }(-s)+P(sA)\right\}
=\min_{s\in K}\left\{ G^{\ast }(-s)+P(sA)-\log d\right\}  \notag \\
&=&\min_{s\in K}\left\{ G^{\ast }(-s)+\sup_{\rho \in \mathcal{M}}\left\{
\rho (sA)+h(\rho )\right\} -\log d\right\} .  \label{klit}
\end{eqnarray}%
By the definition of the entropy functional $h$ on (general) finite measures
(Remark \ref{hdf}), when looking for the supremum in \eqref{klit}, the
non-invariant measures can be disregarded. Moreover, for any fixed $\rho \in 
\mathcal{P}(T)$, the mapping 
\begin{equation*}
s\mapsto G^{\ast }(-s)+\rho (sA)+h(\rho )
\end{equation*}%
from the compact convex set $K$ to $\mathbb{R}$ is convex and lower
semicontinuous, while for any fixed $s\in K$, the mapping 
\begin{equation*}
\rho \mapsto G^{\ast }(-s)+\rho (sA)+h(\rho )
\end{equation*}%
from the weak$^{\ast }$ compact and convex set $\mathcal{P}(T)$ to $\mathbb{R%
}$ is concave and weak$^{\ast }$ upper semicontinuous (see Theorems 6.10,
8.1 and 8.2 in \cite{Walters}). Therefore, by Sion's min-max theorem, we
deduce from (\ref{klit}) that%
\begin{eqnarray}
\hat{P}(F) &=&\sup_{\rho \in \mathcal{M}}\left\{ h(\rho )-\log d+\min_{s\in
K}\left\{ s\rho (A)+G^{\ast }(-s)\right\} \right\}  \label{sdsdsdsd} \\
&=&\sup_{\rho \in \mathcal{M}}\left\{ h(\rho )-\log d-\sup_{s\in \mathbb{R}%
}\left\{ -s\rho (A)-G^{\ast }(-s)\right\} \right\}  \notag \\
&=&\sup_{\rho \in \mathcal{M}}\left\{ h(\rho )-\log d-\sup_{s\in \mathbb{R}%
}\left\{ s\rho (A)-G^{\ast }(s)\right\} \right\}  \notag \\
&=&\sup_{\rho \in \mathcal{M}}\left\{ h(\rho )-\log d-G(\rho (A))\right\} 
\notag \\
&=&\sup_{\rho \in \mathcal{M}}\left\{ F(\rho (A))+h(\rho )-\log d\right\} .
\label{eq pnl var}
\end{eqnarray}%
The functional $\mathfrak{p}:\mathcal{M}\rightarrow \mathbb{R}\cup \{\infty
\}$, defined again (cf. (\ref{eq press func convex})) by 
\begin{equation}
\mathfrak{p}(\rho ):=F(\rho (A))+h(\rho )-\log d\text{ },\text{\qquad }\rho
\in \mathcal{M},  \label{eq press func conc}
\end{equation}%
is the nonlinear pressure functional associated with the (now concave)
function $F$. As before, we call the maximizers of $\mathfrak{p}$ the
nonlinear equilibrium probabilities.

We show now that, as in the convex case, they are self-consistent linear
equilibrium probabilities. To this end, we show that nonlinear equilibrium
probabilities are directly related to saddle points of the functional%
\begin{equation*}
(\rho ,s)\mapsto h(\rho )-\log d+s\rho (A)+G^{\ast }(-s)
\end{equation*}%
on $\mathcal{M}\times \mathbb{R}$, which all belongs to $\mathcal{P}%
(T)\times \mathbb{R}$ (otherwise the functional takes infinite values):
Assuming again (\ref{rtytry}), not only Sion's min-max theorem but also von
Neumann's min-max theorem can be used to obtain (\ref{sdsdsdsd}), because $%
\mathcal{P}(T)$ is is not only convex but also weak$^{\ast }$ compact. In
particular, there is a saddle point $(\omega ,\bar{s})\in \mathcal{P}%
(T)\times \mathbb{R}$, that is,%
\begin{eqnarray*}
\hat{P}(F) &=&\sup_{\rho \in \mathcal{M}}\left\{ h(\rho )-\log d+F(\rho
(A))\right\} \\
&=&\sup_{\rho \in \mathcal{M}}\left\{ h(\rho )-\log d+\min_{s\in K}\left\{
s\rho (A)+G^{\ast }(-s)\right\} \right\} \\
&=&h(\omega )-\log d+\min_{s\in K}\left\{ s\omega (A)+G^{\ast }(-s)\right\}
\\
&=&G^{\ast }(-\bar{s})+\max_{\rho \in \mathcal{M}}\left\{ \rho (\bar{s}%
A)+h(\rho )-\log d\right\} \\
&=&h(\omega )-\log d+\bar{s}\omega (A)+G^{\ast }(-\bar{s}).
\end{eqnarray*}%
In particular, 
\begin{equation*}
\hat{P}(F)=h(\omega )-\log d+\min_{s\in K}\left\{ s\omega (A)+G^{\ast
}(-s)\right\} =h(\omega )-\log d+F(\omega (A))
\end{equation*}%
and $\omega \in \mathcal{P}(T)$ is thus a nonlinear equilibrium probability.
By Bogoliubov's variational principle and the above equalities, we find that%
\begin{eqnarray*}
\hat{P}(F) &=&G^{\ast }(-\bar{s})+\max_{\rho \in \mathcal{M}}\left\{ \rho (%
\bar{s}A)+h(\rho )\right\} -\log d \\
&=&G^{\ast }(-\bar{s})+P(\bar{s}A)-\log d\text{ },
\end{eqnarray*}%
that is, $\bar{s}\in K$ is a solution to Bogoliubov's variational problem (%
\ref{Bogoliubov's variational principle}). The equality%
\begin{equation*}
G^{\ast }(-\bar{s})+\max_{\rho \in \mathcal{M}}\left\{ \rho (\bar{s}%
A)+h(\rho )-\log d\right\} =h(\omega )-\log d+\bar{s}\omega (A)+G^{\ast }(-%
\bar{s})
\end{equation*}%
yields%
\begin{equation*}
h(\omega )-\log d+\omega (\bar{s}A)=\max_{\rho \in \mathcal{M}}\left\{ \rho (%
\bar{s}A)+h(\rho )-\log d\right\} ,
\end{equation*}%
that is, $\omega \in \mathcal{P}(T)$ is the linear equilibrium probability
for the potential $\bar{s}A$. In particular, if $G^{\ast }$ is assumed to be
differentiable (as above for $F^{\ast }$, in the convex case), it follows
that 
\begin{equation}
\varpi (A)=\left. \frac{d}{ds}P(sA)\right\vert _{s=\bar{s}}=-(G^{\ast
})^{\prime }(-\bar{s})\text{ }.  \label{sisi1}
\end{equation}

Assume again that $(G^{\ast })^{\prime }$ is injective, that is, $G^{\ast }$
is strictly convex, and denote by $\chi $ the inverse of $(G^{\ast
})^{\prime }$ on its image. Then,%
\begin{equation}
\left. \frac{d}{ds}P(sA)\right\vert _{s=-\chi (-\varpi (A))}=\varpi (A),
\label{usisi1}
\end{equation}%
which, similar to the convex case, is a self-consistency condition saying
that $\varpi $ is the equilibrium probability associated with the potential $%
-\chi (-\varpi (A))A$. In fact, if $G^{\ast }$ is strictly convex then the
solution $\bar{s}\in K$ to Bogoliubov's variational problem is unique. As
equilibrium probabilities for H\"{o}lder potentials are also unique, it
follows that the nonlinear equilibrium probability $\varpi \in \mathcal{M}$
is unique and $(\varpi ,\bar{s})$ is the unique saddle point of the mapping 
\begin{equation*}
(\rho ,s)\mapsto h(\rho )-\log d+s\rho (A)+G^{\ast }(-s)
\end{equation*}%
from $\mathcal{M}\times \mathbb{R}$ to $\mathbb{R}\cup \{\infty \}$.

This contrasts with the convex case, where the solution to Bogoliubov's
variational problem and the nonlinear equilibrium probability are generally
not unique. In other words, in the (strictly) concave case, there is never a
nonlinear phase transition in contrast to the convex case (see Section \ref%
{mae}).

\subsection{The quadratic pressure}

\label{qua}

In this subsection, we consider the particular case of the convex function $%
F(x)=\beta x^{2}/2$ for a fixed parameter $\beta >0$. In statistical
mechanics, $\beta $ is related to the inverse of temperature. The
expressions obtained here will eventually be used in the next sections to
give explicit examples of nonlinear equilibrium probabilities and nonlinear
phase transitions. The corresponding self-consistency condition is given
below in Equation \eqref{outr71}, which is a particular case of \eqref{ciir}
(see also \eqref{geger}). We will simply give the final expressions without
further details, as they can easily be obtained from the previous
subsections.

In other words, here we are interested in the following problem: Given an H%
\"{o}lder potential $A:\Omega \rightarrow \mathbb{R}$ determine the set of
invariant probabilities $\rho \in \mathcal{P}(T)$ maximizing the nonlinear
pressure functional $\mathfrak{p}$ discussed above (see (\ref{eq press func
convex}) and (\ref{eq press func conc})), in the quadratic case. It
corresponds to the study of the variational problem 
\begin{equation}
\mathfrak{P}_{2,\beta }(A):=\sup_{\rho \in \mathcal{P}(T)}\left\{ \frac{%
\beta }{2}\rho \left( A\right) ^{2}+h(\rho )-\log d\right\} ,  \label{impai}
\end{equation}%
where $h(\rho )$ is the (Kolmogorov-Sinai) entropy of $\rho $. Up to the
explicit constant $-\log d$, $\mathfrak{P}_{2,\beta }(A)$ is nothing else
but $\mathfrak{P}_{F,A}$ for $F(x)=\beta x^{2}/2$ with $\beta >0$. See also
Equations (\ref{Kbomm})--(\ref{Kibom}) and (\ref{klir}).

We point out that for the examples of Section \ref{mae} we will be able to
explicitly give the probabilities maximizing \eqref{impai}. In fact, it
turns out that, in some cases, they are independent and identically
distributed (i.i.d.) probabilities.

Like in Equation (\ref{feij502bis0}), for any H\"{o}lder potential $A:\Omega
\rightarrow \mathbb{R}$ and parameters $\beta >0$, $t\in \mathbb{R}$, we set%
\begin{equation}
\hat{c}_{\beta }(t):=P(\beta tA)-\log d=\lim_{n\rightarrow \infty }\hat{c}%
_{n}(\beta t),  \label{feij509}
\end{equation}%
where, for each $n\in \mathbb{N}$, $\hat{c}_{n}$ is defined by (\ref{feij50}%
), that is, 
\begin{equation*}
\hat{c}_{n}(t):={\frac{1}{n}}\log \int
e^{t\,(A(x)+A(T(x))+A(T^{2}(x))+\cdots +A(T^{n-1}(x))}\mu (\mathrm{d}x).
\end{equation*}%
If $F(x)=\beta x^{2}/2$ then $F^{\ast }(s)=s^{2}/(2\beta )$ and thus $%
(F^{\ast })^{\prime }(s)=\beta ^{-1}s$. It then follows from (\ref{ciir})
that, for any fixed $\beta >0$, the corresponding self-consistency equation
is%
\begin{equation}
\hat{c}_{\beta }^{\prime }(t)=\beta t=\beta \int A(x)\mu _{\beta tA}\left( 
\mathrm{d}x\right) ,  \label{outr71}
\end{equation}%
where $\mu _{\beta tA}$ is the unique linear equilibrium probability for the
H\"{o}lder potential $\beta tA$. This equation determines the possible
values $t$ for which the linear equilibrium probability for the (effective)
potential $\beta tA$ maximizes the quadratic pressure \eqref{impai}. In
other words, a linear equilibrium probability for $\beta tA$ with $t$
satisfying (\ref{outr71}) may be a nonlinear equilibrium probability for $A$
and $F(x)=\beta x^{2}/2$ ($\beta >0$), see Definition \ref{nonlinear
equilibrium probabilities}.

Taking $\hat{c}=\hat{c}_{\beta }$ in \eqref{feij68}, we define the large
deviation rate function $I_{\beta }$ (relative to the maximal entropy
measure $\mu $ and potential $\beta A$) by%
\begin{equation*}
I_{\beta }(x):=\sup_{t\in \mathbb{R}}\{tx-\hat{c}_{\beta }(t)\}=\sup_{t\in 
\mathbb{R}}\{tx-P(\beta tA)+\log d\},\qquad x\in \mathbb{R}.
\end{equation*}%
Again, if $F(x)=\beta x^{2}/2$ ($\beta >0$), then we deduce from Equations (%
\ref{binc78}), (\ref{fgh}) and (\ref{klir}) that 
\begin{eqnarray}
\mathfrak{P}_{2,\beta }(A) &=&\lim_{n\rightarrow \infty }\frac{1}{n}\ln
\left( \int e^{\frac{n}{2}\beta x^{2}}\mu _{n}^{A}(\mathrm{d}x)\right)
\label{foiq98} \\
&=&\lim_{n\rightarrow \infty }\frac{1}{n}\ln \left( \int e^{\frac{n}{2\beta }%
x^{2}}\mu _{n}^{\beta A}(\mathrm{d}x)\right) =\sup_{x\in \mathbb{R}}\left\{ 
\frac{x^{2}}{2\beta }-I_{\beta }(x)\right\} .  \notag
\end{eqnarray}%
Having in mind Equality (\ref{utq}) with a rescaling $s=\beta t$, we define
for this nonlinear pressure its (Bogoliubov) approximating pressure by 
\begin{equation}
P_{\beta ,A}(t):=-\,\frac{\beta }{2}t^{2}+P(\beta tA)-\log d=-\frac{\beta }{2%
}t^{2}+\hat{c}_{\beta }(t),\qquad t\in \mathbb{R}.  \label{quew}
\end{equation}%
Notice that the above expression corresponds to $\varphi _{OS}$ in \cite%
{LeWa1}\footnote{%
See, e.g., Theorem 1.3 (3) of \cite{LeWa1}.}. Observing that $\beta t^{2}/2$
is the Legendre transform of $x^{2}/2\beta $, Bogoliubov's variational
principle (as stated in Equation (\ref{utq})) yields the following theorem:

\begin{theorem}
\label{lgre}For any H\"{o}lder potential $A:\Omega \rightarrow \mathbb{R}$
and each parameter $\beta >0$, 
\begin{equation}
\mathfrak{P}_{2,\beta }(A)=\sup_{t\in \mathbb{R}}P_{\beta ,A}(t).
\label{naoac}
\end{equation}
\end{theorem}

Observe that the critical points of Bogoliubov's approximating pressure (\ref%
{quew}) are nothing else but the solutions to the self-consistency equation (%
\ref{outr}), also stated just above with Equation (\ref{outr71}). The
critical points $t_{0}$ of the functions $t\mapsto P_{\beta ,A}(t)$ may be
local maxima or minima, depending on the sign of $-\beta +\hat{c}_{\beta
}^{\prime \prime }(t_{0})$. From the results of Sections \ref{The convex
case} and \ref{qua}, the global minima $t_{0}$ are in one-to-one
correspondence to the \emph{nonlinear} equilibrium probabilities, which are
proven to be (self-consistent) \emph{linear} equilibrium probabilities for
the potentials $\beta t_{0}A$, satisfying (\ref{outr71}).

\begin{remark}
\label{jkd}A priori, the solution $t_{0}$ to (\ref{outr71}) does not have to
be unique. As the function $t\mapsto c^{\prime }(t)$ is real analytic, the
number of solutions $t_{0}$ is finite in finite intervals. In Figure \ref%
{figt12} (for $u>1$ and $\beta =1$), we give an example with the existence
of two points $t_{0}\neq 0$ and $-t_{0}$ that satisfy the self-consistency
condition (\ref{outr71}). This example refers to the case $d=2$, i.e., $%
\Omega =\{-1,1\}^{\mathbb{N}}$.
\end{remark}

Regarding expression \eqref{quew}, we will need a lemma (the analogous of
Lemma 3.1 in \cite{LeWa2}) to be used later in Section \ref{conj}.

\begin{lemma}
\label{qte} For every $\beta >0$, there exist two constants $R,B>0$ such
that, for all $t>B$,%
\begin{equation*}
P_{\beta ,A}(t)=-\frac{\beta }{2}t^{2}+\hat{c}_{\beta }(t)<R-\frac{\beta }{4}%
t^{2},\qquad t\in \mathbb{R},
\end{equation*}%
and $P_{\beta ,A}$ is maximized for critical points inside the interval $%
[-B,B]$.
\end{lemma}

\begin{proof}
Note that $P_{\beta ,A}(0)=0$ and $P(\beta tA)$ does not grow faster than
linearly in $t$, since $-\beta \Vert A\Vert \leq t^{-1}P(t\,\beta A)\leq
\beta \Vert A\Vert $ (see \cite{BLL}, or combine Theorem \ref{rer1} with
Proposition 124 of Section 6.1 in \cite{LTF}). By (\ref{quew}), the
assertion follows.
\end{proof}

The set of all solutions $t_{0}$ to the self-consistency equation (\ref{outr}%
) or (\ref{outr71}) is denoted by $S_{0}$. In the examples of Section \ref%
{mae}, $S_{0}$ has one or three points, depending on the parameter $\beta $
and the potential $A$. In fact, by symmetry, $t_{0}=0$ is always a solution
to \eqref{outr} or (\ref{outr71}), but, in general, it is not a global
minimum of Bogoliubov's approximating pressure (\ref{quew}) and, thus, does
not yield a nonlinear equilibrium measure.

If we consider the unique linear equilibrium probability $\mu _{f}$ for a
fixed H\"{o}lder potential $f$ instead of the maximum entropy probability $%
\mu $ to define the measures $\mu _{n}$ in (\ref{foiq98}) and, consequently,
a more general nonlinear pressure $\mathfrak{P}_{2,\beta }(A)$, then the
corresponding Bogoliubov approximating pressure is 
\begin{equation}
P_{\beta ,A}(t):=-\frac{\beta }{2}t^{2}+P(f+t\beta A),\qquad t\in \mathbb{R}.
\label{joq2}
\end{equation}%
See, e.g., (\ref{joq1}). In particular, in this case, the self-consistency
equation, which refers to the critical points of this new approximating
pressure, is%
\begin{equation}
\left. \frac{dP(f+\beta tA)}{dt}\right\vert _{t=t_{0}}=\beta t_{0}=\beta
\int A(x)\mu _{f+\beta t_{0}A}\left( \mathrm{d}x\right) .  \label{pirt}
\end{equation}%
Moreover, one can also show that, in this more general case, the nonlinear
pressure satisfies a variational principle for invariant probabilities: 
\begin{equation}
\mathfrak{P}_{2,\beta }(A)=\sup_{\rho \in \mathcal{P}(T)}\left\{ \,\frac{%
\beta }{2}\rho (A)^{2}+h(\rho )+\rho (f)-P(f)\right\} .  \label{kjew}
\end{equation}%
See (\ref{bel1}). Note that when $f=-\log d$, i.e., $\mu _{f}=\mu $ is the
maximum entropy probability, one has $P(-\log d)=0$ and the previous special
case is recovered, that is, $\mathfrak{P}_{2,\beta }(A)$ corresponds to
Equation (\ref{impai}), as expected.

Similarly to the convex case discussed above, it is possible to obtain a
version of Theorem \ref{lgre} for the concave case, i.e., for $F(x)=-\beta
x^{2}/2$ with $\beta >0$. In this case, Bogoliubov's approximating pressure
is 
\begin{equation*}
P_{\beta ,A}(t):=\frac{\beta }{2}t^{2}+P(t\beta A)-\log d=\frac{\beta }{2}%
t^{2}+\hat{c}_{\beta }(t),\qquad t\in \mathbb{R},
\end{equation*}%
(that is, as compared to (\ref{quew}), the sign of the quadratic term
changes) and one has 
\begin{equation*}
\mathfrak{P}_{2,\beta }(\,A)=\inf_{t\in \mathbb{R}}P_{\beta ,A}(t).
\end{equation*}%
(That is, the $\sup $ of (\ref{naoac}) has to be replaced with an $\inf $.)
We omit the details, as they have already been explained for general concave
functions in Section \ref{conc}. In fact, recall that, as discussed above,
in the strictly concave case, the self-consistency equation has only one
solution, which implies that no nonlinear phase transition occurs. This case
is therefore less interesting than the convex one.

\subsection{The mean-field free energy functional and Bogoliubov's
approximation}

\label{Bog}

We proved above that if $F$ is a convex or a concave function, then, for a
fixed H\"{o}lder potential $A$, the associated nonlinear pressure satisfies
the following identity:%
\begin{equation}
\mathfrak{P}_{F,A}-\log d=\hat{P}(F)=\max_{\rho \in \mathcal{P}(T)}\left\{
F(\rho (A))+h(\rho )-\log d\right\} .  \label{hjk}
\end{equation}%
Observe that the solutions to this variational problem are precisely the
nonlinear equilibrium probabilities for $F$ and $A$ of Definition \ref%
{nonlinear equilibrium probabilities}. Remark that the above functional is 
\textbf{not} affine with respect to $\rho $. The (Kolmogorov-Sinai) entropy $%
h(\rho )$ is affine but the energetic part $F(\rho (A))$ generally not.) In
this subsection, we show that the original nonlinear pressure functional can
be replaced with an affine one.

In fact, later on in Section \ref{QLVE}, we show that the new maximizers are
not necessarily nonlinear equilibrium probabilities in the previous sense,
but are always in the closed convex hull of the set of these probabilities.
We then use in Section \ref{QLVE} this property as a step to prove that
mean-field equilibrium probabilities also have these properties, i.e., they
are always in the closed convex hull of the set of all nonlinear equilibrium
probabilities.

For technical simplicity, we again consider the quadratic case, that is, $%
F(x)=\pm \beta x^{2}/2$, with $\beta >0$ being fixed once and for all.
Moreover, this is the case considered here for explicit examples. In fact,
the result can be extended to the general convex and concave cases by using
arguments based on the Legendre-Fenchel transform, similar to what is done
in previous subsections. We refrain from working out all details of such a
more general setting and focus on the main arguments of the proof, which are
more clearly understood in the quadratic case. In fact, we will devote an
entire article \cite{BPL} to explaining the results in a very general
framework.

For any fixed continuous (more generally, bounded Borel-measurable)
potential $A:\Omega \rightarrow \mathbb{R}$ define the affine functional $%
\Delta _{A}:\mathcal{P}(T)\rightarrow \mathbb{R}$ by%
\begin{equation}
\Delta _{A}(\rho ):=\lim_{n\rightarrow \infty }\frac{1}{2}\int A_{n}\left(
x\right) ^{2}\rho \left( \mathrm{d}x\right) ,  \label{Delta0}
\end{equation}%
where $A_{n}$, $n\in \mathbb{N}$, are the Birkhoff averages defined by
Equation (\ref{Eq Birk Av}). This functional is Borel-measurable with
respect to the weak$^{\ast }$ topology, being the pointwise limit of a
sequence of continuous functionals. In fact, one can show that 
\begin{equation}
\Delta _{A}(\rho )=\inf_{n\in \mathbb{N}}\frac{1}{2}\int A_{n}\left(
x\right) ^{2}\rho \left( \mathrm{d}x\right)  \label{Delta1}
\end{equation}%
and $\Delta _{A}$ is thus even weak$^{\ast }$ upper semicontinous. Note
additionally that 
\begin{equation}
\Delta _{A}(\rho )\geq \frac{\rho (A)^{2}}{2}  \label{eq delta 1}
\end{equation}%
for all $\rho \in \mathcal{P}(T)$, with equality when the invariant measure $%
\rho $ is ergodic, i.e., extremal in the weak$^{\ast }$ compact convex space 
$\mathcal{P}(T)$ of $T$-invariant probabilities. For more details, see \cite%
{BPL}.

For fixed $\beta >0$, define the mean-field free energy functional $%
\mathfrak{f}:\mathcal{P}(T)\rightarrow \mathbb{R}$ by 
\begin{equation}
\mathfrak{f}^{\pm }(\rho ):=-\left( \pm \beta \Delta _{A}(\rho )+h(\rho
)-\log d\right) ,\text{\qquad }\rho \in \mathcal{P}(T),  \label{faro}
\end{equation}%
where we recall once again that $h$ is the (Kolmogorov-Sinai) entropy.

We define the \textquotedblleft nonlinear free energy
functional\textquotedblright\ $\mathfrak{g}^{\pm }:\mathcal{P}(T)\rightarrow 
\mathbb{R}$ by 
\begin{equation}
\mathfrak{g}^{\pm }(\rho ):=-\left( \pm \frac{\beta }{2}\rho (A)^{2}+h(\rho
)-\log d\right) ,\text{\qquad }\rho \in \mathcal{P}(T).  \label{farobis}
\end{equation}%
Note that $\mathfrak{g}^{\pm }$\ is nothing else but minus the nonlinear
pressure functional $\mathfrak{p}$ discussed above (see Equations (\ref{eq
press func convex}) and (\ref{eq press func conc})), in the quadratic case.
That is why it is named (nonlinear) free energy functional.

Now we derive the affine variational principle for the nonlinear pressure,
which is based on the functional $\Delta _{A}:\mathcal{P}(T)\rightarrow 
\mathbb{R}$ defined above:

\begin{theorem}
\label{erte} Let $F(x)=\pm \beta x^{2}/2$ with $\beta >0$. Then, 
\begin{equation}
-\inf \mathfrak{f}^{\pm }(\mathcal{P}(T))=-\inf \mathfrak{g}^{\pm }(\mathcal{%
P}(T))=\hat{P}(F).  \label{bon}
\end{equation}
\end{theorem}

\begin{proof}
From the results of the previous subsection (see (\ref{hjk})), note that 
\begin{equation*}
\inf \mathfrak{g}^{\pm }(\mathcal{P}(T))=-\hat{P}(F).
\end{equation*}%
Thus, one has to prove that 
\begin{equation*}
\inf \mathfrak{f}^{\pm }(\mathcal{P}(T))=\inf \mathfrak{g}^{\pm }(\mathcal{P}%
(T)).
\end{equation*}%
From Inequality (\ref{eq delta 1}), $\mathfrak{g}^{+}\geq \mathfrak{f}^{+}$
and $\mathfrak{g}^{-}\leq \mathfrak{f}^{-}$, which trivially yield 
\begin{equation}
\inf \mathfrak{g}^{+}(\mathcal{P}(T))\geq \inf \mathfrak{f}^{+}(\mathcal{P}%
(T))\text{ \ \ and \ \ }\inf \mathfrak{g}^{-}(\mathcal{P}(T))\leq \inf 
\mathfrak{f}^{-}(\mathcal{P}(T)).  \label{eq ineq f e g}
\end{equation}%
Observe further that $\mathfrak{g}^{\pm }(\rho )=\mathfrak{f}^{\pm }(\rho )$
when $\rho $ is ergodic (see \cite{BPL} for more details). As $\mathfrak{f}%
^{+}$ is weak$^{\ast }$ lower semicontinuous and affine, its set of
minimizers is a nonempty compact face of $\mathcal{P}(T)$. In particular, $%
\mathfrak{f}^{+}$ is minimized by some ergodic probability. (Recall once
again that the extreme points of the convex set $\mathcal{P}(T)$\ of $T$%
-invariant probabilities are precisely the ergodic probabilities.) Hence,
from (\ref{eq ineq f e g}), we get the equality%
\begin{equation*}
\inf \mathfrak{g}^{+}(\mathcal{P}(T))=\inf \mathfrak{f}^{+}(\mathcal{P}(T)).
\end{equation*}%
Noting that the mapping $\rho \mapsto \rho (A)^{2}$ (appearing in the
definition of $\mathfrak{g}^{\pm }$) is weak$^{\ast }$\ continuous and the
entropy functional $\rho \mapsto h(\rho )$ is \textquotedblleft
pseudocontinuous\textquotedblright\ along ergodic measures, that is, for all 
$\rho \in \mathcal{P}(T)$ (not necessarily ergodic) there is a sequence $%
(\rho _{n})_{n\in \mathbb{N}}$ of ergodic measures converging to $\rho $ in
the weak$^{\ast }$ topology (see \cite{Kif} or \cite{L3}), such that 
\begin{equation*}
h(\rho )=\lim_{n\rightarrow \infty }h(\rho _{n}),
\end{equation*}%
we conclude again from (\ref{eq ineq f e g}) that 
\begin{equation*}
\inf \mathfrak{g}^{-}(\mathcal{P}(T))=\inf \mathfrak{f}^{-}(\mathcal{P}(T)).
\end{equation*}%
See again \cite{BPL} for all details.
\end{proof}

\section{Explicit examples of nonlinear phase transitions}

\label{mae}

In this section, we will present explicit examples that illustrate some
facts considered in Section \ref{Bogo}, in particular Subsection \ref{qua}.
Throughout this section, we only consider the case $d=2$, which, for
convenience, is identified with $\Omega =\{-1,1\}^{\mathbb{N}}$. Various
results we summarize below are taken from \cite{CDLS}, where explicit
expressions were obtained in the linear case for a certain potential $A$
that depends on infinite coordinates in the symbolic space $\Omega $. Using
these previous results, we will be able to obtain explicit expressions,
yielding examples of quadratic phase transitions\footnote{%
I.e., a nonlinear phase transition for $F(x)=\beta x^{2}/2$ with $\beta >0$.
See Definition \ref{phase transition}.} for a potential $A$ of the form %
\eqref{binc2x} below and of H\"{o}lder class.

Thus, here we are interested in determining explicitly the maximizers of 
\begin{equation}
\mathfrak{P}_{F,A}=\sup_{\rho \in \mathcal{P}(T)}\left\{ F(\rho (A))+h(\rho
)\right\}  \label{Kbombom}
\end{equation}%
for $\Omega =\{-1,1\}^{\mathbb{N}}$, a quadratic function $F(x)=\beta
x^{2}/2 $ for some parameter $\beta >0$, and examples of H\"{o}lder
potentials $A:\Omega \rightarrow \mathbb{R}$.

\begin{remark}
In our examples, we always have $\Omega =\{-1,1\}^{\mathbb{N}}$, but we
could have considered the $XY$ model for which the symbolic space is $%
[-1,1]^{\mathbb{N}}$, $[-1,1]$ being now the closed interval in $\mathbb{R}$%
. The dynamics is given by the shift and similar results as in Section \ref%
{Examples} below can be obtained for the product type potential described in
Section 1 of \cite{Mohr}.
\end{remark}

\subsection{Examples inspired by (anti)ferromagnetic systems \label{Examples}%
}

The H\"{o}lder potential $A:\Omega \rightarrow \mathbb{R}$ we have in mind
here for an explicit study of (\ref{Kbombom}) are defined as follows:\ Given
an absolutely convergent series $\sum_{n}a_{n}$ and two real parameters $%
J,h\in \mathbb{R}$, consider the continuous potential $A_{J,h}:\{-1,1\}^{%
\mathbb{N}}\rightarrow \mathbb{R}$ defined by 
\begin{equation}
A_{J,h}(x)=\frac{J}{2}\sum_{n=1}^{\infty }a_{n}x_{n}+hx_{1},  \label{odd1}
\end{equation}%
where $x=(x_{n})_{n\in \mathbb{N}}\in \{-1,1\}^{\mathbb{N}}$. We assume that 
$A=A_{J,h}$ is a H\"{o}lder potential. For example, this is the case when $%
a_{n}$ decays exponentially to zero, as $n\rightarrow \infty $.

In statistical mechanics, $A_{J,h}$ plays the role of minus the Hamiltonian.
For this reason, the cases $J>0$ and $J<0$ are called ferromagnetic and
antiferromagnetic, respectively. $J$ is the \textit{strength} of the
interaction. As before, the parameter $\beta $ is related to the inverse
temperature, whereas $h\in \mathbb{R}$ represents an external magnetic
field. Our main focus in this section is the case where $h=0$. In fact, note
that for formal computations, as done below, the prefactor $J/2$ could just
be incorporated in $a_{n}$ in \eqref{odd1}.

The following quantities%
\begin{equation}
s_{J,h}:=\sup_{x=(x_{1},x_{2},\ldots )\in \{-1,1\}^{\mathbb{N}}}\left\vert 
\frac{J}{2}\sum_{n=1}^{\infty }a_{n}x_{n}+hx_{1}\right\vert <\infty
\label{odd77}
\end{equation}%
and 
\begin{equation}
u_{J,h}:=\,h+\frac{J}{2}\sum_{n=1}^{\infty }a_{n}  \label{oro78}
\end{equation}%
play an important role in the properties of the potential $A_{J,h}$. To
simplify our expressions, we will also use the notation $u:=u_{2,0}$, which
is nothing else but the sum $\sum_{n}a_{n}$.

Since $A=A_{J,h}$ is by assumption a H\"{o}lder potential, using of course
the maximal entropy probability $\mu $ on $\{-1,1\}^{\mathbb{N}}$, we infer
from (\ref{feij502bis0}) that 
\begin{equation}
\hat{c}(t)=P(tA_{J,h})-\log 2,\qquad t\in \mathbb{R}.  \label{oro77}
\end{equation}%
Note that, for any $t\in \mathbb{R}$, 
\begin{equation*}
P(tA_{J,h})-\log 2=P(tA_{J,h}-\log 2)=\log \lambda _{tA_{J,h}-\log 2}\ ,
\end{equation*}%
thanks to the Ruelle(-Perron-Frobenius) theorem (Theorem \ref{rer1}).

It follows from Theorems 4.1, 3.1, and Corollary 3.2 in \cite{CDLS} (see
also Example 13 in Section 3.2 in \cite{LTF}) that the main eigenvalue of
the Ruelle operator for the potential $tA_{J,h}$, $t\in \mathbb{R}$, is $%
2\cosh (t\beta u_{J,h})$. Thus, taking $h=0$ as a particular case, for any $%
t\in \mathbb{R}$, the linear pressure for $tA_{J,0}$ is equal to%
\begin{equation*}
P(tA_{J,0})=\log (2\cosh (tu_{J,0}))
\end{equation*}%
(see again Theorem \ref{rer1})\ and, hence, 
\begin{equation}
\hat{c}(t)=\hat{c}_{A_{J,0}}(t)=\log (2\cosh (tu_{J,0})-\log 2=\log (\cosh
(tu_{J,0})).  \label{binc2}
\end{equation}%
Note that the pressure $P(t)=P(tA_{J,h})$ is invariant under the reflection $%
t\mapsto -t$. (In particular, Remark \ref{sisim} applies.)

To simplify our example even further, from now on we take a H\"{o}lder
potential of the form%
\begin{equation}
A(x)=\sum_{n=1}^{\infty }a_{n}x_{n}=A_{2,0}\,,  \label{binc2x}
\end{equation}%
where $x=(x_{n})_{n\in \mathbb{N}}\in \{-1,1\}^{\mathbb{N}}$. As the
potential $A$ is by assumption H\"{o}lder, no \emph{linear} phase transition
occurs, i.e., the linear equilibrium probability for the linear pressure for 
$A$ is unique. Here, we are interested in finding probabilities $\rho $ in $%
\{-1,1\}^{\mathbb{N}}$ maximizing the nonlinear pressure \eqref{Kbombom} in
the simplified case given by (\ref{binc2x}).

In this simplified case, for any $t\in \mathbb{R}$, 
\begin{equation}
p(t):=P(tA)=\log (2\cosh (tu)),  \label{p}
\end{equation}%
where we recall that $u:=u_{2,0}$. Note again that $p$ is an even function,
i.e., $p(t)=p(-t)$, and elementary computations yield%
\begin{eqnarray}
p^{\prime }(t) &=&u\tanh (tu)  \label{dePa} \\
p^{\prime \prime }(t) &=&u^{2}\mathrm{sech}^{2}(tu)  \label{dePax}
\end{eqnarray}%
for any $t\in \mathbb{R}$.\ In particular, for all $t\in \mathbb{R}$, $%
p^{\prime }(t)=-p^{\prime }(-t)$ and $p^{\prime \prime }(t)=p^{\prime \prime
}(-t)$. One then gets an explicit expression for the associate (large
deviation) rate function $I_{A}$, as defined by (\ref{feij68}): for any $%
x\in (-u,u)$, 
\begin{eqnarray}
I_{A}(x) &=&\sup_{t\in \mathbb{R}}\left\{ xt-p(t)+\log 2\right\} =\sup_{t\in 
\mathbb{R}}\left\{ xt-\log \left( \cosh \left( tu\right) \right) \right\} 
\notag \\
&=&xu^{-1}\tanh ^{-1}\left( xu^{-1}\right) -\log \left( \cosh \left( \tanh
^{-1}\left( xu^{-1}\right) \right) \right) ,  \label{xili}
\end{eqnarray}%
while $I_{A}(x)=\infty $ when $|x|\geq u$, thanks to (\ref{p})--(\ref{dePax}%
).

We give numerical computations now. Given $u\in \mathbb{R}$, the existence
of two points $t_{0}\neq 0$ and $-t_{0}$ satisfying the self-consistency
condition (\ref{outr71}) for $\beta =1$ is obtained in this particular case
by solving the equation 
\begin{equation*}
R_{u}(t):=u\tanh (tu)=p^{\prime }(t)=t,\text{\qquad }t\in \mathbb{R}.
\end{equation*}%
See Equation \eqref{dePa}. In Figures \ref{figt1}, \ref{figt12} and \ref%
{figt08}, we plot $R_{u}$ and the identity function $t\mapsto t$. Their
intersections thus determine the self-consistent points for the quadratic
case $F(x)=x^{2}/2$.

\begin{remark}
\label{poiu}Figure \ref{figt12} (for $u>1$ and $\beta =1$) shows the
existence of two points $t_{0}\neq 0$ and $-t_{0}$ satisfying the
self-consistency condition (\ref{outr71}). At fixed $\beta >0$ (cf. Section %
\ref{qua}), these two points $t_{1},t_{2}$ are pivotal in Section \ref{conj}
and also determine the value $P(\beta t_{1}A)=P(-\beta t_{2}A)$, which are
important in Section \ref{Til}. When $u<1$ and $\beta =1$, note that no
nonlinear phase transition occurs. See, e.g., Figure \ref{figt08}.
\end{remark}

\begin{figure}[h!]
\centering
\includegraphics[scale=0.8,angle=0]{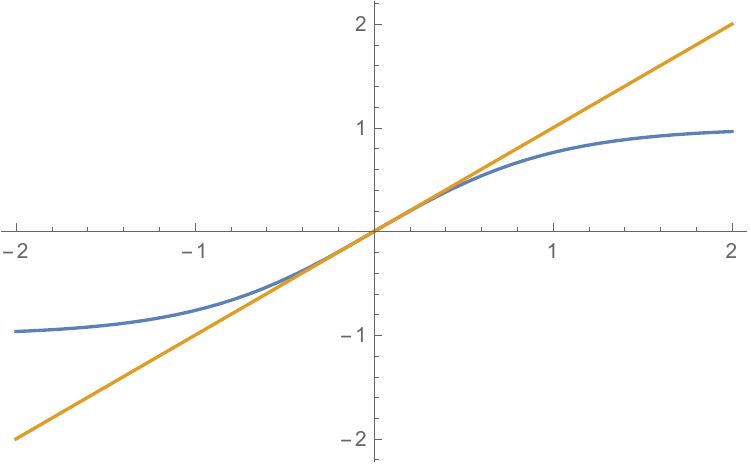}
\caption{In blue is the graph of $t \to R_1(t) $ and in yellow is the graph
of the identity. The two graphs intersect just at $t=0$; the only case which
would correspond to $p^{\prime \prime}(0)=1$. No non-zero point satisfying $%
p^{\prime }(t)=t$ when $u=1$.}
\label{figt1}
\end{figure}

\begin{figure}[h!]
\centering
\includegraphics[scale=0.8,angle=0]{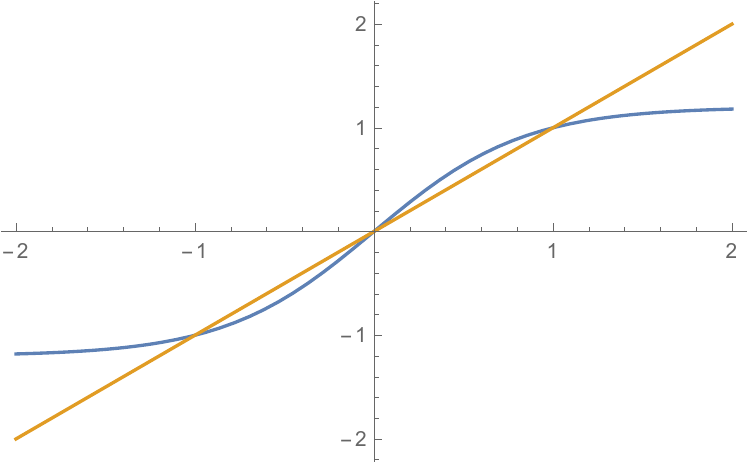}
\caption{In blue the graph of $t \to R_{1.2}(t) $ and in yellow the graph of
the identity. Excluding $t=0$, we get two other symmetric solutions $t_{0}$
and $-t_{0}$ of the equation $t = R_{1.2}(t)$, when $u=1.2>1$.}
\label{figt12}
\end{figure}

\begin{figure}[h]
\centering
\includegraphics[scale=0.8,angle=0]{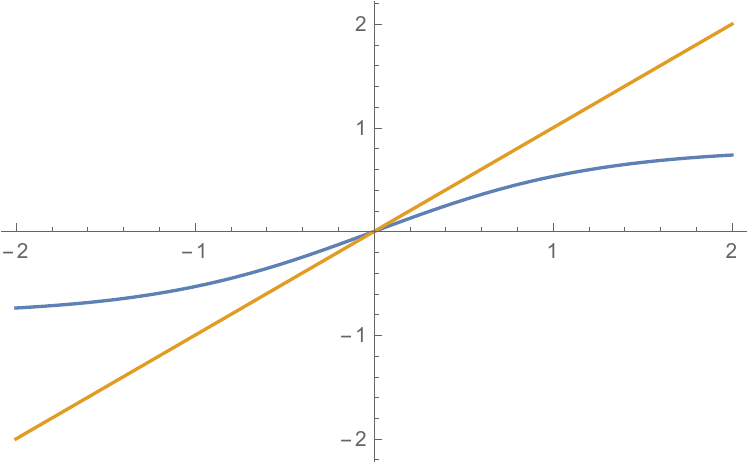}
\caption{In blue the graph of $t\rightarrow R_{0.8}(t)$ and in yellow the
graph of the identity. The two graphs intersect just at $t=0.$ This
corresponds to the case where $u<1$, here when $u=0.8$.}
\label{figt08}
\end{figure}

According to Theorem 4.1 in \cite{CDLS} (see also computations in Example 13
in Section 3.2 of \cite{LTF}), for the potential $A$ given by \eqref{binc2x}%
, the eigenfunction $\psi _{tA}$ associated with the main eigenvalue $%
\lambda _{tA}=2\cosh (\beta tu)$ of the Ruelle operator $\mathcal{L}_{tA}$
for any $t\in \mathbb{R}$ is explicitly given by%
\begin{equation}
\psi _{tA}(x)=\exp \left( t\sum_{n=1}^{\infty }\alpha _{n}x_{n}\right) ,
\label{eei}
\end{equation}%
for any $x=(x_{n})_{n\in \mathbb{N}}\in \{-1,1\}^{\mathbb{N}}$, where, for
any $n\in \mathbb{N}$,%
\begin{equation*}
\alpha _{n}:=\sum_{k=n+1}^{\infty }a_{k}=u-\sum_{k=1}^{n}a_{k}<\infty .
\end{equation*}%
We assume here that $\sum_{n}\alpha _{n}$ absolutely converges, which is
always the case when, for instance, $a_{n}$ tends exponentially fast to zero.

Furthermore, the eigenprobability $\nu _{tA}$ associated to the adjoint
operator $\mathcal{L}_{tA}^{\ast }$ is a product of independent (but not
i.i.d.) distributions. More precisely, it has the form 
\begin{equation}
\nu _{tA}=\prod_{n\in \mathbb{N}}\nu _{n},  \label{hgfr}
\end{equation}%
where $\nu _{n}$ is the probability distribution over $\{-1,1\}$ given by 
\begin{equation}
\nu _{n}(\{1\})=\frac{\exp (t\sum_{k=1}^{n}a_{k})}{2\cosh
(t\sum_{k=1}^{n}a_{k})}\ \ \ \text{and \ \ }\nu _{n}(\{-1\})=\frac{\exp
(-t\sum_{k=1}^{n}a_{k})}{2\cosh (t\sum_{k=1}^{n}a_{k})}.  \label{viu1}
\end{equation}%
Observe that $\nu _{tA}$ is not $T$-invariant. However, as $\psi _{tA}$ and $%
\nu _{tA}$ are explicitly known, one can get the exact equilibrium
probability $\rho _{tA}$ for the (effective) linear problem: By (\ref{eei}%
)--(\ref{viu1}) and Theorem \ref{rer1}, the equilibrium probability $\rho
_{tA}$ for the potential $tA$ defined by (\ref{binc2x}) is the i.i.d.
probability on $\{-1,1\}^{\mathbb{N}}$, with weights 
\begin{equation}
p_{\pm 1,t}:=\frac{e^{\pm tu}}{e^{tu}+e^{-tu}}=\rho _{tA}\left( \{\left( \pm
1,0,\ldots \right) \}\right) .  \label{espo1}
\end{equation}%
See again \cite[Section 6]{CDLS}. For the quadratic case, the solutions of
the nonlinear pressure of $A$ are linear equilibrium probabilities for
potentials of the form $tA$, for some value of $t$. In this case, we get
that the two solutions we are looking for the nonlinear pressure problem %
\eqref{Kbombom} are different i.i.d. probabilities of the form \eqref{espo1}.

\begin{remark}
\label{yyx}\label{yyx1}Let $t_{1},t_{2}\in \mathbb{R}$ with $t_{1}\neq t_{2}$%
. Then, the eigenprobabilities for the two potentials 
\begin{equation*}
-\log 2+\beta t_{1}\sum_{n=1}^{\infty }a_{n}x_{n}\ \ \ \text{and \ \ }-\log
2+\beta t_{2}\sum_{n=1}^{\infty }a_{n}x_{n},
\end{equation*}%
are different from each other. This will be important in Section \ref{conj}.
Thus, by (\ref{espo1}), the equilibrium probabilities $\rho _{t_{1}A}$ and $%
\rho _{t_{2}A}$ are also different, and if $t_{1},t_{2}$ satisfy the
self-consistency condition (\ref{outr71}) (cf. Remark \ref{poiu}), then a
nonlinear Gibbs phase transition occurs for the quadratic pressure problem %
\eqref{Kbombom}.
\end{remark}

\subsection{Example based on the generalized Curie-Weiss model \label{boex}}

We gather now some results on an example related to Section 2.1 of \cite%
{LeWa1}. More precisely, it refers to the limit of the family of
probabilities given by Equation (12) in \cite{LeWa1}, a topic to be
discussed in Section \ref{conj}.

Consider the potential $A:\{-1,1\}^{\mathbb{N}}\rightarrow \mathbb{R}$
defined by 
\begin{equation}
A=3I_{\overline{-1,-1}}-5I_{\overline{-1,1}}+I_{\overline{1,1}}+2I_{%
\overline{1,-1}},  \label{go}
\end{equation}%
where, for any $a,b\in \{-1,1\}$, $I_{\overline{a,b}}$ denotes the
characteristic function of the cylinder set 
\begin{equation*}
\{x=(x_{n})_{n\in \mathbb{N}}\in \{-1,1\}^{\mathbb{N}}\text{ }|\text{ }%
x_{1}=a,\text{ }x_{2}=b\}.
\end{equation*}%
One interesting aspect of this example is that it breaks the symmetry $%
P(-tA)=P(tA)$ of the linear pressure, which was satisfied in the example
given in Section \ref{Examples}.

Taking $F(x)=\beta x^{2}/2$ with $\beta >0$, we will show the possibility of
obtaining more than one self-consistent point, i.e., at least two different
parameters $t_{1},t_{2}$ satisfies the self-consistency condition (\ref%
{outr71}). But more importantly, our explicit results obtained in the
present example can illustrate some issues related to the results of Section
2.1 of \cite{LeWa1}.

To compute things explicitly, we take, for instance, $\beta =0.6$. Through
simple computations, we obtain that the pressure is 
\begin{equation*}
P(t\beta A)=\log \left( \frac{1}{2}e^{-3t}\left( e^{3.6t}+e^{4.8t}+e^{2.1t}%
\sqrt{4+e^{3t}-2e^{4.2t}+e^{5.4t}}\right) \right) .
\end{equation*}%
For this particular choice, we get two self-consistent points $t_{1}\simeq
-1 $ and $t_{2}=3$, see Figure \ref{fign4}. Note that the second
self-consistent point $t_{2}=3$ is mathematically exact (unlike $t_{1}\simeq
-1$).

Having in mind the (Bogoliubov) approximating pressure of Equation (\ref%
{quew}) and Theorem \ref{lgre}, the function 
\begin{equation}
t\,\rightarrow \varphi (t):=P_{\beta ,A}(t)=-\frac{\beta }{2}t^{2}+P(t\beta
A)-\log 2  \label{cory}
\end{equation}%
defined for any $t\in \mathbb{R}$ has local maxima at these points $%
t_{1},t_{2}$ (see Figure \ref{fign4}). However, $\varphi (t_{1})< \varphi
(t_{2}).$ Additionally, there is a local minimum at $t\simeq -0.155$. Note
that $\varphi^{\prime \prime} (t_1)\sim \varphi^{\prime \prime} (t_2).$

The above function $\varphi $ (up to the constant $-\log 2$) is denoted $%
\varphi _{OS}$ in \cite{LeWa1}, see in particular Theorem 1.3 in \cite{LeWa1}%
. Determining the maximum value of $\varphi (t)$ is an important issue in
estimating the limit of the probabilities $(\mu _{n,\beta })_{n\in \mathbb{N}%
}$ described by Equation (12) in \cite{LeWa1}. Indeed, to estimate the limit 
$n\rightarrow \infty $ of the quantity $\mu _{n,\beta }([\omega ])$ of
Equation (20) of \cite{LeWa1} for the proof of Theorem 1.3 of \cite{LeWa1},
the authors use the Laplace method. In our example, as $\varphi (t_{1})\neq
\varphi (t_{2})$ (see Figure \ref{fign4}), we get from \cite{LeWa1} that the
corresponding limit probability will be a unique eigenprobability and not
just a non-trivial convex combination of those mentioned in Theorem 1.3 of 
\cite{LeWa1}. This is, in particular, an important issue regarding the
expression (26) of \cite{LeWa1}.

\begin{figure}[h]
\centering
\includegraphics[scale=0.8,angle=0]{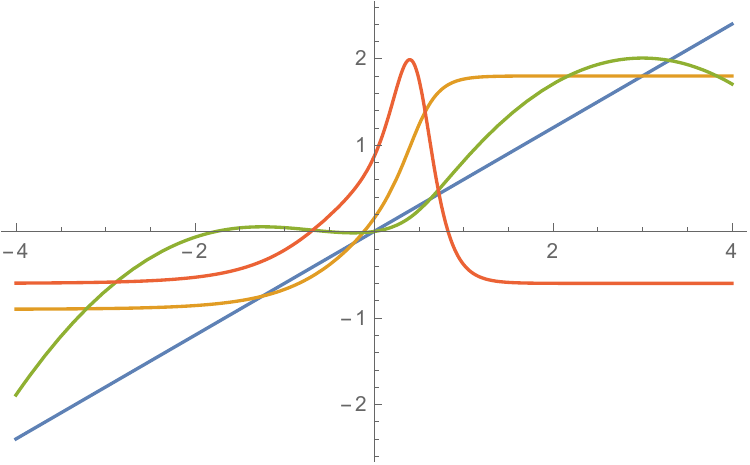}
\caption{For $\protect\beta =0.6$ and for $A$ as in \eqref{go} we show: the
blue line is the graph of $t\rightarrow \protect\beta t$, the yellow curve
is the graph of $t\rightarrow \frac{d}{dt}P(t\protect\beta A)$, the green
curve is the graph of $t\rightarrow \protect\varphi (t)$ and the red curve
is the graph of $t\rightarrow \protect\varphi ^{\prime \prime }(t)$. The
value $t=3$ gives an exact parameter where the self-consistency condition is
true.}
\label{fign4}
\end{figure}

\section{Quadratic mean-field Gibbs probabilities\label{conj}}

In this section, we are interested in the weak$^{\ast }$ limit of the
measure \eqref{xiux23}, i.e., in quadratic mean-field Gibbs probabilities
(Definition \ref{loloy}) and quadratic mean-field Gibbs phase transition
(Definitions \ref{gog1}--\ref{gogrexx}). The results in this section are
related to Theorem 1.3 (4) of the paper \cite{LeWa1} on the generalized
Curie-Weiss model (or Theorem 3 in \cite{LeWa2}). Our main goal is to
present an explicit example of the existence of a nonlinear phase
transition, rather than merely demonstrating the possibility of its
occurrence.

\subsection{Linear mean-field Gibbs probabilities}

Recall that $\Omega :=\{1,2,\ldots ,d\}^{\mathbb{N}}$ with $d\in \mathbb{N}$%
. Consider a linear equilibrium probability $\mu _{f}$ for the H\"{o}lder
(continuous) potential $f:\Omega \rightarrow \mathbb{R}$, see Definition \ref%
{equilibrium probability linear}. We will assume that $f$ is normalized,
that is, $\mathcal{L}_{f}(1)=1$. Let $g:\Omega \rightarrow \mathbb{R}$ be a
second H\"{o}lder function. Given $n\in \mathbb{N}$, we then define the
probability measure $\mathrm{m}_{n}=\mathrm{m}_{n,f,g}$ on $\Omega $ by 
\begin{equation}
\mathrm{m}_{n}(\psi )=\int \psi (x)\mathrm{m}_{n}(\mathrm{d}x)=\frac{\int
\psi (x)e^{ng_{n}(x)}\mu _{f}(\mathrm{d}x)}{\int e^{ng_{n}(x)}\mu _{f}(%
\mathrm{d}x)},  \label{trova}
\end{equation}%
where, for any $\varphi \in C(\Omega )$, we recall that $\varphi _{n}$, $%
n\in \mathbb{N}$, are the so-called the Birkhoff averages defined by (\ref%
{Eq Birk Av}), that is,%
\begin{equation}
\varphi _{n}:=\frac{1}{n}(\varphi +\varphi \circ T+\cdots +\varphi \circ
T^{n-1}),\text{ \ \ }n\in \mathbb{N}.  \label{Eq Birk Avbis}
\end{equation}

The probability measures $\mathrm{m}_{n}$, $n\in \mathbb{N}$, are called
here the \textit{linear mean-field Gibbs probability at time }$n$ for the
pair $\mu _{f}$ and $g$. It is natural to consider the weak$^{\ast }$ limit $%
\mathrm{m}$ of $\mathrm{m}_{n}$, as $n\rightarrow \infty $. We call $\mathrm{%
m}$ the \textit{linear mean-field Gibbs probability} for the pair $\mu _{f}$
and $g$. In fact, they are closely related to the concept of DLR
probabilities as described, for instance, in Sections 4 in \cite{CL1} and 
\cite{CLS}. Linear mean-field Gibbs probability exists as stated in the next
theorem:

\begin{theorem}
\label{trit}For any not necessarily normalized H\"{o}lder functions $%
f,g:\Omega \rightarrow \mathbb{R}$, as $n\rightarrow \infty $, the weak$%
^{\ast }$ limit $\mathrm{m}=\mathrm{m}_{f,g}$ of $\mathrm{m}_{n}$ defined by
(\ref{trova}) exists and equals the eigenprobability $\nu _{f+\,g}$ for the
Ruelle operator $\mathcal{L}_{f+g}$.
\end{theorem}

\noindent See, for instance, Section 4.7 of \cite{LTF} for a proof. Note
that this result refers to the lattice $\mathbb{N}$ and, because of the lack
of $T$-invariance, differs from the corresponding results for the lattice $%
\mathbb{Z}$, as stated, for instance, in Corollary 7.13 of \cite{Rue}, where
one gets stationarity for translation on the lattice $\mathbb{Z}$ for free.

In other words, thanks to Theorem \ref{trit}, in the linear mean-field Gibbs
probability setting, eigenprobabilities for the Ruelle operator appear in a
natural way. We show below that the same phenomenon occurs in the nonlinear
setting.

\subsection{Quadratic mean-field Gibbs probabilities\label{laplace}}

Recall again that $\Omega :=\{1,2,\ldots ,d\}^{\mathbb{N}}$ ($d\in \mathbb{N}
$) and fix again a linear equilibrium probability $\mu _{f}$ for the H\"{o}%
lder potential $f:\Omega \rightarrow \mathbb{R}$, which is normalized, i.e., 
$\mathcal{L}_{f}(1)=1$ (to use, e.g., \eqref{uut}). Recall also that $\mu
_{n}^{f}$ is the probability such that, for any open interval $O\subseteq 
\mathbb{R}$, 
\begin{equation*}
\mu _{n}^{f}(O)=\mu \left( \left\{ z\mid \,f_{n}(z)\in O\right\} \right) ,
\end{equation*}%
$f_{n}$, $n\in \mathbb{N}$, being the so-called the Birkhoff averages
defined by (\ref{Eq Birk Avbis}). See Equation (\ref{binc78}). In (\ref%
{binc781}) this definition is generalized to define for two H\"{o}lder
potentials $f,g:\Omega \rightarrow \mathbb{R}$ the probability measure $\mu
_{n}^{f,g}$ by 
\begin{equation}
\mu _{n}^{f,g}(O)=\mu _{f}\left( \left\{ z\mid \,g_{n}(z)\in O\right\}
\right) .  \label{binc78xx}
\end{equation}%
Recall also that, in this case, the large deviation rate function $I_{f,g}$
(see \eqref{binc78x}) is the Legendre transform of the function $\hat{c}%
_{f,g}$ given in Equation (\ref{feij67}). In other words, 
\begin{equation*}
I_{f,g}(x)=tx-\hat{c}_{f,g}(t)=
\end{equation*}%
\begin{equation}
t\,x-P(f+tg)+P(f)=tx-\log \lambda _{f+tg}+P(f)\geq 0  \label{binc78xxx}
\end{equation}%
for each real parameter $t$ satisfying the self-consistency condition 
\begin{equation*}
\hat{c}_{f,g}^{\prime }(t)=x=\frac{d\,P(f+tg)}{dt}=\mu _{f+tg}(A)=\int
g(x)\mu _{f+tg}(\mathrm{d}x),
\end{equation*}%
where $\mu _{f+tg}$ is the linear equilibrium probability for the potential $%
f+tg$. Alternatively, from \eqref{binc78xz} (see also \eqref{outr}), given $%
\beta >0$ and $t$ satisfying the above self-consistency condition, one has
that%
\begin{equation}
I_{f,\beta g}(x)=P(f)-h(\mu _{f+t\beta g})-\beta \mu _{f+t\beta g}(f)\geq 0,
\label{binc78xzy}
\end{equation}%
where $\mu _{f+\beta tg}$ is the linear equilibrium probability for the
potential $f+\beta tg$.

\begin{remark}
Sometimes in this section we will take $f=-\log d$, in particular $\mathcal{L%
}_{f}^{\ast }(\mu )=\mu $ and $P(f)=0$, where $\mu $ is the probability
measure of maximal entropy (as in \cite{LeWa1}). This will be the case, for
instance, in the explicit examples we will exhibit.
\end{remark}

We are interested in explicit expressions related to Theorem 1.3 (4) of the
paper \cite{LeWa1} on the generalized Curie-Weiss model (or Theorem 3 in 
\cite{LeWa2}). To this end, given $\beta >0$, we study the weak$^{\ast }$
limit of the probability measures $\mathfrak{m}_{n}$, $n\in \mathbb{N}$,
defined for all continuous functions $\psi \in C(\Omega )$ by 
\begin{equation}
\mathfrak{m}_{n}(\psi )=\int \psi (x)\mathfrak{m}_{n}(\mathrm{d}x)=\frac{%
\int \psi (x)e^{\frac{\beta n}{2}g_{n}(x)^{2}}\mu _{f}(\mathrm{d}x)}{\int e^{%
\frac{\beta n}{2}g_{n}(x)^{2}}\mu _{f}(\mathrm{d}x)},  \label{xiu0}
\end{equation}%
named \textit{quadratic mean-field Gibbs probabilities at time }$n$, for the
triple $\beta ,\mu _{f},g$. Compare with Equation (9) in \cite{LeWa1}. In
Definition \ref{loloy}, the weak$^{\ast }$ limit $\mathfrak{m}=\mathfrak{m}%
_{\beta ,f,g}$ is called the \textit{quadratic mean-field Gibbs probability}
for $\beta ,\mu _{f}$ and $g$. Compare with Theorem 1.3 (4) in \cite{LeWa1}.

Here, a particular case of interest is when $d=2$, i.e., $\Omega \equiv
\{-1,1\}^{\mathbb{N}}$, and $g:\{-1,1\}^{\mathbb{N}}\rightarrow \mathbb{R}$
is of the form 
\begin{equation}
g(x)=\sum_{n=1}^{\infty }a_{n}x_{n},  \label{binc2xx}
\end{equation}%
for any $x=(x_{n})_{n\in \mathbb{N}}\in \{-1,1\}^{\mathbb{N}}$, where $%
(a_{n})_{n\in \mathbb{N}}$ is a sequence converging exponentially to zero as 
$n\rightarrow \infty $ (making $g$ H\"{o}lder continuous). We will present a
different expression for the limit 
\begin{equation}
\mathfrak{m}(\psi )=\lim_{n\rightarrow \infty }\mathfrak{m}_{n}(\psi
)=\lim_{n\rightarrow \infty }\frac{\int \psi (x)e^{\frac{\beta n}{2}%
g_{n}(x)^{2}}\mu _{f}(\mathrm{d}x)}{\int e^{\frac{\beta n}{2}%
g_{n}(x)^{2}}\mu _{f}(\mathrm{d}x)},\ \ \ \psi \in C(\Omega ),  \label{xiu}
\end{equation}%
(see (\ref{xiux23})) in Equation \eqref{nanv}. The latter will help us to
get more precise information about the limit \eqref{xiu} via the Laplace
method (also known as the method of stationary phase). In this regard we
follow here the main lines of the proofs presented in Section 2.1 of \cite%
{LeWa1} and Section 4.1 of \cite{LeWa2}, but for a general H\"{o}lder
normalized potential $f$. In fact, in \cite{LeWa1} the authors only consider
the special case $f=-\log d$. Here, we will also address a few other new
issues not explicitly mentioned in \cite{LeWa1}, such as the relation with
the self-consistency condition, the quadratic pressure, and examples of
quadratic mean-field Gibbs phase transitions (cf. Definitions \ref{gog1} and %
\ref{gogrexx}). As mentioned above, however, our explicit examples of phase
transition refer to the special the case $f=-\log 2$ and we will use results
of the last sections to present and discuss them.

Below we will adapt Theorem \ref{trit} to the quadratic (nonlinear) case for
a general $\mu _{f}$, that is, for $f:\{-1,1\}^{\mathbb{N}}\rightarrow 
\mathbb{R}$ being a general H\"{o}lder function. Then we address the
existence of quadratic mean-field Gibbs phase transitions and present
explicit examples of them (see Theorem \ref{otwr}). Notice that we reserved
the (simpler) terminology \textit{quadratic phase transition} for the
non-uniqueness of \textit{quadratic equilibrium states} for a given
potential, a different issue which was already discussed in the last section
(see Remark \ref{yyx1}). See Definitions \ref{gog1} and \ref{gogrexx}.

First, recall that the following identity is referred to as the
Hubbard-Stratonovich transformation: 
\begin{equation}
e^{a^{2}}=\frac{1}{\sqrt{2\pi }}\int_{-\infty }^{\infty }e^{-\frac{y^{2}}{2}+%
\sqrt{2}ay}\mathrm{d}y,  \label{HS}
\end{equation}%
where $a\in \mathbb{R}$ is any real constant. For some fixed $\beta >0$ and
each $n\in \mathbb{N}$ we consider the change of coordinates $t=y/\sqrt{%
\beta n}$ to get 
\begin{equation}
e^{a^{2}}=\sqrt{\frac{\beta n}{2\pi }}\int_{-\infty }^{\infty }e^{-\frac{%
\beta n}{2}t^{2}+a\sqrt{2\beta n}t}\mathrm{d}t.  \label{HS12}
\end{equation}%
In fact, the above expression is used to transform a quadratic (nonlinear)
problem into a linear one. Notice that such an argument was used in \cite%
{LeWa1}, in an essential way (see in particular Section 2.1 of \cite{LeWa1}).

Recall that $\mu _{f}$ is the equilibrium probability $\mu _{f}$ for a H\"{o}%
lder potential $f:\Omega \rightarrow \mathbb{R}$ and $g:\Omega \rightarrow 
\mathbb{R}$ is an arbitrary H\"{o}lder function. Fix the parameter $\beta >0$
and take a continuous function $\psi \in C(\Omega )$. Let%
\begin{eqnarray*}
Z_{n,\beta ,f,g,\psi } &:=&\int e^{\frac{\beta n}{2}g_{n}(x)^{2}}\psi (x)\mu
_{f}(\mathrm{d}x), \\
Z_{n,\beta ,f,g} &:=&\int e^{\frac{\beta n}{2}g_{n}(x)^{2}}\mu _{f}(\mathrm{d}x).
\end{eqnarray*}%
Then, adapting the argument used in Section 2.1 of \cite{LeWa1} we infer
from \eqref{HS12} for $a=g_{n}(x)\sqrt{\beta n}/\sqrt{2}$ and Fubini's
theorem that 
\begin{eqnarray*}
Z_{n,\beta ,f,g,\psi } &=&\sqrt{\frac{\beta n}{2\pi }}\int_{-\infty
}^{\infty }e^{-\frac{\beta n}{2}t^{2}}\int e^{\beta tng_{n}(x)}\psi (x)\mu
_{f}(\mathrm{d}x), \\
Z_{n,\beta ,f,g} &=&\sqrt{\frac{\beta n}{2\pi }}\int_{-\infty }^{\infty }e^{-%
\frac{\beta n}{2}t^{2}}\int e^{\beta tng_{n}(x)}\mu _{f}(\mathrm{d}x)\mathrm{%
d}t.
\end{eqnarray*}%
In this way, we get an alternative expression for the probability measures $%
\mathfrak{m}_{n}$, $n\in \mathbb{N}$, originally defined by (\ref{xiu0}):%
\begin{equation}
\mathfrak{m}_{n}(\psi )=\frac{Z_{n,\beta ,f,g,\psi }}{Z_{n,\beta ,f,g}}=%
\frac{\int_{-\infty }^{\infty }e^{-\frac{\beta n}{2}t^{2}}\int e^{\beta
tng_{n}(x)}\psi (x)\mu _{f}(\mathrm{d}x)\mathrm{d}t}{\int_{-\infty }^{\infty
}e^{-\frac{\beta n}{2}t^{2}}\int e^{\beta tng_{n}(x)}\mu _{f}(\mathrm{d}x)%
\mathrm{d}t},\ \ \ \psi \in C(\Omega ).  \label{HS1}
\end{equation}%
Observe that the right-hand side of \eqref{HS1} does not exactly have the
same form as the right-hand side of Equation (20) in \cite{LeWa1}, but this
is a minor issue (at this point we are closer to Equation (3) of \cite{LeWa1}%
).

In this context, we will look closely at the special case where $\Omega
=\{-1,1\}^{\mathbb{N}}$ and $g:\{-1,1\}^{\mathbb{N}}\rightarrow \mathbb{R}$
is of the form (\ref{binc2xx}). This working example, which was analyzed in
detail in Section \ref{Examples} (see in particular Equation (\ref{binc2x}%
)), will clearly illustrate some of the main issues of our proof.

\begin{theorem}
\label{otwr7}Let $f:\{-1,1\}^{\mathbb{N}}\rightarrow \mathbb{R}$ be a
normalized H\"{o}lder potential, $g$ defined by (\ref{binc2xx}) and $\beta
>0 $. There is a quadratic mean-field Gibbs phase transition in the sense
that 
\begin{equation}
\mathfrak{m}(\psi )=\frac{\mu \left( h_{f+\beta t_{1}g}\right) \nu _{f+\beta
t_{1}g}\left( \psi \right) +\mu \left( h_{f+\beta t_{2}g}\right) \nu
_{f+\beta t_{2}g}\left( \psi \,\right) }{\mu \left( h_{f+\beta
t_{1}g}\right) +\mu \left( h_{f+\beta t_{2}g}\right) },\ \ \ \psi \in
C(\Omega ),  \label{xixu3}
\end{equation}%
where $t_{1},t_{2}$ satisfies the self-consistency condition (\ref{geger})
and (\ref{pirt}), $\mu $ is the measure of maximal entropy, and for $j\in
\{1,2\}$, $h_{f+\beta t_{j}g}$ and $\nu _{f+\beta t_{j}g}$ are,
respectively, the main eigenfunction and the eigenprobability for the Ruelle
operator $\mathcal{L}_{f+\beta t_{j}g}$. More precisely, $\nu _{f+\beta
t_{j}g}$ is given by (\ref{hgfr}), and $h_{f+\beta t_{j}g}$ is obtained in
explicit terms from (\ref{eei}), with $t=1$ and $A=f+\beta t_{j}g$, $j=1,2$.
\end{theorem}

\noindent When $f=-\log 2$, the probabilities $\nu _{f+\beta t_{1}g}$ and $%
\nu _{f+\beta t_{2}g}$, are different from each other, as explained in
Remark \ref{yyx}. However, note that in this case, the corresponding
eigenvalues satisfy $\log \lambda _{f+\beta t_{1}g}=\log \lambda _{f+\beta
t_{2}g}$.

We point out that our analysis of the probability $\mathfrak{m}$
(considering the eigenprobability $\mu _{f}$ for a general H\"{o}lder
potential $f$) is a little bit different from the corresponding one in \cite%
{LeWa1}, where only the case of the maximal entropy probability $\mu $ was
considered. In fact, we get a slightly different expression for $\mathfrak{m}
$ (see below \eqref{nanv}), as compared to Equation (26) of \cite{LeWa1}.
Additionally, we note that, from time to time, we will adapt certain useful
technical results from \cite{LeWa1} and \cite{LeWa2} in our proofs in order
to shorten them.

Fix $\beta >0$. It is known (see \cite{PP}) that, for any $t\in \mathbb{R}$,
any H\"{o}lder functions $f,g:\Omega \rightarrow \mathbb{R}$, all $x\in
\Omega $ and $\psi \in C(\Omega )$, 
\begin{equation}
\lim_{k\rightarrow \infty }\frac{\mathcal{L}_{f+\beta tg}^{k}(\psi )(x)}{%
\lambda _{f+\beta tg}^{k}}=h_{f+\beta tg}(x)\nu _{f+\beta tg}\left( \psi
\right) ,  \label{super134}
\end{equation}%
where $\lambda _{f+\beta tg}$, $h_{f+\beta tg}$ and $\nu _{f+\beta tg}$ are,
respectively, the eigenvalue, the eigenfunction, and the eigenprobability
for the Ruelle operator $\mathcal{L}_{f+\beta tg}$. Moreover, 
\begin{equation*}
\mathcal{L}_{f+\beta tg}^{\ast }(\nu _{f+\beta tg})=\lambda _{f+\beta tg}\nu
_{f+\beta tg}.
\end{equation*}%
We need a uniform error estimation for the limit \eqref{super134}. In fact,
adapting to our setting, the argument used to get (18) in \cite{LeWa1}, we
obtain the following estimate:

\begin{lemma}
\label{tosou}Take $\beta ,R>0$ and two H\"{o}lder functions $f,g:\Omega
\rightarrow \mathbb{R}$, then, for any $t\in \lbrack -R,R]$, $x\in \Omega $
and $\psi \in C(\Omega )$, 
\begin{equation}
\left\vert \frac{\mathcal{L}_{f+\beta tg}^{k}(\psi )(x)}{\lambda _{f+\beta
tg}^{k}}-h_{f+\beta tg}(x)\nu _{f+\beta tg}\left( \psi \right) \right\vert =%
\mathcal{O}\left( e^{-k\epsilon }\right)  \label{naoes}
\end{equation}%
for some strictly positive constant $\epsilon =\epsilon (\beta ,f,g,R)>0$
only depending upon the parameters $\beta ,f,g,R$.
\end{lemma}

\begin{proof}
This is\textbf{\ }a consequence of properties of the spectral gap of the
Ruelle operator (see \cite{PP} and \cite{Henn}).\textbf{\ }Note that the
dependence in $t$ does not appears in \eqref{naoes}.\textbf{\ }The reason is
the following: For each fixed $t$ we have an exponential bound of decay of
the form $e^{-k\,\epsilon (\beta ,f,g,t)}$, but, as the spectral gap is
lower semi-continuous with respect to the parameter $t$, as proved in \cite%
{Henn}, the infimum of $\epsilon (\beta ,f,g,t)$ is attained at some $%
t_{0}\in \lbrack -R,R]$, and the uniformity follows.\textbf{\ }In fact,
later in \eqref{twi00}, \eqref{twi1} and \eqref{tyq}, in an important step
in our proof, we will show that indeed, regarding the large $k$ behavior,
one can consider $t$ only in some suitable finite interval $[-T(\beta
),T(\beta )]$, instead of the whole real line.\textbf{\ }
\end{proof}

Now, starting from \eqref{HS1}, we derive a more convenient expression for
the limit $\mathfrak{m}(\psi )$, as $n\rightarrow \infty $, of the
expectation value $\mathfrak{m}_{n}(\psi )$ in order apply the Laplace
method. At fixed $\beta >0$, using \eqref{super134} and \eqref{naoes}, we
proceed in a similar fashion as in Equation (22) of \cite{LeWa2} to estimate %
\eqref{HS1}: For any $\psi \in C(\Omega )$ and $n\in \mathbb{N}$, 
\begin{eqnarray*}
\mathfrak{m}(\psi ) &=&\lim_{n\rightarrow \infty }\frac{\int_{-\infty
}^{\infty }e^{-\frac{\beta n}{2}t^{2}}\int e^{\beta tng_{n}(x)}\psi (x)\mu
_{f}(\mathrm{d}x)\mathrm{d}t}{\int_{-\infty }^{\infty }e^{-\frac{\beta n}{2}%
t^{2}}\int e^{\beta tng_{n}(x)}\mu _{f}(\mathrm{d}x)\mathrm{d}t} \\
&=&\lim_{n\rightarrow \infty }\frac{\int_{-\infty }^{\infty }e^{-n\frac{%
\beta }{2}t^{2}}\int \mathcal{L}_{f}^{n}(e^{\beta tng_{n}}\psi )(x)\mu _{f}(%
\mathrm{d}x)\mathrm{d}t}{\int_{-\infty }^{\infty }e^{-n\frac{\beta }{2}%
t^{2}}\int \mathcal{L}_{f}^{n}(e^{\beta tng_{n}}1)(x)\mu _{f}(\mathrm{d}x)%
\mathrm{d}t} \\
&=&\lim_{n\rightarrow \infty }\frac{\int_{-\infty }^{\infty }e^{-n\frac{%
\beta }{2}t^{2}+n\log \lambda _{f+\beta tg}}\int \mathcal{L}_{f+\beta
tg}^{n}(\psi )(x)\lambda _{f+\beta tg}^{-n}\mu _{f}(\mathrm{d}x)\mathrm{d}t}{%
\int_{-\infty }^{\infty }e^{-n\frac{\beta }{2}t^{2}+n\log \lambda _{f+\beta
tg}}\int \mathcal{L}_{f+\beta tg}^{n}(1)(x)\lambda _{f+\beta tg}^{-n}\mu
_{f}(\mathrm{d}x)\mathrm{d}t} \\
&=&\lim_{n\rightarrow \infty }\frac{\int_{-\infty }^{\infty }e^{-n\frac{%
\beta }{2}t^{2}+n\log \lambda _{f+\beta tg}}\left[ \nu _{f+\beta tg}\left(
\psi \right) \mu _{f}\left( h_{f+\beta tg}\right) +\mathcal{O}(e^{-n\epsilon
})\right] \mathrm{d}t}{\int_{-\infty }^{\infty }e^{-n\frac{\beta }{2}%
t^{2}+n\log \lambda _{f+\beta tg}}\left[ \nu _{f+\beta tg}\left( 1\right)
\mu _{f}\left( h_{f+\beta tg}\right) +\mathcal{O}(e^{-n\epsilon })\right] 
\mathrm{d}t}.
\end{eqnarray*}%
In the second equality, we used that the equilibrium probability $\mu _{f}$
is the main eigenprobability of the adjoint of the Ruelle operator $\mathcal{%
L}_{f}$ and the corresponding eigenvalue is $1$, for $f$ is a normalized H%
\"{o}lder potential, by assumption, meaning that $\mathcal{L}_{f}^{\ast }\mu
_{f}=\mu _{f}$ (cf. (\ref{feij2})). In the third equality we used that,
directly from the definition (\ref{popoiu1}) of the Ruelle operator and (\ref%
{Eq Birk Avbis}), one has that 
\begin{equation*}
\mathcal{L}_{f}^{n}(e^{\beta tng_{n}}\psi )=\mathcal{L}_{f+\beta
tg}^{n}(\psi )
\end{equation*}%
for any continuous function $\psi \in C(\Omega )$. The fourth equality needs
further explanations: On one hand, by Lemma \ref{tosou}, we can a priori use
the corresponding approximations only for $t$ in compact sets. On the other
hand, by an important observation from Section 2.1 (in particular Lemma 2.2)
of \cite{LeWa1}, considered once more in Section 4.1 of \cite{LeWa2}, we get
the existence of $T(\beta )>0$ such that, at large $n\gg 1$, 
\begin{eqnarray}
&&\int_{-\infty }^{\infty }e^{-n\frac{\beta }{2}t^{2}+n\log \lambda
_{f+\beta tg}}\int \mathcal{L}_{f+\beta tg}^{n}(\psi )(x)\lambda _{f+\beta
tg}^{-n}\mu _{f}(\mathrm{d}x)\mathrm{d}t  \label{twi0} \\
&\sim &\int_{-T(\beta )}^{T(\beta )}e^{-n\frac{\beta }{2}t^{2}+n\log \lambda
_{f+\beta tg}}\int \mathcal{L}_{f+\beta tg}^{n}(\psi )(x)\lambda _{f+\beta
tg}^{-n}\mu _{f}(\mathrm{d}x)\mathrm{d}t  \label{twi00}
\end{eqnarray}%
and, in a similar way, 
\begin{eqnarray}
&&\int_{-\infty }^{\infty }e^{-n\frac{\beta }{2}t^{2}+n\log \lambda
_{f+\beta tg}}\nu _{f+\beta tg}\left( \psi \right) \mu _{f}\left( h_{f+\beta
tg}\right) \mathrm{d}t  \label{twi} \\
&\sim &\int_{-T(\beta )}^{T(\beta )}\,e^{-n\frac{\beta }{2}t^{2}+n\,\log
\lambda _{f+\beta t\,g}}\nu _{f+\beta tg}\left( \psi \right) \mu _{f}\left(
h_{f+\beta tg}\right) \mathrm{d}t.  \label{twi1}
\end{eqnarray}%
That is, the contribution of the integration in $t$ over the set $(-\infty
,-T(\beta ))\cup (T(\beta ),\infty )$ will not interfere in the asymptotic
given by the Laplace method. This property can be obtained by adapting the
argument used in Section 4.1 of \cite{LeWa2}, more specifically, the one
used to get Equation (25) of \cite{LeWa2}. The main issue is that the
contribution of the integration in $t$ in (\ref{twi0}) and \eqref{twi}, over
the set $(-\infty ,-T(\beta ))\cup (T(\beta ),\infty )$, goes to zero, when $%
n\rightarrow \infty $, as 
\begin{equation}
\frac{e^{-n\beta G}}{n}  \label{tyq}
\end{equation}%
for a certain constant $G>0$. It follows that, for any $\psi \in C(\Omega )$%
, \textbf{\ }%
\begin{eqnarray}
\mathfrak{m}(\psi ) &=&\lim_{n\rightarrow \infty }\frac{\int_{-\infty
}^{\infty }e^{-n\frac{\beta }{2}t^{2}+n\log \lambda _{f+\beta tg}}\nu
_{f+\beta tg}\left( \psi \right) \mu _{f}\left( h_{f+\beta tg}\right) 
\mathrm{d}t}{\int_{-\infty }^{\infty }e^{-n\frac{\beta }{2}t^{2}+n\log
\lambda _{f+\beta tg}}\nu _{f+\beta tg}\left( 1\right) \mu _{f}\left(
h_{f+\beta tg}\right) \mathrm{d}t}  \notag \\
&=&\lim_{n\rightarrow \infty }\frac{\int_{-\infty }^{\infty }e^{-n\frac{%
\beta }{2}t^{2}+n\log \lambda _{f+\beta tg}}\nu _{f+\beta tg}\left( \psi
\right) \mu _{f}\left( h_{f+\beta tg}\right) \mathrm{d}t}{\int_{-\infty
}^{\infty }e^{-n\frac{\beta }{2}t^{2}+n\log \lambda _{f+\beta tg}}\mu
_{f}\left( h_{f+\beta tg}\right) \mathrm{d}t}.  \label{nanv}
\end{eqnarray}%
Note that the denominator of \eqref{nanv} does not depend on $\psi $.

For fixed $\beta >0$, a way to handle \eqref{nanv} is to consider that $%
\mathfrak{m}(\psi )$ is the limit as $n\rightarrow \infty $ of the
expectation value of the function 
\begin{equation*}
t\mapsto \nu _{f+\beta tg}\left( \psi \right) =\int \psi \left( x\right) \nu
_{f+\beta tg}\left( \mathrm{d}x\right)
\end{equation*}%
with respect to the following probability densities on $\mathbb{R}$: 
\begin{equation}
\Lambda _{n}(t):=\frac{e^{nv(t)}\mu _{f}\left( h_{f+\beta tg}\right) }{%
\int_{-\infty }^{\infty }e^{nv(s)}\mu _{f}\left( h_{f+\beta sg}\right) 
\mathrm{d}s},\text{ \ \ }n\in \mathbb{N},\ t\in \mathbb{R},  \label{nanv1}
\end{equation}%
with the function $v:\mathbb{R}\rightarrow \mathbb{R}$ defined by%
\begin{equation}
v(t):=-\frac{\beta }{2}t^{2}+\log \lambda _{f+\beta tg}=-\frac{\beta }{2}%
t^{2}+P(f+\beta tg)=P_{\beta ,g}(t),\text{ \ \ }t\in \mathbb{R}.
\label{rewt}
\end{equation}%
Note from (\ref{joq2}) that this function is nothing but the Bogoliubov
approximating pressure $P_{\beta ,g}$. With this formulation, it becomes
clear how the Laplace method can be used to estimate, as $n\rightarrow
\infty $, the integrals 
\begin{equation}
\int_{-\infty }^{\infty }e^{nv(t)}\mu _{f}\left( h_{f+\beta tg}\right) 
\mathrm{d}t\quad \text{and}\quad \int_{-\infty }^{\infty }e^{nv(t)}\nu
_{f+\beta tg}\left( \psi \right) \mu _{f}\left( h_{f+\beta tg}\right) 
\mathrm{d}t  \label{twi5}
\end{equation}%
by analyzing the critical points of the function $v$ given by (\ref{rewt}).
Nevertheless, we need to control integrals over the whole real line $\mathbb{%
R}$.

Using the same arguments given above to restrict the integrals (\ref{twi0})
and \eqref{twi} on the compact set $[-T(\beta ),T(\beta )]$, we can restrict
without loss of generality the integrals in (\ref{twi5}) on the same
interval $[-T(\beta ),T(\beta )]$ for sufficiently large $n\gg 1$. In other
words, the contribution of the integration in $t$ over the set $(-\infty
,-T(\beta ))\cup (T(\beta ),\infty )$ will not interfere in the asymptotic
given by the Laplace method to be used next. Indeed, note that in \cite%
{LeWa2}, using Equation (19) in \cite{LeWa2} (that follows from Lemma 3.1 in 
\cite{LeWa2}), the authors show this property, and here, using a similar
reasoning, this property follows from Lemma \ref{qte}. When estimating the
asymptotic of the right-hand side of \eqref{nanv}, the error term of the
form (\ref{tyq}) has to be used in the numerator but also in the denominator.

As already mentioned, for the analysis of the asymptotics of \eqref{nanv},
in particular expressions like \eqref{twi5}, the Laplace method requires
analyzing the critical points of the function $v$ given by (\ref{rewt}).
Note that the term $\log \lambda _{f+\beta tg}$ in (\ref{rewt}) does not
grow faster then linearly in $t$ (see Lemma \ref{qte}). We point out that
when $f\neq -\log 2$, the eigenvalues $\log _{f+\beta t_{1}g}$ and $\log
_{f+\beta t_{2}g}$ may be different, where $t_{1},t_{2}$ are the solutions
to the self-consistency equation (\ref{pirt}), that is, 
\begin{equation}
t=\mu _{f+\beta tg}\left( g\right) =\int g\left( x\right) \mu _{f+\beta
tg}\left( \mathrm{d}x\right)  \label{bel2222}
\end{equation}%
for the H\"{o}lder function $g$ of Equation \eqref{binc2xx}.

As already said, we can restrict without loss of generality the integrals in
(\ref{twi5}) on the same interval $[-T(\beta ),T(\beta )]$ for sufficiently
large $n\gg 1$ and, in this interval, it is possible to have more than one
critical point of the function $v$, but only a finite number of them by
analyticity of the function $v$. For each single critical point, we select
an interval $[r,s]$ containing only that critical point. We then apply the
Laplace method to each one of these intervals. The contribution of the
integration on $t$ on the complement of the union of those intervals $[r,s]$
is negligible, by the Laplace method. Later, we have to consider the
different asymptotic contributions associated with each critical point to
get the final estimate. In fact, the global maxima of the function $v$
determines the asymptotics of the entire integral. In Theorem \ref{otwr} we
consider a particular example, which is described in detail in Section \ref%
{mae}, with $f=-\log 2$.

Assume that the interval $[r,s]\subseteq \lbrack -T(\beta ),T(\beta )]$
contains a unique critical point $t_{0}\in \lbrack r,s]$ of the function $v$
(\ref{rewt}). In other words, $t_{0}$ in the unique point of the interval $%
[r,s]$ such that $v^{\prime }(t_{0})=0$. Note that in the general case, such
a point is a solution to a self-consistency equation, as explained in
Sections \ref{qua} and \ref{mae}. Assume, moreover, that $v^{\prime \prime
}(t_{0})<0$, i.e., $t_{0}$ refers to a local maximum of the function $v$. By
Morse's lemma, the local maximum is then isolated. Note that the second
derivative of $v$ equals 
\begin{eqnarray}
v^{\prime \prime }(t_{0}) &=&-\beta +\frac{d^{2}}{dt^{2}}P(f+\beta
tg)|_{t=t_{0}}  \label{the7} \\
&=&-\beta +\text{asymptotic variance of}\,\beta g\,\text{w.r.t.}\,\mu
_{f+\beta t_{0}g}.  \label{nico7x}
\end{eqnarray}%
(see Proposition 4.12 in \cite{PP}). This second derivative can be
explicitly computed in some cases. See Section \ref{Examples}, in particular %
\eqref{dePa} and \eqref{dePax} when $g$ is of the form \eqref{binc2xx}.

It is instructive here to take the example $f=-\log d$ to understand the
connection between this critical point $t_{0}$ and a quadratic equilibrium
probability, as described for instance in Section \ref{qua}. Observe indeed
that 
\begin{equation*}
\mu _{\beta t_{0}g}=\mu _{-\log d+\beta t_{0}g}\text{ \ \ and \ \ }P(-\log
d+\beta t_{0}g)=-\log d+P(\beta t_{0}g).
\end{equation*}%
In fact, for $f=-\log d$, the equation $v^{\prime }(t_{0})=0$ can be
rewritten as%
\begin{equation}
t_{0}=\mu _{-\log d+\beta t_{0}g}\left( g\right) =\mu _{\beta t_{0}g}\left(
g\right) ,  \label{sosu}
\end{equation}%
which is nothing but (\ref{bel2222}) for $t=t_{0}$. (Remember also the
condition \eqref{outr} of Section \ref{qua}.) For the constant function $%
f=-\log d$, it follows from (\ref{rewt}) that 
\begin{eqnarray}
v(t_{0}) &=&-\frac{\beta }{2}\mu _{f+\beta t_{0}g}\left( g\right)
^{2}+P(\beta t_{0}g)-\log d  \notag \\
&=&-\frac{\beta }{2}\mu _{\beta t_{0}g}\left( g\right) ^{2}+h(\mu _{\beta
t_{0}g})+\beta t_{0}\mu _{\beta t_{0}g}\left( g\right) -\log d  \notag \\
&=&\frac{\beta }{2}\mu _{\beta t_{0}g}\left( g\right) ^{2}+h(\mu _{\beta
t_{0}g})-\log d.  \label{nico6}
\end{eqnarray}%
Recall from (\ref{joq2}) and (\ref{rewt}) that $v=P_{\beta ,g}$. Therefore,
if $t_{0}$ is not only a local maximum but a global maximum of the function $%
v$, then we infer from Theorem \ref{lgre} and Equation (\ref{nico6}) that 
\begin{equation}
\frac{\beta }{2}\mu _{\beta t_{0}g}\left( g\right) ^{2}+h(\mu _{\beta
t_{0}g})=\sup_{\rho \in \mathcal{P}(T)}\left\{ \frac{\beta }{2}\rho \left(
g\right) ^{2}+h(\rho )\right\} =\mathfrak{P}_{2,\beta }(g)+\log d.
\label{nico7}
\end{equation}%
In other words, the linear equilibrium probability $\mu _{\beta t_{0}g}$
solves the above variational problem. In the particular example of Theorem %
\ref{otwr}, we assume that $f=-\log 2$ and the alphabet has only two
elements (i.e., $d=2$). See also Section \ref{mae}. In this case, we get two
self-consistent points $t_{1}$ and $t_{2}=-t_{1}$ for the potential %
\eqref{binc2x} and it thus follows from \eqref{nico7} that $%
v(t_{1})=v(t_{2}) $.

This observation shows that there exists a natural link between the critical
parameter $t_{0}$ for the Laplace method (which is associated to the
mean-field Gibbs probability) and the self-consistent parameter that is
associated with the quadratic equilibrium probability.

Now we are in a position to apply the Laplace method for analyzing the
asymptotic limit of the numerator of \eqref{nanv}, as already observed in 
\cite{LeWa1}: Recall that $[r,s]\subseteq \lbrack -T(\beta ),T(\beta )]$ is
assumed to contain a unique critical point $t_{0}\in \lbrack r,s]$ of the
function $v$ (\ref{rewt}). Then, by the Laplace method (see Section 5.1 of 
\cite{AE}), for any continuous function $\xi :[r,s]\rightarrow \mathbb{R}$,
in the limit $n\rightarrow \infty $, 
\begin{equation}
\int_{r}^{s}e^{nv(t)}\xi (t)\mathrm{d}t\sim \sqrt{\frac{2\pi }{n|v^{\prime
\prime }(t_{0})|}}e^{nv(t_{0})}\xi (t_{0})  \label{the1}
\end{equation}%
(cf. Equation (5.1.9) in \cite{AE}). In particular, in the limit $%
n\rightarrow \infty $, 
\begin{equation}
\int_{r}^{s}e^{nv(t)}\mu _{f}\left( h_{f+\beta tg}\right) \nu _{f+\beta
tg}\left( \psi \right) \mathrm{d}t\sim \sqrt{\frac{2\pi }{n|v^{\prime \prime
}(t_{0})|}}e^{nv(t_{0})}\mu _{f}\left( h_{f+\beta t_{0}g}\right) \nu
_{f+\beta t_{0}g}\left( \psi \right) ,  \label{the33}
\end{equation}%
which gives the asymptotics of the numerator of \eqref{nanv} for any
continuous function $\psi \in C(\Omega )$, and for two H\"{o}lder potentials 
$f$ (normalized) and $g$. In the same way as before we also get that%
\begin{equation}
\int_{r}^{s}e^{nv(t)}\mu _{f}\left( h_{f+\beta tg}\right) \mathrm{d}t\sim 
\sqrt{\frac{2\pi }{n|v^{\prime \prime }(t_{0})|}}e^{nv(t_{0})}\mu _{f}\left(
h_{f+\beta t_{0}g}\right)   \label{the2}
\end{equation}%
for the denominator of \eqref{nanv}. It is important to observe at this
point that the eigenfunction $h_{f+\beta t_{0}g}$ is strictly positive.
Then, it is necessary to analyze each term of the right-hand side of %
\eqref{the33} and (\ref{the2}). In particular, the contribution of the
asymptotic variance of the potential $f+\beta t_{0}g$, which is the second
derivative of $v(t)$ at the critical point $t_{0}$ (see (\ref{nico7x})), is
of great importance.

Clearly, the above arguments can be applied to the more general case of a
finite number of critical points $t_{j}\in \lbrack -T(\beta ),T(\beta
)]\subseteq \mathbb{R}$, $j\in \{1,2,\ldots ,q\}$, of the function $v$ and
the leading term will be given by the finite subset of global maximizers of $%
v$. Recall here that, as our particular $g$ is assumed to be H\"{o}lder, the
pressure function $P(f+\beta tg)$ and so the function $v$ are real analytic
in $t$, which in turn implies the existence of only a finite number of
critical points.

For a generic potential $g$ (i.e., not necessarily of the form %
\eqref{binc2xx}) it is natural to expect the existence of a unique maximizer 
$t_{0}$ of $v$. In this case, it produces the maximum asymptotic grow 
\begin{equation}
\sqrt{\frac{2\pi }{n|v^{\prime \prime }(t_{0})|}}e^{nv(t_{0})}\mu _{f}\left(
h_{f+\beta t_{0}g}\right) \nu _{f+\beta t_{0}g}\left( \psi \right) \quad 
\text{and}\quad \sqrt{\frac{2\pi }{n|v^{\prime \prime }(t_{0})|}}%
e^{nv(t_{0})}\mu _{f}\left( h_{f+\beta t_{0}g}\right)  \label{the2q}
\end{equation}%
respectively for the numerator and denominator of \eqref{nanv}, leading from %
\eqref{nanv} to 
\begin{equation}
\mathfrak{m}(\psi )=\lim_{n\rightarrow \infty }\nu _{f+\beta t_{0}g}\left(
\psi \right) ,\qquad \psi \in C(\Omega ).  \notag
\end{equation}%
In other words, a unique maximizer $t_{0}$ of $v$ dominates the contribution
of the other possible critical points and there is no mean-field Gibbs phase
transition (Definition \ref{gogrexx}).

\begin{remark}
\label{klyt}This happens, for instance, for the potential of Section \ref%
{boex}, given by (\ref{go}), and $\beta =0.6$ (see Figure \ref{fign4}).
Indeed, the relevant function $\varphi $ here is the one given by Equation (%
\ref{cory}), which is nothing but $\varphi =v$ (see (\ref{rewt})). Its graph
is plotted in green, and it can be seen that there are two different
critical points for $\varphi =v$, at which $\varphi $ takes different
values. The corresponding self-consistent parameter is equal to $3$.
Moreover, the second derivative of $\varphi =v$, plotted in red, which is
related to the asymptotic variance, is not the same in the different
critical points, even though their values are close to each other. In this
case, the asymptotic is dominated by a single critical point, which is the
point $t_{0}=3$, and, for the limit of the quotient (\ref{xiu}) (or (\ref%
{HS1})), we get 
\begin{equation}
\mathfrak{m}\left( \psi \right) =\nu _{-\log 2+3\beta g}\left( \psi \right)
,\qquad \psi \in C(\Omega ),  \label{xixu}
\end{equation}%
and there is no phase transition in the sense of Definition \ref{gogrexx}.
\end{remark}

However, in the example given by Remark \ref{poiu} for $d=2$, $f=-\log 2$
and $g=A:\{0,1\}^{\mathbb{N}}\rightarrow \mathbb{R}$ defined by (\ref{binc2x}%
) or \eqref{binc2xx}, recall the existence of two self-consistent points $%
t_{1}$ and $t_{2}=-t_{1}$, but this is not always the case when $f\neq -\log
d$. Moreover, because of the symmetry of $g$, one can show that $P(\beta
t_{1}g)=P(-\beta t_{1}g)$ for the case $f=-\log d$ (see, e.g., Remark \ref%
{poiu}) Therefore, in this case, $v(t_{1})=v(t_{2})$ (see (\ref{rewt})) and
the contributions of the terms 
\begin{equation}
e^{nv(t_{1})}\,\,\ \text{and}\,\ \,e^{nv(t_{2})}  \label{the21}
\end{equation}%
are the same in the case $f=-\log d$.

\begin{remark}
\label{enfw} If $t_{1}$ and $t_{2}$\ are the corresponding self-consistent
constants, that is, the stationary or critical points of $v$, then, for a
general (normalized) H\"{o}lder function $f$, $v(t_{1})$ may be different
from $v(t_{2})$, even in the example where we take $g$ given by (\ref{binc2x}%
) or (\ref{binc2xx}). In this situation, only one of the two terms of (\ref%
{the21}) are relevant for the asymptotics, as already explained above.
\end{remark}

Moreover, we are also able to estimate the second derivatives $v^{\prime
\prime }(t_{1})$ and $v^{\prime \prime }(t_{2})$. For example, from %
\eqref{dePax} we get that 
\begin{equation}
v^{\prime \prime }(t_{1})=v^{\prime \prime }(-t_{1})=v^{\prime \prime
}(t_{2}).  \label{the8}
\end{equation}%
This accounts for the term $\sqrt{2\pi /|v^{\prime \prime }(t_{0})|}$ in %
\eqref{the33}. For any $j\in \{1,2\}$ and $f=-\log 2$, $h_{f+\beta t_{j}g}$
is obtained in explicit terms from \eqref{eei} with $t=1$ and $A=f+\beta
t_{j}g$. In particular, $h_{f+\beta t_{1}g}(x)=$ $h_{f-\beta t_{1}g}(x)^{-1}$
and thus, the expectation values 
\begin{equation*}
\mu \left( h_{f+\beta t_{1}g}\right) \qquad \text{and}\qquad \mu \left(
h_{f+\beta t_{2}g}\right) =\mu \left( h_{f-\beta t_{1}g}(x)\right)
\end{equation*}%
may be different. In fact, it is possible to get their exact values by using
again \eqref{eei}. Moreover, by (\ref{hgfr})--(\ref{viu1}), the terms $\nu
_{f+\beta t_{1}g}\left( \psi \right) $ and $\nu _{f+\beta t_{2}g}\left( \psi
\right) $ may also be different from each other, and one can get their exact
values. In this way, we will get that the two probabilities (which are
generally not $T$-invariant) that appear in Theorem \ref{otwr7} are indeed
different from each other, and one can thus give an explicit example of
nonlinear phase transition.

In conclusion, the asymptotics of both the numerator and the denominator of %
\eqref{nanv} can be explicitly written in the example given by Remark \ref%
{poiu}. In this case, according to \eqref{the21} and \eqref{the8}, the two
critical points $t_{1}$ and $t_{2}=-t_{1}$ are global maximizers of the
function $v$ and%
\begin{equation}
\sqrt{\frac{2\pi }{|v^{\prime \prime }(t_{1})|}}e^{nv(t_{1})}=\sqrt{\frac{%
2\pi }{|v^{\prime \prime }(t_{2})|}}e^{nv(t_{2})}.  \label{the21x}
\end{equation}%
Therefore, in this case, by estimating the asymptotics of the quotient %
\eqref{xiu} (or \eqref{HS1}) via the Laplace method, we get from (\ref{the33}%
) and (\ref{the2}) that, for any $\psi \in C(\Omega )$, 
\begin{equation}
\mathfrak{m}\left( \psi \right) =\frac{\mu \left( h_{f+\beta t_{1}g}\right)
\nu _{f+\beta t_{1}g}\left( \psi \right) +\mu \left( h_{f+\beta
t_{2}g}\right) \nu _{f+\beta t_{2}g}\left( \psi \right) }{\mu \left(
h_{f+\beta t_{1}g}\right) +\mu \left( h_{f+\beta t_{2}g}\right) }
\label{xixu4}
\end{equation}%
and there is a (binary) mean-field Gibbs phase transition in the sense of
Definition \ref{gogrexx}. An interesting fact here is that the potentials $%
-\log d+\beta t_{1}g$ and $-\log d+\beta t_{2}g$, with the constants $t_{1}$
and $t_{2}$ being self-consistent, play the main role for both the nonlinear
pressure problem (producing equilibrium measures) and the canonical Gibbs
setting (producing eigenprobabilities), as already mentioned in \cite{LeWa1}.

Again, the above arguments with only two self-consistent constants $t_{1}$
and $t_{2}$ can be generalized to the more general case of $q\in \mathbb{N}$
global maximizers of $v$, leading in this case to a generalization Theorem %
\ref{otwr7} with $\mathfrak{m}$ being a non-trivial convex combination of $q$
different eigenprobabilities. Notice finally that the limit probability $%
\mathfrak{m}$ is not necessarily $T$-invariant.

\section{Quadratic mean-field equilibrium probabilities}

\label{QLVE}

In this section, we will prove Theorem \ref{kde} together with Corollary \ref%
{kde coro}, which refer again to the quadratic case, but we point out that
our arguments can be adapted for more general nonlinear pressures. In fact,
the quadratic function can be easily replaced with a general convex or
concave function, or even with a sum of both types of functions. Notice that
in \cite{BPL}, we consider nonlinear pressures from a purely abstract
perspective and with great generality. Here, our aim is rather to illustrate
important aspects of nonlinear phase transitions using explicit examples, an
aim that the quadratic case fulfils optimally.

Let $\mu $ be the maximum entropy measure and $g:\Omega \rightarrow \mathbb{R%
}$ any fixed H\"{o}lder potential. For some fixed $\beta >0$ and all $n\in 
\mathbb{N}$, define the probability measure $\mathfrak{M}^{(n)}$ on $\Omega $
by (\ref{q1a1}), that is, 
\begin{equation}
\mathfrak{M}^{(n)}(\psi )=\mathfrak{M}_{g,\beta }^{(n)}(\psi ):=\frac{\mu
\left( \psi _{n}e^{\frac{\beta n}{2}g_{n}^{2}}\right) }{\mu \left( e^{\frac{%
\beta n}{2}g_{n}^{2}}\right) }=\frac{\int \psi _{n}\left( x\right) e^{\frac{%
\beta n}{2}g_{n}\left( x\right) ^{2}}\mu (\mathrm{d}x)}{\int e^{\frac{\beta n%
}{2}g_{n}\left( x\right) ^{2}}\mu (\mathrm{d}x)}  \label{q1}
\end{equation}%
for any continuous (real-valued) function $\psi \in C(\Omega )$, where, for
any $\varphi \in C(\Omega )$, we recall again that $\varphi _{n}$, $n\in 
\mathbb{N}$, are the so-called Birkhoff averages defined by Equation (\ref%
{Eq Birk Av}) (or (\ref{Eq Birk Avbis})). Recall also Definition \ref{dedeus}%
: Any probability $\mathfrak{M}^{\infty }=\mathfrak{M}_{g,\beta }^{\infty }$%
, which is the weak$^{\ast }$ limit of a convergent subsequence $\mathfrak{M}%
^{(n_{k})}$, $k\rightarrow \infty $, is called here a \textit{quadratic
mean-field equilibrium probability} for the pair $g,\beta $.

We will show Theorem \ref{kde} together with Corollary \ref{kde coro}, which
refer to the following assertion:

\begin{theorem}
\label{sdsdfsdf}Given a H\"{o}lder potential $g:\Omega \rightarrow \mathbb{R}
$, any quadratic mean-field equilibrium probability is $T$-invariant and
lies in the closed convex hull of the quadratic equilibrium probabilities
for $g$. In particular, if there is a non-ergodic mean-field equilibrium
probability, then the quadratic equilibrium probabilities for $g$ is
non-unique, i.e., a nonlinear phase transition takes place.
\end{theorem}

Recall that a quadratic equilibrium probability for $g$ is a linear
equilibrium probability for a potential of the form $\beta tg$, where $t\in 
\mathbb{R}$ satisfies a self-consistency condition. See Section \ref{Bogo}
for more details, in particular Section \ref{qua} for the particular case of
the convex function $F(x)=\beta x^{2}/2$ (with $\beta >0$) analyzed here.

In order to prove the above theorem, we need some preliminary results. As
before, $\Omega =\{1,2,\ldots ,d\}^{\mathbb{N}}$ for general $d\in \mathbb{N}
$ and the shift operator is denoted by $T:\Omega \rightarrow \Omega $.
Recall once again that $\mu $ denotes the maximal entropy probability, i.e.,
the equilibrium probability for the constant potential $A=-\log d$.

For all $n\in \mathbb{N}$, define the finite-volume (quadratic) pressure%
\begin{equation}
p^{(n)}(\psi ):=\frac{1}{n}\ln \mu \left( e^{n\left( \frac{\beta }{2}%
g_{n}^{2}+\psi _{n}\right) }\right) ,\text{ \ \ }\psi \in C(\Omega ).
\label{q3}
\end{equation}%
In particular, 
\begin{equation}
p^{(n)}(0)=\frac{1}{n}\ln \mu \left( e^{n\left( \frac{\beta }{2}%
g_{n}^{2}+\psi _{n}\right) }\right) .  \label{q32}
\end{equation}%
It defines a continuous convex mapping $\psi \mapsto p^{(n)}(\psi )$ from $%
C(\Omega )$\ to $\mathbb{R}$. In particular, there is at least one
continuous tangent functional to $p^{(n)}:C(\Omega )\rightarrow \mathbb{R}$,
at any $\psi \in C(\Omega )$. Clearly, for all $\psi \in C(\Omega )$ and $%
n\in \mathbb{N}$,%
\begin{equation}
\left. \frac{d}{d\alpha }p^{(n)}(\alpha \psi )\right\vert _{\alpha =0}=%
\mathfrak{M}^{(n)}(\psi ).  \label{q4}
\end{equation}%
In other words, the above defined probability measure $\mathfrak{M}^{(n)}$
is the unique continuous functional that is tangent to $p^{(n)}(\cdot )$ at $%
0$.

Recall meanwhile that the functional $\mathfrak{f}^{+}:\mathcal{P}%
(T)\rightarrow \mathbb{R}$, defined by \eqref{faro}, satisfies in the convex
case 
\begin{equation}
\mathfrak{f}^{+}(\rho )=-\left( \beta \Delta _{g}(\rho )+h(\rho )-\log
d\right) ,\text{ \ \ }\rho \in \mathcal{P}(T),  \label{q5}
\end{equation}%
with $h:\mathcal{P}(T)\rightarrow \mathbb{R}$ being the affine and weak$%
^{\ast }$ upper semi-continuous functional defined via the
(Kolmogorov-Sinai) entropy, while $\Delta _{A}:\mathcal{P}(T)\rightarrow 
\mathbb{R}$ is the affine and weak$^{\ast }$ upper semi-continuous
functional of Equations (\ref{Delta0})--(\ref{Delta1}), i.e.,%
\begin{equation*}
\Delta _{g}(\rho ):=\lim_{n\rightarrow \infty }\frac{1}{2}\int g_{n}\left(
x\right) ^{2}\rho \left( \mathrm{d}x\right) =\inf_{n\in \mathbb{N}}\frac{1}{2%
}\int g_{n}\left( x\right) ^{2}\rho \left( \mathrm{d}x\right) ,\text{ \ \ }%
\rho \in \mathcal{P}(T).
\end{equation*}%
From Theorem \ref{erte}, 
\begin{equation}
\inf \mathfrak{f}^{+}(\mathcal{P}(T))=-\sup_{\rho \in \mathcal{P}(T)}\left\{
\,\frac{\beta }{2}\rho (g)^{2}+h(\rho )-\log d\right\} .  \label{jer}
\end{equation}%
Since $\mathfrak{f}^{+}$ is lower semicontinuous with respect to the weak$%
^{\ast }$ topology, it has minimizers. As proven above, the (not necessarily
unique) $T$-invariant probability measure at which the minimum value of $%
\mathfrak{f}^{+}$ is attained is a linear equilibrium probability for the
potential $\beta tg$, where $t$ satisfies the self-consistency condition as
given in Section \ref{qua}.

By Equations (\ref{impai}) and (\ref{foiq98}), we have that 
\begin{equation}
\lim_{n\rightarrow \infty }p^{(n)}(0)=-\inf_{\rho \in \mathcal{P}(T)}%
\mathfrak{f}^{+}(\rho )=\sup_{\rho \in \mathcal{P}(T)}\left\{ \,\frac{\beta 
}{2}\rho (g)^{2}+h(\rho )-\log d\right\} .  \label{q6}
\end{equation}%
More generally, one proves that, for all $\psi \in C(\Omega )$,%
\begin{equation}
p^{(\infty )}(\psi ):=\lim_{n\rightarrow \infty }p^{(n)}(\psi )=-\inf_{\rho
\in \mathcal{P}(T)}(\mathfrak{f}^{+}(\rho )-\rho (\psi )).  \label{q8}
\end{equation}%
Hence, $p^{(\infty )}$ defines a continuous convex mapping $\psi \mapsto
p^{(\infty )}(\psi )$ from $C(\Omega )$\ to $\mathbb{R}$. In fact, the last
equation says that $p^{(\infty )}$ and $\mathfrak{f}^{+}$ are related to
each other by the Legendre-Fenchel transform.

We give now a preliminary assertion on weak$^{\ast }$ accumulation point of
probability measures $\mathfrak{M}^{(n)}$, $n\in \mathbb{N}$.

\begin{lemma}
Any weak$^{\ast }$ accumulation point $\mathfrak{M}^{(\infty )}$ of the
sequence $\mathfrak{M}^{(n)}$, $n\in \mathbb{N}$, of probability measures is
necessarily an element of $\mathcal{P}(T)$, i.e., an invariant probability
measure.
\end{lemma}

\begin{proof}
As $\Omega $ is a separable metric space, $\mathcal{P}(T)$ is a metrizable
weak$^{\ast }$ compact convex space and there is a subsequence $\mathfrak{M}%
^{(n_{k})}$, $k\in \mathbb{N}$, converging in the weak$^{\ast }$ topology to 
$\mathfrak{M}^{(\infty )}$. For any $n\in \mathbb{N}$, define the function%
\begin{equation*}
\gamma _{n}(x):=\frac{\exp \left( \frac{\beta n}{2}g_{n}(x)^{2}\right) }{\mu
\left( \exp \left( \frac{\beta n}{2}g_{n}(x)^{2}\right) \right) },\qquad
x\in \mathbb{R},
\end{equation*}
and consider the linear functional $l$ on $C(\Omega )$ defined by the
equilibrium probability, that is, here, 
\begin{equation*}
l(\psi ):=\lim_{k\rightarrow \infty }\mu \left( \psi _{n_{k}}\gamma
_{n_{k}}\right) =\lim_{k\rightarrow \infty }\int \psi _{n_{k}}\left(
x\right) \gamma _{n_{k}}\left( x\right) \mu \left( \mathrm{d}x\right) .
\end{equation*}%
Note from (\ref{Eq Birk Av}) that, for any $\psi \in C(\Omega )$, 
\begin{equation*}
\mu \left( \psi _{n_{k}}\gamma _{n_{k}}\right) -\mu \left( \left( \psi \circ
T\right) _{n_{k}}\gamma _{n_{k}}\right) =\int \left( \frac{\psi \left(
x\right) -\psi \circ T^{n_{k}}}{n_{k}}\right) \gamma _{n_{k}}\left( x\right)
\mu \left( \mathrm{d}x\right) .
\end{equation*}%
Since $\Omega $ is a compact metric space and any continuous function $\psi
\in C(\Omega )$ is uniformly bounded, it follows that, for any $\psi \in
C(\Omega )$, 
\begin{equation*}
l(\psi )=l(\psi \circ T).
\end{equation*}
Hence, $\mathfrak{M}^{(\infty )}$ is a $T$-invariant probability.
\end{proof}

\medskip \noindent We are now in a position to prove Theorem \ref{sdsdfsdf}%
:\medskip

\begin{proof}
Take any weak$^{\ast }$ accumulation point $\mathfrak{M}^{(\infty )}$ of the
sequence $\mathfrak{M}^{(n)}$, $n\in \mathbb{N}$. By the previous lemma, $%
\mathfrak{M}^{(\infty )}\in \mathcal{P}(T)$ and we will prove that $%
\mathfrak{M}^{(\infty )}$ is a minimizer of $\mathfrak{f}^{+}$ on $\mathcal{P%
}(T)$. The claim in Theorem \ref{sdsdfsdf} regarding the fact that a
minimizer of $\mathfrak{f}^{+}$ is necessarily in the closed convex hull of
the quadratic equilibrium probabilities essentially follows from the results
of Section \ref{Bog}, more precisely from Theorem \ref{erte}. In fact, using
this theorem, one proves that $\mathfrak{f}^{+}$ is the $\Gamma $%
-regularization of the function $\mathfrak{g}^{+}$ defined by (\ref{farobis}%
), that is, 
\begin{equation*}
\mathfrak{g}^{+}\mathfrak{(\rho )=-}\,\frac{\beta }{2}\rho (g)^{2}-h(\rho
)+\log d,\text{ \ \ }\rho \in \mathcal{P}(T).
\end{equation*}%
See, e.g., \cite{BP3} for the precise definition of the $\Gamma $%
-regularization of functions. By Theorem 1.4 of \cite{BP3}, this property
implies that any minimizer of $\mathfrak{f}^{+}$ on $\mathcal{P}(T)$ belongs
to the convex hull of the set of minimizers of $\mathfrak{g}$, which are
nothing but quadratic equilibrium probabilities. Note additionally that, if
there is a non-ergodic minimizer of $\mathfrak{f}^{+}$, then the quadratic
equilibrium probabilities for $g$ are non-unique, i.e., a (nonlinear) phase
transition takes place. This results from the fact that the set of
minimizers of $\mathfrak{f}^{+}$ is a (non-empty) face of $\mathcal{P}(T)$\
for $\mathfrak{f}^{+}$ is an affine weak$^{\ast }$ lower semicontinous
functional. See \cite{BPL} for much more details on nonlinear pressures and
their equilibrium probabilities.

We now prove that $\mathfrak{M}^{(\infty )}$ is a minimizer of $\mathfrak{f}%
^{+}$: By well-known properties of the Legendre-Fenchel transform, as $%
\mathfrak{f}^{+}$ is convex (it is even affine) and lower semicontinuous, to
prove that $\mathfrak{M}^{(\infty )}$ minimizes $\mathfrak{f}^{+}$\ it
suffices to show that $\mathfrak{M}^{(\infty )}$ is tangent to $p^{(\infty
)}:C(\Omega )\rightarrow \mathbb{R}$ at $0$, i.e., for all $\psi \in
C(\Omega )$, 
\begin{equation}
p^{(\infty )}(\psi )-p^{(\infty )}(0)\geq \mathfrak{M}^{(\infty )}(\psi ).
\label{q9}
\end{equation}%
This fact follows, for instance, from Theorem 10.47 in \cite{BP2} (a
classical result on convex analysis about tangent functionals as
minimizers). Now, we note that for all $k\in \mathbb{N}$, $\mathfrak{M}%
^{(n_{k})}$ is tangent to $p^{(n_{k})}:C(\Omega )\rightarrow \mathbb{R}$ at $%
0$, i.e., for all $\psi \in C(\Omega )$, 
\begin{equation}
p^{(n_{k})}(\psi )-p^{(n_{k})}(0)\geq \mathfrak{M}^{(n_{k})}(\psi ).
\label{q10}
\end{equation}%
Thus, taking the limit $k\rightarrow \infty $ we arrive at Inequality (\ref%
{q9}) for all $\psi \in C(\Omega )$.
\end{proof}

\section{The tilting LDP property}

\label{Til}

The aim of the present section is to highlight a relation between the
Bogoliubov variational principle discussed above and the tilting principle
of large deviation theory. For simplicity, here we set $f=-\log d$, that is, 
$\mu _{f}=\mu $. Given a continuous potential $A:\Omega \rightarrow \mathbb{R%
}$, remember that, for any $n\in \mathbb{N}$, $\mu _{n}$ denotes the
probability (\ref{binc78}) on $\mathbb{R}$, i.e., for any open interval $%
O\subseteq \mathbb{R}$,%
\begin{equation}
\mu _{n}(O)=\mu _{n}^{A}(O)=\mu \left( \left\{ z\mid A_{n}(z)\in O\right\}
\right) .  \label{binc78xy}
\end{equation}%
where ${A}_{n}$, $n\in \mathbb{N}$, are the Birkhoff averages of the
potential $A$, defined by Equation (\ref{Eq Birk Av}).

Further, given a continuous potential $A:\Omega \rightarrow \mathbb{R}$,
along with a continuous function $F:\mathbb{R}\rightarrow \mathbb{R}$, let $%
\mathrm{m}_{n}^{F,A}$, $n\in \mathbb{N}$, denote the family of probabilities
on $\mathbb{R}$ such that, for any open interval $O\subseteq \mathbb{R}$,%
\begin{equation}
\mathrm{m}_{n}^{F,A}(O)=\frac{\int_{O}e^{nF(x)}\mu _{n}^{A}(\mathrm{d}x)}{%
Z_{n}^{F,A}}=\frac{\int_{O}e^{nF({A}_{n}(x))}\mu (\mathrm{d}x)}{Z_{n}^{F,A}}
\label{zeno21}
\end{equation}%
where 
\begin{equation*}
Z_{n}^{F,A}=\mu _{n}^{A}\left( e^{nF}\right) =\int e^{nF(x)}\mu _{n}^{A}(%
\mathrm{d}x)=\int e^{nF({A}_{n}(x))}\mu (\mathrm{d}x).
\end{equation*}%
(Notice that $\mathrm{m}_{n}^{F,A}$ is a measure on the real line and not in
the symbolic space $\Omega $ as for example the probability $\mathfrak{M}%
^{(n)}$ defined by (\ref{q1}).)

One important question in large deviation theory is the existence of the
limit 
\begin{equation}
\lim_{n\rightarrow \infty }\mathrm{m}_{n}^{F,A}(B)=:\theta (B)=\theta
^{F,A}(B),  \label{zero312}
\end{equation}%
where $B\subseteq \mathbb{R}$ is an arbitrary interval, as well as the
corresponding convergence rate. From \eqref{zeno21}, for any continuous
function $\varphi :\mathbb{R}\rightarrow \mathbb{R}$, one has 
\begin{multline*}
\mathrm{m}_{n}^{F,A}(\varphi )=\int \varphi (x)\mathrm{m}_{n}^{F,A}(\mathrm{d%
}x)=\int \varphi (x)e^{nF(x)}\mu _{n}^{A}(\mathrm{d}x) \\
=\int \varphi \left( {A}_{n}(x)\right) e^{nF\left( {A}_{n}(x)\right) }\mu (%
\mathrm{d}x).
\end{multline*}%
Then, the tilting principle says (see, for instance, Theorem 1.2 in \cite%
{Reza} or Lemma 3 in \cite{Kos}) that, if the function $F$ is not only
continuous but also bounded, then such a family of probabilities $\mathrm{m}%
_{n}^{F,A}$, $n\in \mathbb{N}$, satisfies a Large Deviation Principle (LDP)
with rate function 
\begin{equation}
\mathcal{I}^{F,A}(x)=I_{A}(x)-F(x)-\inf_{y\in \mathbb{R}}\{I_{A}(y)-F(y)\},
\label{zero21}
\end{equation}%
for any $x\in \mathbb{R}$, where $I_{A}=I$ is the rate function (\ref{feij68}%
) for the family $\mu _{n}^{A}$, $n\in \mathbb{N}$ (see (\ref{aret23x})).
Note that $\inf \mathcal{I}^{F,A}=0$.

More precisely, the LDP refers here to the following properties:

\begin{itemize}
\item \emph{LD upper bound}. For any closed interval $C\subseteq \mathbb{R}$%
, 
\begin{equation}
\limsup_{n\rightarrow \infty }{\frac{1}{n}}\log \mathrm{m}_{n}^{F,A}(C)\leq
-\inf_{x\in C}\{\mathcal{I}^{F,A}(x)\}.  \label{areta231}
\end{equation}

\item \emph{LD lower bound}. For any open interval $O\subseteq \mathbb{R}$,%
\begin{equation}
\liminf_{n\rightarrow \infty }{\frac{1}{n}}\log \mathrm{m}_{n}^{F,A}(O)\geq
-\inf_{x\in O}\{\mathcal{I}^{F,A}(x)\}.  \label{areto232}
\end{equation}
\end{itemize}

\noindent From Theorem 1 in \cite{Kos} (i.e., Varadhan(-Bryc) lemma) we
additionally get that, for every continuous and bounded function $U:\mathbb{R%
}\rightarrow \mathbb{R}$,%
\begin{equation}
\lim_{n\rightarrow \infty }{\frac{1}{n}}\log \int e^{nU\left( x\right) }%
\mathrm{m}_{n}^{F,A}(\mathrm{d}x)=\sup_{x\in \mathbb{R}}\{U(x)-\mathcal{I}%
^{F,A}(x)\}.  \label{aretu233}
\end{equation}

In other words, in a sense made precise by the above properties, the
probability density $\mathrm{m}_{n}^{F,A}$ at $x\in \mathbb{R}$ tends to
zero at an exponential rate given by $\mathcal{I}^{F,A}(x)$, as $%
n\rightarrow \infty $. In particular, in the limit $n\rightarrow \infty $, $%
\mathrm{m}_{n}^{F,A}$ is concentrated at the points $\bar{x}\in \mathbb{R}$
at which the function $\mathcal{I}^{F,A}$ takes the zero value. Note in
particular that, as $I_{A}$ is a real analytic function, if $F$ is also real
analytic, then the set of such points $\bar{x}$ is finite. Observe further
that if $F$ is convex then $\mathcal{I}^{F,A}$ typically takes the zero
value at more than one point, as illustrated in an explicit example below.

Clearly, the rate function $\mathcal{I}^{F,A}$ vanishes at $\bar{x}$ iff $%
\bar{x}$ maximizes the quantity $F(x)-I_{A}(x)$ and, for an arbitrary $x\in 
\mathbb{R}$, the probability density $\mathrm{m}_{n}^{F,A}$ at $x\in \mathbb{%
R}$ tends to zero at an exponential rate 
\begin{equation*}
\mathcal{I}^{F,A}(x)=-\left( F(x)-I_{A}(x)-\sup_{y\in \mathbb{R}%
}\{F(y)-I_{A}(y)\}\right) .
\end{equation*}%
Thus, in order to control this rate one has to determine the supremum $%
\sup_{y\in \mathbb{R}}\{F(y)-I_{A}(y)\}$. But, if $F$ is a convex function,
as shown in the beginning of Section \ref{The convex case}, 
\begin{equation*}
\sup_{y\in \mathbb{R}}\{F(y)-I_{A}(y)\}=\sup_{s\in \mathbb{R}}\left\{
-F^{\ast }(s)+P(sA)-\log d\right\} ,
\end{equation*}%
where $F^{\ast }$ is the Legendre transform of $F$ and $P(sA)$\ is the
pressure for the potential $sA$. In this way, we conclude that, for a convex 
$F$, the probability density $\mathrm{m}_{n}^{F,A}$ at $x\in \mathbb{R}$
tends to zero at an exponential rate 
\begin{equation*}
\mathcal{I}^{F,A}(x)=-\left( F(x)-I_{A}(x)-\sup_{s\in \mathbb{R}}\left\{
-F^{\ast }(s)+P(sA)-\log d\right\} \right) .
\end{equation*}

Similarly, from the results of Section \ref{conc} (see (\ref{erer})), if $F$
is a concave function then the probability density $\mathrm{m}_{n}^{F,A}$ at 
$x\in \mathbb{R}$ tends to zero at exponential rate 
\begin{equation*}
\mathcal{I}^{F,A}(x)=-\left( F(x)-I_{A}(x)-\inf_{s\in \mathbb{R}}\left\{
G^{\ast }(-s)+P(sA)-\log d\right\} \right) ,
\end{equation*}%
with $G=-F$. This remark gives a new interesting view on Bogoliubov's
variational principle.

As an example, we plot in Figure \ref{fig21} the exact the rate function $%
\mathcal{I}^{F,A}$ in a particular case. We use the expression given by %
\eqref{xili} for $I_{A}$. Figure \ref{fig21} should be compared with Figure
1 in \cite{Kos} and Figure 2.3 in \cite{Vele}. Both refer to the classical
Curie-Weiss model (a simple potential with no dynamics attached), and
corresponds to a convex $F$.

\begin{figure}[h!]
\centering
\includegraphics[scale=0.4,angle=0]{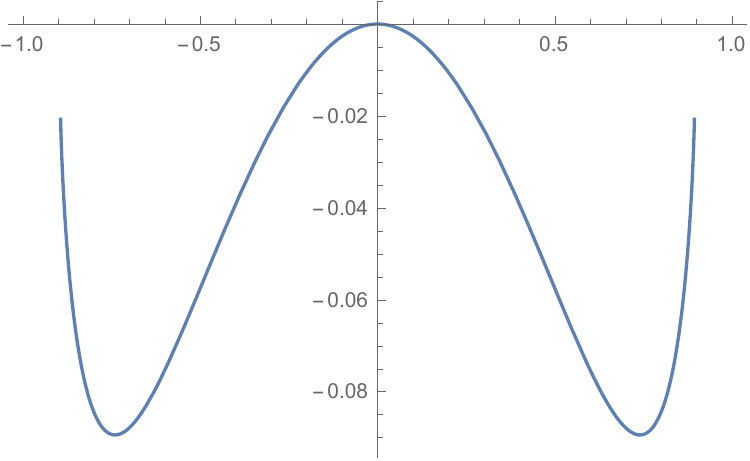}
\caption{Given the potential $\protect\beta\, A(x)=\protect\beta\,\frac{J}{2}
\sum_{n=1}^\infty 2^{-n} x_n$, $F(x)=\frac{\protect\beta J}{2} x^2$, $J=2$,
and the maximal entropy measure $\protect\mu$, we show above the graph of
the function $y\to\mathcal{I}^{A, F}= I_{A}(y) - \, u_{A,J,\protect\beta,0}
\,\,\,y^2= I_{A}(y) - \,\frac{\protect\beta \, 2}{2} \,\,\,y^2$, when $%
\protect\beta = \frac{ 4^{1/3} + 0.2}{2}$.}
\label{fig21}
\end{figure}

Finally, notice that the recent article \cite{Hirsh} also uses large
deviation properties to analyze canonical Gibbs probabilities, but in a
different context.

\bigskip

\noindent \textit{Acknowledgments:} This work is partially  supported by CNPq
(303682/2025-6), CNPq(310053/2020-0) and the Basque Government through the grant IT1615-22 and
the BERC 2022-2025 program as well as the following grants:

\noindent Grant PID2024-156184NB-I00 funded by
MICIU/AEI/10.13039/501100011033 and cofunded by the European Union.

Jean-Bernard Bru  - Basque Center for Applied Mathematics - Spain - jb.bru@ikerbasque.org

\medskip

Walter de Siqueira Pedra - ICMC - USP -  Brazil - wpedra@icmc.usp.br
\medskip

Artur O. Lopes - IME - UFRGS - Brazil - arturoscar.lopes@gmail.com


\begin{thebibliography}{99}
\bibitem{BLL} A. Baraviera, R. Leplaideur, and A. O. Lopes, Ergodic
optimization, zero temperature limits and the max-plus algebra. IMPA
Mathematical Publications. Rio de Janeiro, XXIX Coloquio Brasileiro de
Matematica (2013).

\bibitem{Barre} L. Barreira and C. Holanda, Higher-dimensional nonlinear
thermodynamical formalism. J. Stat. Phys. Volume 187, article number 18
(2022).

\bibitem{Barre1} L. Barreira and C. Holanda, Nonlinear thermodynamic
formalism for flows. Dyn. Syst.,Volume 37, No. 4, 603-629 (2022).

\bibitem{AE} N. Bleistein and R. A. Handelsman, Asymptotic Expansions of
Integrals. Dover (2010).

\bibitem{Bogoliubov1} N.N. Bogoliubov, On the theory of superfluidity. J.
Phys. (USSR) Volume 11, 23-32 (1947).

\bibitem{Bogjunior} N.N. Bogoliubov Jr., On model dynamical systems in
statistical mechanics. Physica, Volume 32, No. 5, 933-944 (1966).

\bibitem{approx-hamil-method0} N.N. Bogoliubov Jr., J.G. Brankov, V.A.
Zagrebnov, A.M. Kurbatov and N.S. Tonchev, Metod approksimiruyushchego
gamil'toniana v statisticheskoi fizike\textit{\footnote{%
The Approximating Hamiltonian Method in\ Statistical Physics.}.} Sofia:
Izdat. Bulgar. Akad. Nauk\footnote{%
Publ. House Bulg. Acad. Sci.} (1981).

\bibitem{approx-hamil-method} N.N. Bogoliubov Jr., J.G. Brankov, V.A.
Zagrebnov, A.M. Kurbatov and N.S. Tonchev, Some classes of exactly soluble
models of problems in Quantum Statistical Mechanics: the method of the
approximating Hamiltonian. Russ. Math. Surv., Volume 39,\textbf{\ }1-50
(1984).

\bibitem{AHM-non-poly1} J.G. Brankov, N.S. Tonchev and V.A. Zagrebnov, A
nonpolynomial generalization of exactly soluble models in statistical
mechanics. Ann. Phys. (N. Y.), Volume 107, No. 1-2, 82-94 (1977).

\bibitem{AHM-non-poly2} J.G. Brankov, N.S. Tonchev and V.A. Zagrebnov, On a
class of exactly soluble statistical mechanical models with nonpolynomial
interactions. J. Stat. Phys., Volume 20, No. (3), 317-330 (1979).

\bibitem{approx-hamil-method2} J.G. Brankov, D.M. Danchev and N.S. Tonchev,
Theory of Critical Phenomena in Finite--size Systems: Scaling and Quantum
Effects. Singapore-New Jersey-London-Hong Kong, Word Scientific (2000).

\bibitem{BP1} J-B. Bru and W. de Siqueira Pedra, $C^{\ast }$-Algebras and
Mathematical Foundations of Quantum Statistical Mechanics. Springer Verlag
(2023)

\bibitem{BP2} J-B. Bru and W. de Siqueira Pedra, Non-cooperative Equilibria
of Fermi Systems With Long Range Interactions. Memoirs of the AMS, Volume
224, No. 1052 (2013).

\bibitem{BP3} J-B. Bru and W. de Siqueira Pedra, Remarks on the $\Gamma $%
-regularization of Non-convex and Non-semi-continuous Functions on
Topological Vector Spaces. Journal of Convex Analysis, Volume 19, No. 2,
467-483 (2012)

\bibitem{BPL} J-B. Bru, W. de Siqueira Pedra and A. O. Lopes, Nonlinear
Thermodynamic Formalism: Mean-field probabilities and the Approximating
Hamiltonian Method. In preparation (2025).

\bibitem{LDP-BZ} J-B. Bru and V.A. Zagrebnov, Large Deviations in the
Superstable Weakly Imperfect Bose Gas, J. Stat. Phys. Volume 133, No. 2,
379-400 (2008).

\bibitem{BKL} J. Buzzy, B. Kloeckner and R. Leplaideur, Nonlinear
thermodynamical formalism Annales Henri Lebesgue. Volume 6, 1429-1477 (2023).

\bibitem{CL1} L. Cioletti and A. O. Lopes, Interactions, Specifications, DLR
probabilities and the Ruelle Operator in the One-Dimensional Lattice.
Discrete and Cont. Dyn. Syst. - Series A, Volume 37, No. 12, 6139-6152
(2017).

\bibitem{CLS} L. Cioletti, A. O. Lopes and M. Stadlbauert, Ruelle Operator
for Continuous Potentials and DLR-Gibbs Measures. Disc and Cont. Dyn. Syst.
A, Volume 40, No. 8, 4625-4652 (2020).

\bibitem{CDLS} L. Cioletti, M. Denker, A. O. Lopes and M. Stadlbauer,
Spectral Properties of the Ruelle Operator for Product Type Potentials on
Shift Spaces. Journal of the London Mathematical Society, Volume 95, Issue
2, 684-704 (2017).

\bibitem{Comman1} H. Comman, Strengthened large deviations for rational maps
and full shifts, with unified proof. Nonlinearity, Volume 22, No. 6, 1413
(2009)

\bibitem{Comman} H. Comman, Entropy approximation versus uniqueness of
equilibrium for a dense affine space of continuous functions. Stochastics
and Dynamics Volume 16, No. 6, 1650020 (12 pages) (2016).

\bibitem{ComLe} H. Comman and J. Rivera-Letelier, Large deviation principles
for non-uniformly hyperbolic rational maps. Ergodic Theory and Dynamical
Systems, Volume 31(2), 321-349 (2011).


\bibitem{DZ} A. Dembo and O. Zeitouni, Large Deviations Techniques and
Applications. Springer Verlag (1998).

\bibitem{DS89} J.-D. Deuschel and D. W. Stroock, Large Deviations. American
Mathematical Soc., Providence, Rohde Island (1989).

\bibitem{Ding} B. Ding and T. Wang, Some variational principles for
nonlinear topological pressure. Dyn. Syst., Volume 40, No. 1, 35-55 (2025).

\bibitem{Ellis} R. Ellis, Entropy, Large Deviations, and Statistical
Mechanics. Springer Verlag (2005).

\bibitem{FaJi} A. Fan and Y. Jiang, Spectral theory of transfer operators.
Lectures from the Morningside Center of Mathematics, on line

\bibitem{Vele} S. Friedli and Y. Velenik, Statistical Mechanics of Lattice
Systems. Cambridge Press (2017).

\bibitem{G1} E. Garibaldi, Ergodic Optimization in the expanding case.
Springer Verlag (2017).

\bibitem{GKLM} P. Giulietti, B. Kloeckner, A. O. Lopes and D. Marcon, The
calculus of thermodynamical formalism. Journal of the EMS, Volume 20, Issue
10, 2357-2412 (2018).

\bibitem{Henn} H. Hennion and L. Herve, Limit theorems for Markov chains and
stochastic properties of dynamical systems by quasi-compactness. Volume 1766
of Lecture Notes in Mathematics, Springer-Verlag, (2001)

\bibitem{Hirsh} C. Hirsch and M. Petrakova, Large-Deviation Analysis for
Canonical Gibbs Measures. Journal of Statistical Physics, Volume 192, 71
(2025).

\bibitem{Kif} Y. Kifer, Large Deviations in Dynamical Systems and Stochastic
processes. TAMS, Volume 321, No.2, 505-524 (1990).

\bibitem{Sion} H. Komiya, Elementary Proof For Sion's minimax theorem. Kodai
Math. J. Volume 11(1), 5-7 (1988).

\bibitem{Kos} E. Kosygina, Varadhan's lemma and applications. Curie-Weiss
model, published as Lecture Notes on the 2nd Northwestern Summer School in
Probability (2018).

\bibitem{Ku} T. Kucherenko, Nonlinear thermodynamic formalism through the
lens of rotation theory. Disc. and Cont. Dyn. Systems, Volume 44, Issue 12,
3760-3773 (2024).

\bibitem{LeWa1} R. Leplaideur and F. Watbled, Generalized Curie-Weiss model
and quadratic pressure in ergodic theory. Bull. Soc. Math. France, Volume
147(2), 197-219 (2019).

\bibitem{LeWa2} R. Leplaideur and F. Watbled, Curie-Weiss Type Models for
General Spin Spaces and Quadratic Pressure in Ergodic Theory. Journal of
Statistical Physics, Volume 181, 263-292 (2020).

\bibitem{L3} A. O. Lopes, Entropy and Large Deviation. NonLinearity, Volume
3, No. 2, 527-546 (1990).

\bibitem{L4} A. O. Lopes, Entropy, Pressure and Large Deviation. Cellular
Automata, Dynamical Systems and Neural Networks, E. Goles e S. Martinez
(eds.), Kluwer, Massachusets, pp. 79-146 (1994).

\bibitem{LMMS} A. O. Lopes, J. K. Mengue, J. Mohr and R. R. Souza, Entropy
and Variational Principle for one-dimensional Lattice Systems with a general
a-priori probability: positive and zero temperature. Erg. Theory and Dyn
Systems, Volume 35(6), 1925-1961 (2015).

\bibitem{LTF} A. O. Lopes, Thermodynamic Formalism, Maximizing Probabilities
and Large Deviations. notes online (UFRGS). See http://mat.ufrgs.br/$\sim $%
alopes/pub3/notesformteherm.pdf

\bibitem{Mohr} J. Mohr, Product type potential on the XY model: selection of
maximizing probability and a large deviation principle. Qual. Theo. of Dyn.
Syst. 21, Article number: 44 (2022)

\bibitem{Mord} B. S. Mordukhovich and N. M. Nam, Convex Analysis and Beyond
- Volume I, Springer Verlag (2022).

\bibitem{Orey} S. Orey, Large Deviations in Ergodic Theory (1986). In: \c{C}%
inlar, E., Chung, K.L., Getoor, R.K. (eds) Seminar on Stochastic Processes,
1984. Progress in Probability and Statistics, vol 9. Birkh\"{a}user Boston.

\bibitem{PP} W. Parry and M. Pollicott, Zeta functions and the periodic
orbit structure of hyperbolic dynamics. Ast\'{e}risque, tome 187-188 (1990).

\bibitem{Reza} F. Rezakhanlou, Lectures on the Large Deviation Principle.
Notes UC Berkeley (2017). See https://math.berkeley.edu/$\sim $%
rezakhan/LD.pdf

\bibitem{Rue} D. Ruelle, Thermodynamic Formalism. Second edition, Cambridge
(2004).

\bibitem{Walters} P. Walters, Introduction to Ergodic Theory. Graduate Texts
in Mathematics, Springer-Verlag New York, Inc. (1982).

\bibitem{Zhu} J. Zhu and R. Zou, A Variationl principle for nonlinear local
pressure. arXiv:2506.17555v1 [math.DS] (2025)
\end{thebibliography}
\end{document}